\documentclass[11pt]{amsart}
\usepackage{amsmath, amssymb, amscd, mathrsfs, url, pinlabel,verbatim}
\usepackage[pagebackref]{hyperref}
\usepackage[margin=1.25in,marginparwidth=1in,centering,letterpaper,dvips]{geometry}
\usepackage{color,dcpic,latexsym,graphicx,epstopdf,comment}
\usepackage[all]{xy}
\usepackage[dvipsnames]{xcolor}
\usepackage{tikz,pgfplots,tikz-cd}
\usepackage{enumitem,upquote,tabularx,textcomp}
\usepackage[all]{hypcap}
\usepackage[color=blue!20!white,textsize=tiny]{todonotes}

\title{Characterizing slopes for $5_2$}
\author[John A. Baldwin]{John A. Baldwin}
\address{Department of Mathematics \\ Boston College}
\email{john.baldwin@bc.edu}

\author[Steven Sivek]{Steven Sivek}
\address{Department of Mathematics \\ Imperial College London}
\email{s.sivek@imperial.ac.uk}

\thanks{JAB was supported by NSF FRG Grant DMS-1952707.}

\makeatletter
\newtheorem*{rep@theorem}{\rep@title}
\newcommand{\newreptheorem}[2]{%
\newenvironment{rep#1}[1]{%
 \def\rep@title{#2 \ref{##1}}%
 \begin{rep@theorem}}%
 {\end{rep@theorem}}}
\makeatother

\newtheorem {theorem}{Theorem}
\newreptheorem{theorem}{Theorem}
\newtheorem {lemma}[theorem]{Lemma}
\newtheorem {proposition}[theorem]{Proposition}
\newtheorem {corollary}[theorem]{Corollary}

\numberwithin{equation}{section}
\numberwithin{theorem}{section}

\theoremstyle{definition}
\newtheorem{definition}[theorem]{Definition}

\newtheorem{remark}[theorem]{Remark}
\newtheorem*{remark*}{Remark}

\newtheorem{example}[theorem]{Example}

\newlist{pcases}{enumerate}{1}
\setlist[pcases]{
  label=\bf{Case~\arabic*:}\protect\thiscase.~,
  ref=\arabic*,
  align=left,
  labelsep=0pt,
  leftmargin=0pt,
  labelwidth=0pt,
  parsep=0pt
}
\newcommand{\case}[1][]{%
  \if\relax\detokenize{#1}\relax
    \def\thiscase{}%
  \else
    \def\thiscase{~#1}%
  \fi
  \item
}

\newcommand{\Z}{\mathbb{Z}}

\newcommand{\Q}{\mathbb{Q}}
\newcommand{\spc}{\operatorname{Spin}^c}
\newcommand{\spinc}{\mathfrak{s}}

\newcommand\bA{\mathbb{A}}
\newcommand\bB{\mathbb{B}}

\newcommand{\cF}{\mathcal{F}}

\newcommand{\cT}{\mathcal{T}}

\newcommand{\bX}{\mathbb{X}}

\newcommand\cf{\mathit{CF}}
\newcommand\cfhat{\widehat{\cf}}

\newcommand\hfk{\mathit{HFK}}
\newcommand\hfkhat{\widehat{\hfk}}

\newcommand\cfkinfty{\mathit{CFK}^\infty}
\newcommand\Hred{H_{\text{red}}}

\DeclareFontFamily{U}{mathx}{\hyphenchar\font45}
\DeclareFontShape{U}{mathx}{m}{n}{
      <5> <6> <7> <8> <9> <10>
      <10.95> <12> <14.4> <17.28> <20.74> <24.88>
      mathx10
      }{}
\DeclareSymbolFont{mathx}{U}{mathx}{m}{n}
\DeclareFontSubstitution{U}{mathx}{m}{n}
\DeclareMathAccent{\widecheck}{0}{mathx}{"71}

\newcommand{\hfhat}{\widehat{\mathit{HF}}}
\newcommand{\cfp}{\mathit{CF}^+}
\newcommand{\hfp}{\mathit{HF}^+}
\newcommand{\hfred}{\mathit{HF}^+_{\mathrm{red}}}

\newcommand{\coker}{\operatorname{coker}}

\newcommand{\mirror}[1]{\overline{#1}}

\newcommand{\nuhat}{\hat{\nu}}
\newcommand{\rhat}{\hat{r}_0}

\newcommand{\dcover}{\Sigma_2}

\newcommand{\Wh}{\operatorname{Wh}}

\makeatletter
\DeclareFontFamily{OMX}{MnSymbolE}{}
\DeclareSymbolFont{MnLargeSymbols}{OMX}{MnSymbolE}{m}{n}
\SetSymbolFont{MnLargeSymbols}{bold}{OMX}{MnSymbolE}{b}{n}
\DeclareFontShape{OMX}{MnSymbolE}{m}{n}{
    <-6>  MnSymbolE5
   <6-7>  MnSymbolE6
   <7-8>  MnSymbolE7
   <8-9>  MnSymbolE8
   <9-10> MnSymbolE9
  <10-12> MnSymbolE10
  <12->   MnSymbolE12
}{}
\DeclareFontShape{OMX}{MnSymbolE}{b}{n}{
    <-6>  MnSymbolE-Bold5
   <6-7>  MnSymbolE-Bold6
   <7-8>  MnSymbolE-Bold7
   <8-9>  MnSymbolE-Bold8
   <9-10> MnSymbolE-Bold9
  <10-12> MnSymbolE-Bold10
  <12->   MnSymbolE-Bold12
}{}

\let\llangle\@undefined
\let\rrangle\@undefined
\DeclareMathDelimiter{\llangle}{\mathopen}%
                     {MnLargeSymbols}{'164}{MnLargeSymbols}{'164}
\DeclareMathDelimiter{\rrangle}{\mathclose}%
                     {MnLargeSymbols}{'171}{MnLargeSymbols}{'171}
\makeatother

\usetikzlibrary{calc,intersections,patterns}
\tikzset{every picture/.style=thick}
\tikzset{link/.style = { white, double = black, line width = 1.75pt, double distance = 1.25pt, looseness=1.75 }}
\tikzset{linkred/.style = { white, double = red, line width = 1.75pt, double distance = 1.25pt, looseness=1.75 }}
\tikzset{thinlink/.style = { white, double = black, line width = 1.25pt, double distance = 0.75pt, looseness=1.75 }}
\tikzset{link2/.style = { white, double = blue, line width = 1.75pt, double distance = 1.25pt, looseness=1.75 }}
\tikzset{crossing/.style = {draw, circle, dotted, minimum size=0.5cm, inner sep=0, outer sep=0}}
\pgfplotsset{compat=1.12}

\begin{document}

\begin{abstract}
We prove that all rational slopes are characterizing for the knot $5_2$, except possibly for positive integers.  Along the way, we classify the Dehn surgeries on knots in $S^3$ that produce the Brieskorn sphere $\Sigma(2,3,11)$, and we study knots on which large integral surgeries are almost L-spaces.
\end{abstract}

\maketitle
\section{Introduction} \label{sec:intro}

Let $K \subset S^3$ be a knot.  Then a rational number $r$ is said to be a \emph{characterizing slope} for $K$ if the result $S^3_r(K)$ of Dehn surgery on $K$ with slope $r$ does not arise as $r$-surgery on any other knot: in other words, if whenever there is an orientation-preserving homeomorphism
\[ S^3_r(K) \cong S^3_r(K'), \]
the knot $K'$ must be isotopic to $K$.

All rational numbers are characterizing slopes for the unknot, as well as for the trefoils and the figure eight knot.  These are theorems of Kronheimer--Mrowka--Ozsv\'ath--Szab\'o \cite{kmos} and of Ozsv\'ath--Szab\'o \cite{osz-characterizing}, respectively, each relying on a theorem (due to Ghiggini \cite{ghiggini} in the latter case) asserting that some form of Floer homology detects the knot in question.  Ni--Zhang and McCoy \cite{ni-zhang-torus-1,mccoy-torus,mccoy-sharp} have proved that many slopes are characterizing for torus knots, especially $T_{2,5}$ \cite{ni-zhang-torus}.  More generally, Lackenby \cite{lackenby-characterizing} has shown that every knot has infinitely many characterizing slopes, and McCoy \cite{mccoy-hyperbolic} has strengthened this in the hyperbolic case.

Our main result, Theorem \ref{thm:main-characterizing}, says that almost all slopes are characterizing for the knot $5_2$, shown in Figure \ref{fig:5_2}. This is strongest result to date for any non-fibered knot and for any hyperbolic knot other than the figure eight:

\begin{theorem} \label{thm:main-characterizing}
Let $r$ be any rational number other than a positive integer.  If for some knot $K \subset S^3$ there is an orientation-preserving homeomorphism
\[ S^3_r(K) \cong S^3_r(5_2), \]
then $K$ is isotopic to $5_2$.  In other words, every $r \in \Q \setminus \Z_{>0}$ is characterizing for $5_2$.
\end{theorem}

It is possible that no positive integer is characterizing for $5_2$ (and hence that Theorem \ref{thm:main-characterizing} is  optimal).  Indeed, Baker--Motegi \cite{baker-motegi} have exhibited hyperbolic knots such as $8_6$ with no integral characterizing slopes, and Abe--Tagami \cite{abe-tagami} proved similar results for many other low-crossing knots.  At the very least, Proposition~\ref{prop:slope-1} says that the positive integer $1$ is not characterizing for $5_2$:

\begin{proposition}
There is an orientation-preserving homeomorphism
\[ S^3_1(5_2) \cong S^3_1(P(-3,3,8)), \]
so $1$ is not a characterizing slope for $5_2$.
\end{proposition}

This fact was originally discovered by Akbulut \cite{akbulut-exotic}, who also showed that the traces of the corresponding surgeries are homeomorphic but not diffeomorphic.

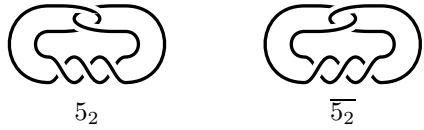
\begin{figure}
\begin{tikzpicture}[scale=0.85,font=\small]
\begin{scope}
\draw[link] (0.2,-0.4) to[out=0,in=180,looseness=0.75] ++(0.4,-0.4) to[out=0,in=0] ++(0,1.2) -- ++(-0.2,0) arc (90:180:0.6 and 0.2) ++(0.4,0) arc (360:270:0.6 and 0.2) -- ++(-0.2,0) to[out=180,in=180] ++(0,-0.4) to[out=0,in=180,looseness=0.75] ++(0.4,-0.4);
\draw[link] (-0.2,-0.4) to[out=180,in=0,looseness=0.75] ++(-0.4,-0.4) to[out=180,in=180] ++(0,1.2) -- ++(0.2,0) arc (90:0:0.6 and 0.2) ++(-0.4,0) arc (180:270:0.6 and 0.2) -- ++(0.2,0) to[out=0,in=0] ++(0,-0.4) to[out=180,in=0,looseness=0.75] ++(-0.4,-0.4);
\begin{scope}
\clip ([shift={(-\pgflinewidth,-0.5*\pgflinewidth)}]-0.5,-0.9) rectangle ([shift={(\pgflinewidth,0.5*\pgflinewidth)}]0.5,-0.3);
\draw[link,looseness=0.75] (-0.6,-0.4) to[out=0,in=180] ++(0.4,-0.4) to[out=0,in=180] ++(0.4,0.4) to[out=0,in=180] ++(0.4,-0.4);
\draw[link,looseness=0.75] (-0.6,-0.8) to[out=0,in=180] ++(0.4,0.4) to[out=0,in=180] ++(0.4,-0.4) to[out=0,in=180] ++(0.4,0.4);
\foreach \i in {-1,1} {
  \begin{scope}
  \clip ([shift={(-\pgflinewidth,-\pgflinewidth)}]0.4*\i,-0.4) ++(-0.1,0) rectangle ++([shift={(\pgflinewidth,1.5*\pgflinewidth)}]0.2,-0.4);
  \draw[link,looseness=0.75] (0.4*\i,-0.4) ++ (-0.2,0) to[out=0,in=180] ++(0.4,-0.4);
  \end{scope}
}
\end{scope}
\begin{scope}
\clip ([shift={(-0.5*\pgflinewidth,-0.5*\pgflinewidth)}]-0.5,0.15) rectangle ([shift={(0.5*\pgflinewidth,1*\pgflinewidth)}]0.5,0.25);
\draw[link] (-0.4,0.4) arc (90:-90:0.6 and 0.2);
\draw[link] (0.4,0.4) arc (90:270:0.6 and 0.2);
\end{scope}
\node at (0,-1.25) {$5_2\vphantom{\mirror{5_2}}$};
\end{scope}
\begin{scope}[xshift=4cm]
\draw[link] (-0.2,-0.4) to[out=180,in=0,looseness=0.75] ++(-0.4,-0.4) to[out=180,in=180] ++(0,1.2) -- ++(0.2,0) arc (90:0:0.6 and 0.2) ++(-0.4,0) arc (180:270:0.6 and 0.2) -- ++(0.2,0) to[out=0,in=0] ++(0,-0.4) to[out=180,in=0,looseness=0.75] ++(-0.4,-0.4);
\draw[link] (0.2,-0.4) to[out=0,in=180,looseness=0.75] ++(0.4,-0.4) to[out=0,in=0] ++(0,1.2) -- ++(-0.2,0) arc (90:180:0.6 and 0.2) ++(0.4,0) arc (360:270:0.6 and 0.2) -- ++(-0.2,0) to[out=180,in=180] ++(0,-0.4) to[out=0,in=180,looseness=0.75] ++(0.4,-0.4);
\begin{scope}
\clip ([shift={(-\pgflinewidth,-0.5*\pgflinewidth)}]-0.5,-0.9) rectangle ([shift={(\pgflinewidth,0.5*\pgflinewidth)}]0.5,-0.3);
\draw[link,looseness=0.75] (-0.6,-0.8) to[out=0,in=180] ++(0.4,0.4) to[out=0,in=180] ++(0.4,-0.4) to[out=0,in=180] ++(0.4,0.4);
\draw[link,looseness=0.75] (-0.6,-0.4) to[out=0,in=180] ++(0.4,-0.4) to[out=0,in=180] ++(0.4,0.4) to[out=0,in=180] ++(0.4,-0.4);
\foreach \i in {-1,1} {
  \begin{scope}
  \clip ([shift={(-\pgflinewidth,-\pgflinewidth)}]0.4*\i,-0.4) ++(-0.1,0) rectangle ++([shift={(\pgflinewidth,1.5*\pgflinewidth)}]0.2,-0.4);
  \draw[link,looseness=0.75] (0.4*\i,-0.8) ++ (-0.2,0) to[out=0,in=180] ++(0.4,0.4);
  \end{scope}
}
\end{scope}
\begin{scope}
\clip ([shift={(-0.5*\pgflinewidth,-0.5*\pgflinewidth)}]-0.5,0.15) rectangle ([shift={(0.5*\pgflinewidth,1*\pgflinewidth)}]0.5,0.25);
\draw[link] (-0.4,0.4) arc (90:-90:0.6 and 0.2);
\draw[link] (0.4,0.4) arc (90:270:0.6 and 0.2);
\end{scope}
\node at (0,-1.25) {$\mirror{5_2}$};
\end{scope}
\end{tikzpicture}
\caption{The knot $5_2$ (left), and its mirror $\mirror{5_2}$ (right).}
\label{fig:5_2}
\end{figure}

\begin{remark}
The orientation-preserving condition is a necessary part of Theorem~\ref{thm:main-characterizing}.  For example, there are homeomorphisms
\begin{align*}
S^3_{1/2}(5_2) &\cong -S^3_{1/2}(6_1), &
S^3_1(5_2) &\cong -S^3_1(\mirror{6_1}).
\end{align*}
This can be deduced from \cite[Proposition~7.2]{bs-concordance}, in which $5_2 = K(2,4)$ and $6_1 = K(-2,4)$.
\end{remark}

As an application, we determine all of the ways in which the Brieskorn sphere $\Sigma(2,3,11)$ can arise from Dehn surgery on a knot in $S^3$:

\begin{theorem} \label{thm:main-2-3-11}
Given a knot $K \subset S^3$ and a rational number  $r$, there exists an orientation-preserving homeomorphism \[S^3_r(K) \cong \Sigma(2,3,11)\] if and only if $(K,r)$ is either $(T_{-2,3},-\frac{1}{2})$ or $(5_2,-1)$.
\end{theorem}

Similar results have been achieved for $\Sigma(2,3,5)$ by Ghiggini \cite[Corollary~1.7]{ghiggini}, and for $\Sigma(2,3,7)$ by Ozsv\'ath--Szab\'o \cite[Corollary~1.3]{osz-characterizing}.

The proof of Theorem~\ref{thm:main-characterizing} relies on our recent classification \cite{bs-nonfibered} of genus-1 knots which are \emph{nearly fibered} from the point of view of knot Floer homology:

\begin{theorem}[{\cite[Theorem 1.2]{bs-nonfibered}}] \label{thm:main-hfk}
Let $K \subset S^3$ be a knot of Seifert genus 1 such that \[\dim_\Q \hfkhat(K,1;\Q) = 2.\]  Then $K$ is one of the knots
\[5_2,  \,\, 15n_{43522},\,\, \Wh^-(T_{2,3},2),\,\,
\Wh^+(T_{2,3},2), \,\,P(-3,3,2n+1) \ (n\in\Z)\]
or their mirrors; the   knot Floer homologies of these knots are given in Table~\ref{fig:hfk-table}.
\end{theorem}

\begin{table}
\[ \arraycolsep=1em
\begin{array}{cccc}
K & \hfkhat(K,1;\Q) & \hfkhat(K,0;\Q) & \hfkhat(K,-1;\Q) \\[0.25em] \hline &&& \\[-0.75em]
5_2 & \Q^2_{(2)} & \Q^3_{(1)} & \Q^2_{(0)} \\[0.5em]
15n_{43522} & \Q^2_{(0)} & \Q^4_{(-1)} \oplus \Q^{\vphantom{0}}_{(0)} & \Q^2_{(-2)} \\[0.5em]
\Wh^-(T_{2,3},2) & \Q^2_{(0)} & \Q^4_{(-1)} \oplus \Q^{\vphantom{0}}_{(0)} & \Q^2_{(-2)} \\[0.5em] \hline&&&\\[-0.75em]
P(-3,3,2n+1) & \Q^2_{(1)} & \Q^5_{(0)} & \Q^2_{(-1)} \\[0.5em]
\Wh^+(T_{2,3},2) & \Q^2_{(-1)} & \Q^4_{(-2)} \oplus \Q^{\vphantom{0}}_{(0)} & \Q^2_{(-3)}
\end{array} \]
\caption{The knot Floer homologies of the knots in Theorem~\ref{thm:main-hfk}, grouped by whether the Alexander polynomial  is $2t-3+2t^{-1}$ or $-2t+5-2t^{-1}$.  The subscripts denote Maslov gradings.}
\label{fig:hfk-table}
\end{table}

Theorem~\ref{thm:main-characterizing}   is then a combination of  Theorems \ref{thm:main-m5_2-nonnegative} and \ref{thm:main-5_2-positive} below.  By way of notation, whenever we discuss an isomorphism between Heegaard Floer homologies of the form \[\hfp(Y;\Q) \cong \hfp(Y';\Q)\] in this paper, we will always mean an isomorphism of $\Q[U]$-modules which respects a decomposition of each side into summands indexed by $\spc$ structures on $Y$ and $Y'$, respectively.

\begin{theorem}[Theorem~\ref{thm:m5_2-nonnegative}] \label{thm:main-m5_2-nonnegative}
Suppose for some knot $K \subset S^3$ and rational number $r\geq 0$ that there is an isomorphism
\[ \hfp(S^3_r(K);\Q) \cong \hfp(S^3_r(\mirror{5_2});\Q) \]
of graded $\Q[U]$-modules.  Then $K$ is isotopic to $\mirror{5_2}$.
\end{theorem}

Theorem~\ref{thm:main-m5_2-nonnegative} immediately implies the case $r \leq 0$ of Theorem~\ref{thm:main-characterizing}, via the relation
\[ S^3_r(K) \cong -S^3_{-r}(\mirror{K}), \] and the relationship between the Heegaard Floer homologies of $Y$ and $-Y$.
For the case $r > 0$, we prove the following:

\begin{theorem} \label{thm:main-5_2-positive}
Suppose for some knot $K \subset S^3$ and rational number $r > 0$ that there is an orientation-preserving homeomorphism
\[ S^3_r(K) \cong S^3_r(5_2), \]
but that $K$ is not isotopic to $5_2$.  Then $r$ is a positive integer, and $g=g(K)$ is at least $2$; if $g$ is even then $r$ divides $g-1$, while if $g$ is odd then $r$ divides $2g-2$.  Moreover, $K$ has Alexander polynomial
\[ \Delta_K(t) = t^g - 2t^{g-1} + t^{g-2} + 1 + t^{2-g} - 2t^{1-g} + t^{-g}, \]
and the knot Floer homology $\hfkhat(K;\Q)$ is completely determined as a bigraded $\Q$-vector space by $r$ and $g$: it is 9-dimensional, and there is a $\Q$ summand in Alexander--Maslov bigrading $(0,0)$ while the rest is supported in bigradings $(a,m) = (a,a+\delta)$, where
\[ \delta = 2-g + \begin{cases} -(g-1)\left(\frac{g-1}{r}-1\right), & r \mid g-1 \\ -\frac{1}{4r}(2g-2-r)^2, & r \nmid g-1. \end{cases} \]
\end{theorem}

Most of the content of Theorem~\ref{thm:main-5_2-positive} is in Theorem~\ref{thm:5_2-positive}, which makes heavy use of the Heegaard Floer mapping cone formula for Dehn surgeries.  However, the latter assumes that $g \geq 2$, and it only concludes that $r$ divides $2g-2$.  We use an obstruction due to Ito \cite{ito-lmo} involving finite type invariants to handle the case $g=1$ in Proposition~\ref{prop:5_2-positive-genus-1}, and to improve the condition $r\mid 2g-2$ to $r \mid g-1$ for even $g$ in Proposition~\ref{prop:quantum-obstruction-g-even}.

\begin{remark}
In fact, the proof of Theorem~\ref{thm:main-5_2-positive} shows that \[\hfp(S^3_r(K);\Q) \not\cong \hfp(S^3_r(5_2);\Q)\] in nearly all cases where it asserts that $S^3_r(K) \not\cong S^3_r(5_2)$.  The exceptions are when $g(K) \geq 2$ is even and $r$ divides $2g(K)-2$ but not $g(K)-1$, and when $g(K)=1$ and $K$ is one of the knots listed in Theorem~\ref{thm:main-hfk} with Alexander polynomial $2t-3+2t^{-1}$.  In the latter case, we require the full strength of Theorem~\ref{thm:main-hfk}, rather than just the claim that $\hfkhat$ detects $5_2$, in order to enumerate the remaining cases and to rule them out one by one in Proposition~\ref{prop:5_2-positive-genus-1}.
\end{remark}

\begin{remark}
If $g(K)=2$ and $S^3_r(K) \cong S^3_r(5_2)$, then Theorem~\ref{thm:main-5_2-positive} says that $r=1$ and $\delta=0$.  This implies that $K$ has the same knot Floer homology as of any of the pretzel knots $P(-3,3,2n)$, where $n\in\Z$.  We conjecture that it must then actually be isotopic to $P(-3,3,2n)$ for some $n$, in which case Remark~\ref{rem:pretzel-1-surgery} will show that it is $P(-3,3,8)$.
\end{remark}

The proofs of Theorems~\ref{thm:main-m5_2-nonnegative} and \ref{thm:main-5_2-positive} rely heavily on formulas which determine the Heegaard Floer homology of Dehn surgeries on a knot $K$ in terms of the canonical $\Z\oplus\Z$ filtration on $\cfkinfty(K)$.  This includes both the ``large surgeries'' formula of \cite{osz-knot}, which applies to surgeries of integral slope $n\geq 2g(K)-1$, and the ``mapping cone'' formula of \cite{osz-rational}, which applies to surgeries of any positive rational slope.  This should come as no surprise to readers familiar with previous works on characterizing slopes such as \cite{osz-characterizing}, although the application of these formulae to the problem considered here is substantially more involved.  We briefly outline their uses below.

First, for Theorem~\ref{thm:main-m5_2-nonnegative}, we are able to avoid the heavy machinery of the mapping cone formula by making use of the fact that $\mirror{5_2}$ is very nearly an L-space knot.

\begin{definition} \label{def:almost-lspace}
We say that a closed 3-manifold $Y$ is an \emph{almost L-space} if it is a rational homology 3-sphere and satisfies \[\dim \hfhat(Y;\Q) = |H_1(Y;\Z)|+2.\]
We say that a nontrivial knot $K \subset S^3$ is an \emph{almost L-space knot} if
\[ \dim \hfhat(S^3_n(K);\Q) = n+2 \]
(that is, if $S^3_n(K)$ is an almost L-space) for some integer $n \geq 2g(K)-1$, in which case one can show that it holds for all $n \geq 2g(K)-1$.
\end{definition}

Then $\mirror{5_2}$ is an almost L-space knot, and there are very few other examples with genus 1.  The following is a combination of Propositions~\ref{prop:almost-fibered} and \ref{prop:almost-implies-sqp}.

\begin{theorem} \label{thm:main-almost}
If $K \subset S^3$ is an almost L-space knot, then  one of the following is true:
\begin{enumerate}
\item $K$ is the left-handed trefoil, figure eight knot, or $\mirror{5_2}$.
\item $g(K) \geq 2$, and $K$ is fibered and strongly quasipositive.
\end{enumerate}
\end{theorem}

With this theorem at hand, we are able to show quickly that if there is an isomorphism \[\hfp(S^3_r(K);\Q) \cong \hfp(S^3_r(\mirror{5_2});\Q)\] for some rational $r \geq 0$, then $K$ must also be an almost L-space knot of genus 1, and then we only have to rule out the left-handed trefoil and the figure eight.  The following is also a straightforward consequence of Theorem~\ref{thm:main-almost}.

\begin{theorem}[Theorem~\ref{thm:1-surgery-3d}] \label{thm:main-1-surgery-3d}
Let $K \subset S^3$ be a knot.  Then $\dim_\Q \hfhat(S^3_1(K);\Q) = 3$ if and only if $K$ is either the left-handed trefoil, figure eight, or $\mirror{5_2}$.
\end{theorem}

Theorem~\ref{thm:main-5_2-positive} requires substantially more effort than Theorem~\ref{thm:main-m5_2-nonnegative}.  The key input is a computation in \S\ref{ssec:5_2-computations} showing that for any $r > 0$ and any $\spc$ structure $\spinc$ on $S^3_r(5_2)$, the Heegaard Floer homology $\hfp(S^3_r(5_2),\spinc;\Q)$ is always isomorphic to something of the form
\[ \cT^+_{(0)} \oplus \Q^{2n}_{(0)} \]
as a relatively graded $\Q[U]$-module.  Here, \[\cT^+ \cong \frac{\Q[U,U^{-1}]}{U\cdot \Q[U]},\] the $U$-action lowers the grading by $2$, and the ``$(0)$'' subscripts indicate that the element $1\in \cT$ and the $\Q^{2n}$ summand both lie in grading $0$.  If $S^3_r(K) \cong S^3_r(5_2)$ for some $r > 0$, then in \S\ref{sec:5_2} we find that this imposes strong restrictions on $\cfkinfty(K)$.  In the case $g(K) \geq 2$, we see in \S\ref{sec:5_2-mapping-cone} that these restrictions often imply that for some $\spinc \in \spc(S^3_r(K))$, either:
\begin{itemize}
\item $\ker(U) \subset \hfp(S^3_r(K),\spinc;\Q)$ cannot lie in a single grading, or
\item $\hfp(S^3_r(K),\spinc;\Q) \cong \cT^+ \oplus \Q$.
\end{itemize}
The first of these applies when $0 < r < 1$, or when $r = p/q \geq 1$ is non-integral and $p \mid 2g(K)-2$, and the second applies when $r = p/q \geq1$ and $p \nmid 2g(K)-2$.  Both of these contradict the computation of $\hfp(S^3_r(5_2);\Q)$, completing the proof in these cases.


\subsection{Notation}
All Floer homologies  in this paper will be taken with coefficients in $\Q$. We will therefore omit the coefficients from the notation going forward.

\subsection{Organization}

In \S\ref{sec:surgery}, we review some facts about knot Floer homology and the large surgery and mapping cone formulas, and then carry out some computations for the knots of Theorem~\ref{thm:main-hfk}.  In \S\ref{sec:dim-hfhat}, we use this to study the dimension of $\hfhat$ of Dehn surgeries, proving Theorem~\ref{thm:main-almost} about almost L-space knots.  We apply this in \S\ref{sec:mirror} to prove Theorem~\ref{thm:main-m5_2-nonnegative}.  In \S\ref{sec:5_2}, we begin to work toward Theorem~\ref{thm:main-5_2-positive}, eliminating all but finitely many $K$ in the case $g(K)=1$ and then obtaining some restrictions in the case $g(K) \geq 2$, and in \S\ref{sec:5_2-mapping-cone} we apply the mapping cone formula together with these restrictions to complete the proof of Theorem~\ref{thm:main-5_2-positive} for $g(K) \geq 2$, modulo the modest improvement of Proposition~\ref{prop:quantum-obstruction-g-even}.  In \S\ref{sec:quantum}, we use finite type invariants to achieve that improvement and to finish off the case $g(K)=1$, completing the proof of Theorem~\ref{thm:main-5_2-positive} and hence of Theorem~\ref{thm:main-characterizing}.

In the last few sections we study some specific examples of surgeries.  In \S\ref{sec:non-characterizing}, we prove Proposition~\ref{prop:slope-1}, asserting that $1$ is not a characterizing slope for $5_2$, and then in \S\ref{sec:2-3-11} we prove Theorem~\ref{thm:main-2-3-11} on the Dehn surgery characterization of $\Sigma(2,3,11)$.

\subsection{Acknowledgements}

We thank Tetsuya Ito for helpful correspondence, and the anonymous referees for their feedback.

\section{Heegaard Floer homology of surgeries on knots} \label{sec:surgery}

\subsection{The Heegaard Floer mapping cone formula}

Knot Floer homology \cite{osz-knot,rasmussen-thesis} assigns to any nullhomologous knot $K\subset S^3$ a graded, $\Z\oplus\Z$-filtered chain complex
\[ (\cfkinfty(K),\partial^\infty), \]
whose filtered chain homotopy type completely determines the Heegaard Floer homology of Dehn surgeries on $K$, where we recall that we are working with coefficients in $\Q$ throughout.

As a matter of convention, we use coordinates $(i,j)$ to refer to the two filtration levels, and notation like
\[ C\{i=0,j \leq 1\} \subset \cfkinfty(K) \]
to refer to the subquotient spanned by generators which lie in the indicated subset of the $(i,j)$-plane.  We will also use the shorthand
\[ C\{i_0,j_0\} := C\{i=i_0,j=j_0\}. \]
The differential lowers the grading by $1$ and does not increase either filtration, meaning that each $C_*\{i\leq i_0,\ j\leq j_0\}$ is a subcomplex: we have
\[ \partial^\infty\left( C_*\{i \leq i_0,\ j\leq j_0\}\right ) \subset C_{*-1}\{i \leq i_0,\ j\leq j_0\} \]
for all $(i_0,j_0)$.

With this in mind, following \cite[\S4.5 and \S5.1]{rasmussen-thesis}, one can take $\cfkinfty$ to be freely generated over $\Q[U,U^{-1}]$ by $\hfkhat(K;\Q)$.  We take
\[ C_*\{0,a\} \cong \hfkhat_*(K,a;\Q), \]
and the $U$-action gives isomorphisms
\[ U^k: C_*\{0,a\} \xrightarrow{\cong} C_{*-2k}\{-k,a-k\} \]
for all $k\in\Z$.  In the form specified here, the restriction of the differential $\partial^\infty$ to each $C\{i_0,j_0\}$ is zero.  See the ``reduction lemma'' of \cite[\S2.1]{hedden-watson} for details.

Given this, there are by definition a pair of chain homotopy equivalences
\[ C\{i=0\} \simeq \cfhat(S^3), \]
so the induced complex $(C\{i=0\},\partial')$ has homology $\hfhat(S^3) \cong \Q$ supported in grading $0$.  The Ozsv\'ath--Szab\'o tau invariant $\tau(K)$ \cite{osz-tau} is the minimum $j$-filtration level at which this generator appears.  Similarly, we have a chain homotopy equivalence
\[ C\{i \geq 0\} \simeq \cfp(S^3) \]
and then
\[ H_*\{i \geq 0\} \cong \hfp(S^3) \cong \cT^+ :=  \frac{\Q[U,U^{-1}]}{U\cdot \Q[U]}. \]

\begin{definition}
Given $\cfkinfty(K)$ as above, we define subquotient complexes
\begin{align*}
A^+_s &= C\{ \max(i,j-s) \geq 0 \}, &
B^+ &= C \{i \geq 0\} \\
\hat{A}_s^{\hphantom{+}} &= C\{ \max(i,j-s) = 0 \}, &
\hat{B}^{\hphantom{+}} &= C\{i = 0\}
\end{align*}
with differentials induced by $\partial^\infty$, for all $s \in \Z$.  These come with chain maps
\begin{align*}
v^+_s: A^+_s &\to B^+, &
h^+_s: A^+_s &\to B^+
\end{align*}
in which $v^+_s$ is defined by projection onto $C\{i\geq 0\}$, and $h^+_s$ is a composition
\[ A^+_s \xrightarrow{\text{proj}} C\{j\geq s\} \xrightarrow{U^{s}} C\{j\geq 0\} \xrightarrow{\simeq} C\{i\geq 0\} = B^+. \]
The last arrow is a homotopy equivalence which exchanges the $i$ and $j$ filtrations; we omit its definition.
\end{definition}

\begin{remark} \label{rem:v-h-isomorphisms}
The projection $v^+_s$ is an isomorphism at the chain level for all $s \geq g(K)$, since the kernel consists of the direct sum of subspaces
\[ C\{i,j\} \cong C\{0,j-i\} \cong \hfkhat(K,j-i) \]
with $i\leq -1$ and $j \geq s \geq g(K)$, and then $\hfkhat(K,j-i)=0$ because $j-i \geq g(K)+1$.  Similarly, each $h^+_s$ is an isomorphism for all $s \leq -g(K)$.
\end{remark}

These complexes determine the Heegaard Floer homology of ``large'' surgeries on $K$, in the following sense.

\begin{theorem}[{\cite[Theorem~4.4]{osz-knot}}] \label{thm:large-surgeries}
Choose a positive integer $p \geq 2g(K) - 1$.  Then there is a canonical affine map $\spc(S^3_p(K)) \cong \Z/p\Z$ (see \cite[Lemma~2.2]{osz-integral}) such that we have relatively graded isomorphisms
\[ \hfp(S^3_p(K),s) \cong H_*(A^+_s) \quad\text{and}\quad \hfhat(S^3_p(K),s) \cong H_*(\hat{A}_s) \]
for any integer $s$ with $|s| \leq \frac{p}{2}$.
\end{theorem}

\begin{remark} \label{rem:s-conjugate}
The definition of the map $\spc(S^3_p(K)) \cong \Z/p\Z$ in \cite[Lemma~2.2]{osz-integral} implies that if $\spinc \in \spc(S^3_p(K))$ is identified with $s\in\Z/p\Z$, then the conjugate $\spc$ structure $\overline{\spinc}$ is identified with $-s$.
\end{remark}

They also determine the invariants of arbitrary Dehn surgery, though in a more complicated way.  Given relatively prime integers $p,q > 0$ and arbitrary $i\in\Z$, we define 
\begin{align*}
\bA^+_i &= \bigoplus_{s\in\Z} \left(s, A^+_{\lfloor(i+ps)/q\rfloor}\right), &
\bB^+_i &= \bigoplus_{s\in\Z} (s, B^+)
\end{align*}
and a chain map
\begin{align*}
D^+_{i,p/q}: \bA^+_i &\to \bB^+_i \\
(s,a_s) &\mapsto \left(s, v^+_{\lfloor (i+ps)/q\rfloor}(a_s)\right) + \left(s+1, h^+_{\lfloor (i+ps)/q\rfloor}(a_s)\right).
\end{align*}
The various $A^+$ and $B^+$ summands each inherit relative gradings from $\cfkinfty(K)$.  We place a relative grading on their direct sums $\bA^+_i$ and $\bB^+_i$, respecting the relative gradings on each individual summand, so that $D^+_{i,p/q}$ lowers the grading by $1$.

\begin{theorem}[{\cite[Theorem~1.1]{osz-rational}}] \label{thm:mapping-cone}
Let $\bX^+_{i,p/q}$ denote the mapping cone of the chain map $D^+_{i,p/q}: \bA^+_i \to \bB^+_i$.  Then there is a natural identification $\spc(S^3_{p/q}(K)) \cong \Z/p\Z$ for which we have a relatively graded isomorphism
\[ H_*(\bX^+_{i,p/q}) \cong \hfp(S^3_{p/q}(K),i) \]
for all $i\in \Z/p\Z$.
\end{theorem}

The $A^+_s$ complexes have homology of the form
\[ H_*(A^+_s) \cong \cT^+ \oplus \Hred(A^+_s), \]
 where $\Hred(A^+_s)$ is finitely generated over $\Q$, and the maps $v^+_s$ and $h^+_s$ restrict to surjections
\[ (v^+_s)_*, (h^+_s)_*: \cT^+ \to H_*(B^+) \cong \cT^+. \]
Each of these maps is then multiplication by some nonnegative power of $U$, and we define
\[ V_s(K), H_s(K) \in \Z_{\geq 0} \]
to be these exponents.

\begin{proposition}[\cite{ni-wu,hlz}] \label{prop:vs-hs}
The invariants $V_s = V_s(K)$ and $H_s = H_s(K)$ satisfy the following constraints.
\begin{enumerate}
\item $V_s \geq V_{s+1}$ and $H_s \leq H_{s+1}$ for all $s\in\Z$. \cite[Lemma~2.4]{ni-wu}
\item $V_s = 0$ for all $s \geq g(K)$. \cite[\S2.2]{ni-wu}
\item $V_{-s} = V_s+s$ for all $s\in\Z$. \cite[Lemma~2.5]{hlz}
\item $H_{-s} = V_s$ for all $s\in\Z$. \cite[Lemma~2.3]{hlz}
\item $V_{s+1} \leq V_s \leq V_{s+1}+1$ for all $s\in\Z$. \label{i:vs-plus-one}
\end{enumerate}
\end{proposition}

\begin{proof}
Only the inequality $V_s \leq V_{s+1}+1$ of item \eqref{i:vs-plus-one} needs to be proved.  Combining the other parts of the proposition, we have
\[ V_s = V_{-s}-s \leq V_{-s-1}-s = \left(V_{-(s+1)}-(s+1)\right)+1 = V_{s+1}+1 \]
as desired.
\end{proof}

The following results relate the invariants $V_s(K)$ and $H_s(K)$ to $\hfp(S^3_{p/q}(K))$.

\begin{theorem}[{\cite[Proposition~1.6]{ni-wu}}] \label{thm:ni-wu-d}
Given relatively prime $p,q > 0$ and an integer $i$ with $0 \leq i \leq p-1$, we have
\[ d(S^3_{p/q}(K),i) - d(S^3_{p/q}(U),i) = -2\max\big( V_{\lfloor\frac{i}{q}\rfloor}(K), H_{\lfloor\frac{i-p}{q}\rfloor}(K) \big). \]
\end{theorem}

\begin{lemma} \label{lem:ker-v-g-1}
If $K \subset S^3$ has genus $g\geq 1$, then there is a short exact sequence%
\[ 0 \to \hfkhat_{*+2}(K,g) \to H_*(A^+_{g-1}) \xrightarrow{(v^+_{g-1})_*} H_*(B^+) \to 0 \]
of $\Q[U]$-modules, in which $\hfkhat(K,g)$ has trivial $U$-action and $A^+_{g-1}$ and $B^+$ are equipped with absolute gradings as quotients of $\cfkinfty(K)$.  In particular, for $N \geq 2g-1$ we have $U \cdot \hfred(S^3_N(K), g-1) = 0$ and
\[ \dim \hfred(S^3_N(K),g-1) = \dim \hfkhat(K,g) - V_{g-1}(K). \]
\end{lemma}

\begin{proof}
The short exact sequence is \cite[Lemma~3.3]{osz-characterizing}.  To prove it, we use the short exact sequence of chain complexes
\begin{equation} \label{eq:hfk-top-sequence}
0 \to C\{-1,g-1\} \to A^+_{g-1} \xrightarrow{v^+_{g-1}} B^+ \to 0
\end{equation}
defined by the natural inclusion and projection maps, which induces a long exact sequence
\[ \dots \to H_*(C\{-1,g-1\}) \to H_*(A^+_{g-1}) \xrightarrow{(v^+_{g-1})_*} H_*(B^+) \to \dots \]
on homology.  The complex $C\{-1,g-1\}$ has zero differential and trivial $U$ action, and it is the image under $U$ of
\[ C\{0,g\} \cong \hfkhat(K,g), \]
hence its homology is just $\hfkhat_{*+2}(K,g)$.  Meanwhile we know that $H_*(B^+) \cong \cT^+$, and $v^+_{g-1}$ is an isomorphism in all sufficiently large gradings, so it follows that $H_*(A^+_{g-1})$ also contains a tower $\cT^+$ which surjects onto $H_*(B^+)$.  Thus the long exact sequence splits.

The claim about $\dim \hfred(S^3_N(K),g-1)$ now follows quickly from Theorem~\ref{thm:large-surgeries}, because we can identify $\ker v^+_{g-1}$ with all of $\hfred(S^3_N(K),g-1)$ plus whatever portion of $\cT^+ \subset H_*(A^+_{g-1})$ is in the kernel, and the latter has dimension $V_{g-1}$ by definition.
\end{proof}

Although Theorem~\ref{thm:mapping-cone} as stated only determines the relative grading on $\hfp(S^3_{p/q}(K))$, we can use the integers $V_s$ and $H_s$ to recover the absolute grading by Theorem \ref{thm:ni-wu-d}.

\begin{proposition} \label{prop:compare-gradings}
Suppose for some knots $K,K' \subset S^3$ and some relatively prime $p,q > 0$ that
\[ \hfp(S^3_{p/q}(K)) \cong \hfp(S^3_{p/q}(K')) \]
as graded $\Q[U]$-modules.  Then we have $\Delta''_K(1) = \Delta''_{K'}(1)$.
Moreover, if $g(K)=1$ then $V_0(K)=V_0(K')$, and if in addition $\frac{p}{q} > 1$ then $V_s(K')=0$ for all $s \geq 1$.
\end{proposition}

\begin{proof}
Let $Y$ be the rational homology 3-sphere $S^3_{p/q}(K)$.  Rustamov \cite[Theorem~3.3]{rustamov} proved that its Casson--Walker invariant satisfies
\[ |H_1(Y;\Z)| \lambda(Y) = \sum_{\spinc \in \spc(Y)} \left( \chi(\hfred(Y,\spinc)) - \tfrac{1}{2}d(Y,\spinc) \right), \]
and the right hand side is completely determined by $\hfp(Y)$, hence so is $\lambda(Y)$.  The surgery formula for the Casson--Walker invariant \cite[Theorem~4.2]{walker} then says that
\[ \lambda(Y) - \lambda(S^3_{p/q}(U)) = \frac{q}{p} \frac{\Delta''_K(1)}{2}, \]
so we conclude that $\Delta''_K(1)$ is determined by $\frac{p}{q}$ and $\hfp(S^3_{p/q}(K))$.  By hypothesis the same data determines $\Delta''_{K'}(1)$ in exactly the same way, so these second derivatives are equal.

Now suppose that $g(K)=1$.  Then Proposition~\ref{prop:vs-hs} says that $V_s(K)=H_{-s}(K) = 0$ for all $s \geq 1$, and then that $V_0(K)$ is either $0$ or $1$ since $V_1(K)=0$. We therefore have
\[ d(S^3_{p/q}(K),i) - d(S^3_{p/q}(U),i) = \begin{cases} -2V_0(K), & 0 \leq i \leq \min(p,q)-1 \\ 0, & \min(p,q) \leq i \leq p-1 \end{cases} \]
by Theorem~\ref{thm:ni-wu-d}.  It follows that
\begin{equation} \label{eq:sum-d-k-surgery}
\sum_{i\in\Z/p\Z} \left(d(S^3_{p/q}(K),i) - d(S^3_{p/q}(U),i)\right) = -2V_0(K) \cdot \min(p,q).
\end{equation}
By the same argument we have
\begin{equation} \label{eq:sum-d-kp-surgery}
\sum_{i\in\Z/p\Z} \left(d(S^3_{p/q}(K'),i) - d(S^3_{p/q}(U),i)\right) \leq -2V_0(K') \cdot \min(p,q),
\end{equation}
and the left sides of \eqref{eq:sum-d-k-surgery} and \eqref{eq:sum-d-kp-surgery} are equal, so $V_0(K') \leq V_0(K) \leq 1$.    If $V_0(K') = 0$ then $V_s(K')=0$ for all $s \geq 0$, so the left side of \eqref{eq:sum-d-kp-surgery} is equal to $0$, hence $V_0(K) = 0$ as well.  Otherwise $V_0(K')=1$ implies that $V_0(K)=1$, so in any case we have $V_0(K)=V_0(K')$.

Finally, if $V_0(K) = V_0(K') = 1$ and $p > q$ then we have by Theorem \ref{thm:ni-wu-d} that
\begin{align*}
d(S^3_{p/q}(K'),q) - d(S^3_{p/q}(U),q) &= -2\max(V_1(K'),H_{\lfloor\frac{q-p}{q}\rfloor}(K')) \\
&\leq -2 V_1(K'),
\end{align*}
which implies that the left side of \eqref{eq:sum-d-kp-surgery} is at most $-2q V_0(K') - 2V_1(K')$.  But this is equal to the left side of \eqref{eq:sum-d-k-surgery},
which is equal to 
\[ -2V_0(K)\cdot q = -2q V_0(K'), \]
so we must have $V_1(K') = 0$.  Then $V_s(K')=0$ for all $s \geq 1$ by Proposition~\ref{prop:vs-hs}.
\end{proof}

\subsection{Computations for nearly fibered knots} \label{sec:compute-nearly-fibered}

In this subsection we work out some examples of the large surgery formula.  Let $K$ be a genus-1 knot for which $\hfkhat(K,1)$ is 2-dimensional.  Then $K$ is one of the knots listed in Theorem~\ref{thm:main-hfk}, with $\hfkhat(K)$ shown in Table~\ref{fig:hfk-table}, and in every case there is some integer $m\in\Z$ such that
\[ \hfkhat(K,1) \cong \Q^2_{(m)}, \]
where the subscripts denote the Maslov grading.  (For the mirrors of the knots in Table~\ref{fig:hfk-table}, this follows from the relation $\hfkhat_m(\mirror{K},a) \cong \hfkhat_{-m}(K,-a)$.)

We first determine $\hfp(S^3_1(K))$ in the cases where $K$ is either $5_2$ or its mirror.

\begin{proposition} \label{prop:1-surgery-5_2}
We have 
\[ \hfp(S^3_1(5_2)) \cong \cT^{+}_{(0)} \oplus \Q^{2\vphantom{+}}_{(0)} \quad\text{and}\quad \hfp(S^3_1(\mirror{5_2})) \cong \cT^+_{(-2)} \oplus \Q^{\vphantom{+}}_{(-2)} \]
as graded $\Q[U]$-modules.
\end{proposition}

\begin{proof}
In these cases $K$ is alternating, so $\hfkhat(K)$ is \emph{thin} --- there is some $s\in\Z$ such that each $\hfkhat(K,a)$ is supported in homological grading $a-s$ --- and for alternating knots we have $s=-\frac{1}{2}\sigma(K)$ \cite[Theorem~1.3]{osz-alternating}, where $\sigma(K)$ is the signature.
(This uses the convention that positive knots such as $\mirror{5_2}$ have negative signature, so $\sigma(\mirror{5_2}) = -2$ and $\sigma(5_2) = 2$.)  In this case the differential on $\cfkinfty(K)$ has a fairly simple form, namely
\[ \partial^\infty\left( C\{i_0,j_0\} \right) \subset C\{i_0-1,j_0\} \oplus C\{i_0,j_0-1\}, \]
by the fact that $\deg(\partial^\infty)=-1$.  Since $H_*(C\{i=0\}) \cong \Q$ is supported at Alexander grading $j=\tau(K)$ in homological grading $0$, we have $\tau(K) = -\frac{1}{2}\sigma(K)$, so
\[ \tau(5_2) = -1 \quad\text{and}\quad \tau(\mirror{5_2}) = 1. \]

We can therefore find bases for the complexes $(C\{i=0\}, \partial')$ so that they are represented by the diagrams
\[ \begin{tikzpicture}
\foreach \j in {-1,0,1} {
  \node[left] at (-2,\j) {$j=\j$:};
}
\begin{scope}
\foreach \i/\j in {-1/1,1/1,-2/0,0/0,2/0,-1/-1,1/-1} {
  \draw[fill=black] (0.15*\i,\j) coordinate (x\i_\j) circle (0.075);
}
\draw[-latex] ($(x-1_1)+(0,-0.125)$) -- ($(x0_0)+(0,0.1)$);
\draw[-latex] ($(x1_1)+(0,-0.125)$) -- ($(x2_0)+(0,0.1)$);
\draw[-latex] ($(x-2_0)+(0,-0.125)$) -- ($(x-1_-1)+(0,0.1)$);
\end{scope}
\node at (1.5,0) {or};
\begin{scope}[xshift=3cm]
\foreach \i/\j in {-1/1,1/1,-2/0,0/0,2/0,-1/-1,1/-1} {
  \draw[fill=black] (0.15*\i,\j) coordinate (x\i_\j) circle (0.075);
}
\draw[-latex] ($(x1_1)+(0,-0.125)$) -- ($(x2_0)+(0,0.1)$);
\draw[-latex] ($(x-2_0)+(0,-0.125)$) -- ($(x-1_-1)+(0,0.1)$);
\draw[-latex] ($(x0_0)+(0,-0.125)$) -- ($(x1_-1)+(0,0.1)$);
\end{scope}
\end{tikzpicture} \]
for $5_2$ and $\mirror{5_2}$ respectively.  (Here each dot represents a generator of a $\Q$ summand, and an arrow of the form ``$\bullet\to\bullet$'' means that the corresponding generators $x$ and $y$ satisfy $\partial'x=y$.)  In turn, this together with the chain homotopy equivalence $C\{i=0\} \simeq C\{j=0\}$ and the requirement that $(\partial^\infty)^2=0$ completely determines $\cfkinfty(K)$ for each of these knots $K$.

\begin{figure}
\begin{tikzpicture}
\begin{scope}[scale=0.8] 
\begin{scope}
\clip (-3.3,-3.3) rectangle (3.3,3.3);
\path[fill=gray!10] (-4,-4) rectangle (4,4);
\path[fill=white] (-0.5,-0.5) rectangle (-4,-4);
\draw[gray!30,step=1.5] (-4,-4) grid (4,4);
\end{scope}
\draw[gray!50,very thick,<->] (-3.5,0) -- (3.5,0) node[right,black] {\small$i$};
\draw[gray!50,very thick,<->] (0,-3.5) -- (0,3.5) node[above,black] {\small$j$};
\begin{scope}
\clip (-3.3,-3.3) rectangle (3.3,3.3);
\foreach \i in {-2,...,2} {
  \foreach \j in {0,1} {
    \draw[fill=black] (1.5*\i,1.5*\i+1.5) ++ (0.2*\j,-0.2*\j) coordinate (a1_i\i_j\j) circle (0.075);
  }
  \foreach \j in {-1,0,1} {
    \draw[fill=black] (1.5*\i,1.5*\i) ++ (0.2*\j,0.2*\j) coordinate (a0_i\i_j\j) circle (0.075);
  }
  \foreach \j in {-1,0} {
    \draw[fill=black] (1.5*\i,1.5*\i-1.5) ++ (0.2*\j,-0.2*\j) coordinate (a-1_i\i_j\j) circle (0.075);
  }
}
\foreach \i in {-2,-1,0,1} { 
  \draw[-latex] ($(a1_i\i_j0)+(0,-0.125)$) -- ($(a0_i\i_j0)+(0,0.1)$); 
}
\foreach \i/\k in {-1/-2,0/-1,1/0,2/1} { 
  \draw[-latex] ($(a-1_i\i_j0)+(-0.125,0)$) -- ($(a0_i\k_j0)+(0.1,0)$); 
  \draw[-latex] ($(a0_i\i_j-1)+(-0.125,0)$) -- ($(a1_i\k_j1)+(0.1,0)$); 
  \draw[-latex] ($(a0_i\i_j-1)+(0,-0.125)$) -- ($(a-1_i\i_j-1)+(0,0.1)$); 
  \draw[-latex] ($(a1_i\k_j1)+(0,-0.125)$) -- ($(a0_i\k_j1)+(0,0.1)$); 
  \draw[-latex] ($(a-1_i\i_j-1)+(-0.125,0)$) -- node[pos=0.45,above,inner sep=1pt,outer sep=0pt] {\tiny$-$} ($(a0_i\k_j1)+(0.1,0)$); 
}
\node[above left,inner sep=2pt] at (a1_i-1_j0) {\small$a$};
\node[left,inner sep=3pt] at ($(a1_i0_j0)+(0,0.075)$) {\scriptsize$U^{-1}a$};
\node[left,inner sep=3pt] at ($(a1_i1_j0)+(0,0.075)$) {\scriptsize$U^{-2}a$};
\node[below right,inner sep=2pt] at (a-1_i0_j0) {\small$d$};
\node[below right,inner sep=2pt] at (a-1_i1_j0) {\scriptsize$U^{-1}d$};
\node[below left,inner sep=3pt] at ($(a-1_i2_j0)+(0.375,0.025)$) {\scriptsize$U^{-2}d$};
\draw[thin,->] (-1.5,0.75) node[left,inner sep=2pt] {\small$b$} to[bend left=50] ($(a1_i-1_j1)+(0.025,0.1)$);
\draw[thin,->] (0.75,-1.625) node[below,inner sep=2pt] {\small$c$} to[bend right=50] ($(a-1_i0_j-1)+(0.1,0.025)$);
\end{scope}
\node[below] at (0,-3.5) {$\cfkinfty(5_2)\vphantom{\mirror{5_2}}$};
\end{scope}

\begin{scope}[scale=0.8,xshift=8cm] 
\begin{scope}
\clip (-3.3,-3.3) rectangle (3.3,3.3);
\path[fill=gray!10] (-4,-4) rectangle (4,4);
\path[fill=white] (-0.5,-0.5) rectangle (-4,-4);
\draw[gray!30,step=1.5] (-4,-4) grid (4,4);
\end{scope}
\draw[gray!50,very thick,<->] (-3.5,0) -- (3.5,0) node[right,black] {\small$i$};
\draw[gray!50,very thick,<->] (0,-3.5) -- (0,3.5) node[above,black] {\small$j$};
\begin{scope}
\clip (-3.3,-3.3) rectangle (3.3,3.3);
\foreach \i in {-2,...,2} {
  \foreach \j in {0,1} {
    \draw[fill=black] (1.5*\i,1.5*\i+1.5) ++ (0.2*\j,-0.2*\j) coordinate (a1_i\i_j\j) circle (0.075);
  }
  \foreach \j in {-1,0,1} {
    \draw[fill=black] (1.5*\i,1.5*\i) ++ (0.2*\j,0.2*\j) coordinate (a0_i\i_j\j) circle (0.075);
  }
  \foreach \j in {-1,0} {
    \draw[fill=black] (1.5*\i,1.5*\i-1.5) ++ (0.2*\j,-0.2*\j) coordinate (a-1_i\i_j\j) circle (0.075);
  }
}
\foreach \i/\k in {-1/-2,0/-1,1/0,2/1} { 
  \draw[-latex] ($(a0_i\i_j0)+(-0.125,0)$) -- ($(a1_i\k_j0)+(0.1,0)$); 
  \draw[-latex] ($(a0_i\i_j0)+(0,-0.125)$) -- ($(a-1_i\i_j0)+(0,0.1)$); 
  \draw[-latex] ($(a0_i\i_j-1)+(-0.125,0)$) -- ($(a1_i\k_j1)+(0.1,0)$); 
  \draw[-latex] ($(a0_i\i_j-1)+(0,-0.125)$) -- ($(a-1_i\i_j-1)+(0,0.1)$); 
  \draw[-latex] ($(a1_i\k_j1)+(0,-0.125)$) -- ($(a0_i\k_j1)+(0,0.1)$); 
  \draw[-latex] ($(a-1_i\i_j-1)+(-0.125,0)$) -- node[pos=0.45,above,inner sep=1pt,outer sep=0pt] {\tiny$-$} ($(a0_i\k_j1)+(0.1,0)$); 
}
\node[above left,inner sep=2pt] at (a1_i-1_j0) {\small$x$};
\node[left,inner sep=3pt] at ($(a1_i0_j0)+(0,0.075)$) {\scriptsize$U^{-1}x$};
\node[left,inner sep=3pt] at ($(a1_i1_j0)+(0,0.075)$) {\scriptsize$U^{-2}x$};
\draw[thin,->] (a1_i-1_j1) ++(-1,-0.2) node[left, inner sep=2pt] {\small$y$} to[bend right=20] ++(0.9,0.2);
\end{scope}
\node[below] at (0,-3.5) {$\cfkinfty(\mirror{5_2})$};
\end{scope}
\end{tikzpicture}
\caption{The complexes $(\cfkinfty(K),\partial^\infty)$ for $K=5_2$ and $K=\mirror{5_2}$, with $A^+_0$ shaded.  The dots represent generators of $C\{i,j\}$, all of which lie in grading $i+j+\frac{\sigma(K)}{2}$.  Minus signs on arrows indicate a coefficient of $-1$.}
\label{fig:cfk-5_2}
\end{figure}
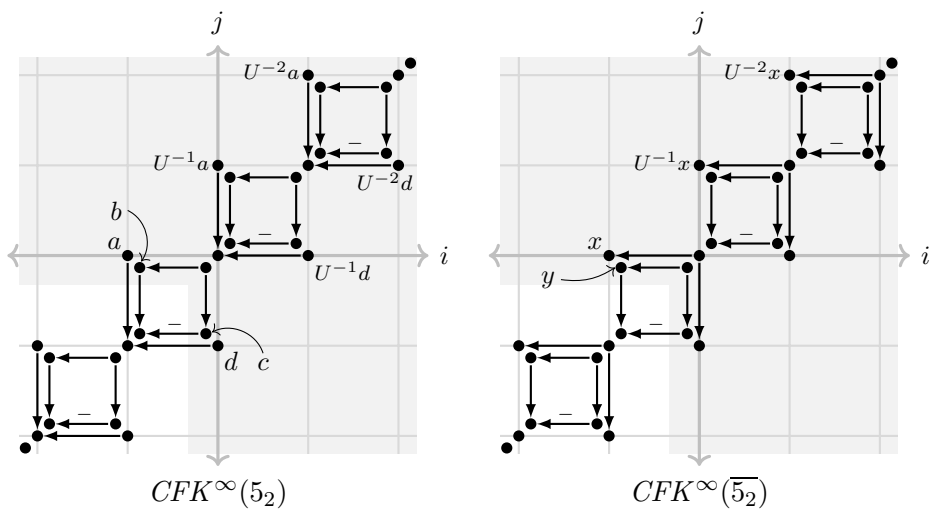

Now by inspecting Figure~\ref{fig:cfk-5_2} we see that
\begin{equation} \label{eq:5_2-A0}
H_*(A^+_0(5_2)) \cong \frac{\Q[U,U^{-1}]}{U\cdot \Q[U]} \langle d-a\rangle \oplus \Q\langle a,b \rangle \cong \cT^{+}_{(0)} \oplus \Q^{2\vphantom{+}}_{(0)},
\end{equation} since the indicated elements $a,b,d$ all have homological grading $-1+\frac{\sigma(5_2)}{2} = 0$.  The homology $H_*(B^+(5_2)) \cong \cT^+$ has bottom-most element $[d] = (v^+_0)_*([d-a])$, so then $(v^+_0)_*|_{\cT^+}: \cT^+ \to \cT^+$ is an isomorphism and we have $V_0(5_2) = 0$.  Now Theorem~\ref{thm:large-surgeries} says that $\hfp(S^3_1(5_2)) \cong H_*(A^+_0(5_2))$ as relatively graded groups, while Theorem~\ref{thm:ni-wu-d} says that the tower $\cT^+$ in $\hfp(S^3_1(5_2))$ has bottom-most grading $d(S^3_1(5_2))=-2V_0(5_2) = 0$, so we conclude that $\hfp(S^3_1(5_2))$ is exactly as claimed.

Similarly, we see from Figure~\ref{fig:cfk-5_2} that
\[ H_*(A^+_0(\mirror{5_2})) \cong \frac{\Q[U,U^{-1}]}{U\cdot \Q[U]} \langle x\rangle \oplus \Q\langle y\rangle \cong \cT^+_{(-2)} \oplus \Q^{\vphantom{+}}_{(-2)}, \]
since the indicated elements $x,y \in C\{-1,0\}$ have homological grading $-1+\frac{\sigma(\mirror{5_2})}{2} = -2$.  The kernel of $(v^+_0)_*$ contains $[x]$ but not $[U^{-1}x]$, so the restriction $(v^+_0)_*|_{\cT^+}: \cT^+ \to \cT^+$ is multiplication by $U$, hence $V_0(\mirror{5_2}) = 1$.  Now we conclude exactly as before that $d(S^3_1(\mirror{5_2})) = -2V_0(\mirror{5_2}) = -2$ and hence that $\hfp(S^3_1(\mirror{5_2}))$ is exactly as claimed.
\end{proof}

For the knots of Theorem~\ref{thm:main-hfk} other than $5_2$ and $\mirror{5_2}$, it is a little bit harder to determine $\cfkinfty(K)$.  We will avoid this problem by using the large surgery formula to compute $\hfhat(S^3_1(K))$, and then deducing $\hfp(S^3_1(K))$ from this in Proposition~\ref{prop:nf-1-surgery-plus}.

\begin{proposition} \label{prop:nf-1-surgery-hat}
Let $K$ be a genus-1 knot for which $\hfkhat(K,1) \cong \Q^2_{(m_0+1)}$.  If $K$ is neither $5_2$ nor its mirror, then $\tau(K)=0$ and
\[ \hfhat(S^3_1(K)) \cong \Q_{(0)} \oplus \left(\Q_{(m_0)} \oplus \Q_{(m_0-1)}\right)^{\oplus 2} \]
as relatively graded $\Q$-vector spaces.
\end{proposition}

\begin{proof}
We attempt to construct the full knot Floer complex $\cfkinfty(K)$.  The relation
\[ \hfkhat_m(K,a) \cong \hfkhat_{m-2a}(K,-a) \]
tells us that if $\hfkhat(K,1) \cong \Q^2_{(m_0+1)}$ then $\hfkhat(K,-1) \cong \Q^2_{(m_0-1)}$, so the model complex $(C\{i=0\},\partial')$ for $\cfhat(S^3)$ has the form
\[ \begin{tikzcd}
\Q^2_{(m_0+1)} \ar[d,"\partial'_1"] \ar[dd,bend right=90,swap,"\partial'_2"] \\
\hfkhat(K,0) \ar[d,"\partial'_0"] \\
\Q^2_{(m_0-1)},
\end{tikzcd} \]
and in fact the $\partial'_2$ component of the differential must be zero since it cannot lower the grading by $2$. 

If $K$ is neither $5_2$ nor its mirror, then we can read $\dim \hfkhat(K,0) = 5$ off of Table~\ref{fig:hfk-table}, and so $H_*(C\{i=0\},\partial') \cong \Q$ is only possible if $\partial'_1$ is injective and $\partial'_2$ is surjective.  Moreover, the homology is necessarily supported in Alexander grading $j=0$, so $\tau(K) = 0$.  This completely determines the $i$-preserving (vertical) component of $\partial^\infty$, as illustrated in Figure~\ref{fig:cfk-nearly-fibered}.
\begin{figure}
\begin{tikzpicture}
\begin{scope}
\path[fill=gray!15] (0.75,0.75) -- ++(-4,0) -- ++(0,-1.5) -- ++(2.5,0) -- ++(0,-2.5) -- ++(1.5,0) -- ++(0,4);
\draw[gray!30,step=2] (-3.25,-3.25) grid (3.25,3.25);
\draw[gray!50,very thick,<->] (-3.5,0) -- (3.5,0) node[right,black] {\small$i$};
\draw[gray!50,very thick,<->] (0,-3.5) -- (0,3.5) node[above,black] {\small$j$};
\begin{scope}
\clip (-3.25,-3.25) rectangle (3.25,3.25);
\foreach \i in {-2,...,2} {
  \foreach \j in {-2,-1,1,2} {
    \draw[fill=black] (2*\i,2*\i) ++ (0.2*\j,0.2*\j) coordinate (i\i_j\j-0) circle (0.075);
  }
  \draw[fill=white] (2*\i,2*\i) circle (0.1);
  \foreach \j in {2,1} {
    \draw[fill=black] (2*\i,2*\i+2) ++ (0.2*\j,-0.2*\j) coordinate (i\i_j\j_plus1) circle (0.075);
    \draw[fill=black] (2*\i,2*\i-2) ++ (-0.2*\j,0.2*\j) coordinate (i\i_j\j_minus1) circle (0.075);
  }
  \draw[-latex] ($(i\i_j2_plus1)+(0,-0.125)$) -- ($(i\i_j2-0)+(0,0.1)$); 
  \draw[-latex] ($(i\i_j1_plus1)+(0,-0.125)$) -- ($(i\i_j1-0)+(0,0.1)$);
  \draw[-latex] ($(i\i_j-2-0)+(0,-0.125)$) -- ($(i\i_j2_minus1)+(0,0.1)$); 
  \draw[-latex] ($(i\i_j-1-0)+(0,-0.125)$) -- ($(i\i_j1_minus1)+(0,0.1)$);
}
\foreach \i/\k in {-1/-2,0/-1,1/0,2/1} { 
  \draw[-latex] ($(i\i_j-2-0)+(-0.125,0)$) -- ($(i\k_j2_plus1)+(0.1,0)$);
  \draw[-latex] ($(i\i_j-1-0)+(-0.125,0)$) -- ($(i\k_j1_plus1)+(0.1,0)$);
  \draw[-latex] ($(i\i_j2_minus1)+(-0.125,0)$) -- node[above,midway,inner sep=1pt,outer sep=0pt] {\tiny$-$} ($(i\k_j2-0)+(0.1,0)$);
  \draw[-latex] ($(i\i_j1_minus1)+(-0.125,0)$) -- node[below,midway,inner sep=1pt,outer sep=0pt] {\tiny$-$} ($(i\k_j1-0)+(0.1,0)$);
}
\draw[thin,->] (-1.5,1) node[left,inner sep=2pt] {\small$x$} to[bend left=20] (-0.1,0.1);
\end{scope}
\end{scope}
\end{tikzpicture}
\caption{The complex $(\cfkinfty(K),\partial^\infty)$, with $\hat{A}_0$ shaded and possible diagonal arrows omitted.  The black dots represent generators of $C\{i,j\}$ in grading $m_0+i+j$, while the white dots represent generators $U^{-i} x$ in grading $i+j$.  The minus signs on some arrows indicate a coefficient of $-1$ in $\partial^\infty$.}
\label{fig:cfk-nearly-fibered}
\end{figure}
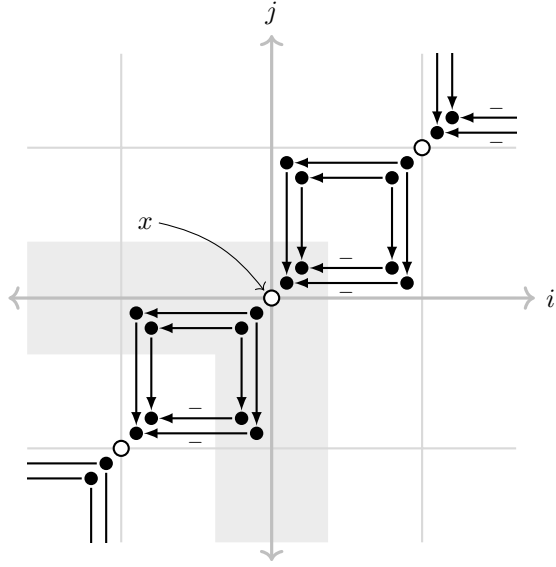
The chain homotopy equivalence $C\{i=0\} \simeq C\{j=0\}$ and the relation $(\partial^\infty)^2 = 0$ then nearly suffice to determine $\partial^\infty$; the only ambiguity is whether there are any arrows involving the generators $U^k x \in C\{i=j=-k\}$, and these must be diagonal (meaning neither vertical nor horizontal) if they exist.

This discussion completely determines the subquotient complex $\hat{A}_0$, which is shaded in Figure~\ref{fig:cfk-nearly-fibered}, since it does not see any diagonal arrows that might exist in $\cfkinfty(K)$.  The complex has nine generators, only two of which have nonzero differentials, and the hat version of the large surgery formula in Theorem~\ref{thm:large-surgeries} tells us that
\[ \hfhat(S^3_1(K)) \cong H_*(\hat{A}_0) \cong \Q_0 \oplus \Q^2_{(m_0)} \oplus \Q^2_{(m_0-1)} \]
as relatively graded vector spaces.
\end{proof}

If $Y$ is an arbitrary 3-manifold with torsion $\spc$ structure $\spinc$, so that its homological grading is $\Z$-valued, then the short exact sequence of complexes
\[ 0 \to \cfhat_*(Y,\spinc) \to \cfp_*(Y,\spinc) \xrightarrow{U} \cfp_{*-2}(Y,\spinc) \to 0 \]
turns into a long exact sequence of $\Q[U]$-modules
\[ \dots \to \hfp_{*+1}(Y,\spinc) \xrightarrow{U} \hfp_{*-1}(Y,\spinc) \to \hfhat_*(Y,\spinc) \to \hfp_*(Y,\spinc) \xrightarrow{U} \hfp_{*-2}(Y,\spinc) \to \dots, \]
from which we can extract a short exact sequence
\begin{equation} \label{eq:hfhat-from-u}
0 \to \frac{\hfp_{*-1}(Y,\spinc)}{U \cdot \hfp_{*+1}(Y,\spinc)} \to \hfhat_*(Y,\spinc) \to \ker(U|_{\hfp_*(Y,\spinc)}) \to 0.
\end{equation}
Equation~\eqref{eq:hfhat-from-u} immediately implies the following.

\begin{lemma} \label{lem:hfhat-vs-hfred}
If $U\cdot \hfred(Y,\spinc) = 0$ and we have an isomorphism
\[ \hfp(Y) \cong \cT^+_{(d)} \oplus \bigoplus_{i=1}^k \Q_{(n_i)} \]
of graded $\Q[U]$-modules, where the $(d)$ subscript denotes the grading of $\ker(U) \subset \cT^+$, then
\[ \hfhat(Y) \cong \Q_{(d)} \oplus \bigoplus_{i=1}^k \left( \Q_{(n_i+1)} \oplus \Q_{(n_i)} \right). \]
This implies in particular that $\dim \hfhat(Y,\spinc) = 1 + 2\dim\hfred(Y,\spinc)$.
\end{lemma}

\begin{corollary} \label{cor:hfhat-vs-hfkhat}
If $K \subset S^3$ has genus $g\geq 1$, then
\[ \frac{\dim\hfhat(S^3_{2g-1}(K),g-1)-1}{2} = \dim \hfkhat(K,g) - V_{g-1}(K). \]
\end{corollary}

\begin{proof}
Lemma~\ref{lem:ker-v-g-1} says that $U\cdot \hfred(S^3_{2g-1}(K),g-1) = 0$ and that
\[ \dim \hfred(S^3_{2g-1}(K),g-1) = \dim \hfkhat(K,g) - V_{g-1}(K). \]
Now apply Lemma~\ref{lem:hfhat-vs-hfred}.
\end{proof}

\begin{proposition} \label{prop:nf-1-surgery-plus}
Let $K$ be a genus-1 knot for which $\dim_\Q \hfkhat(K,1) = 2$.  If $K$ is neither $5_2$ nor its mirror, then $V_0(K)=0$ and
\[ \hfp(S^3_1(K)) \cong \cT^+_{(0)} \oplus \hfkhat(K,-1) \]
as graded $\Q[U]$-modules.
\end{proposition}

\begin{proof}
We write $\hfkhat(K,1) \cong \Q^2_{(m_0+1)}$ as before, and then the symmetry
\[ \hfkhat_m(K,a) \cong \hfkhat_{m-2a}(K,-a) \]
of \cite[Equation~(2)]{osz-knot} implies that $\hfkhat(K,-1) \cong \Q^2_{(m_0-1)}$.

We observe from Lemma~\ref{lem:ker-v-g-1} that $U \cdot \hfred(S^3_1(K)) = 0$, since $g(K)=1$.  In Proposition~\ref{prop:nf-1-surgery-hat} we saw that $\dim \hfhat(S^3_1(K))=5$, so Lemma~\ref{lem:hfhat-vs-hfred} says that $\dim \hfred(S^3_1(K)) = 2$.  But then
\[ V_0(K) = \dim \hfkhat(K,1) - \dim \hfred(S^3_1(K)) = 2 - 2 = 0 \]
by another application of Lemma~\ref{lem:ker-v-g-1}.  With this information at hand, Theorem~\ref{thm:ni-wu-d} tells us that
\[ d(S^3_1(K)) = d(S^3_1(U)) - 2 V_0(K) = 0. \]

Now if we write 
\[ \hfp(S^3_1(K)) \cong \cT^+_{(0)} \oplus \Q^{\vphantom{+}}_{(d)} \oplus \Q^{\vphantom{+}}_{(e)} \]
for some integers $d$ and $e$, then Lemma~\ref{lem:hfhat-vs-hfred} says that
\[ \hfhat(S^3_1(K)) \cong \Q_{(0)} \oplus \Q_{(d)} \oplus \Q_{(d+1)} \oplus \Q_{(e)} \oplus \Q_{(e+1)}. \]
Up to translation by an overall constant, Proposition~\ref{prop:nf-1-surgery-hat} says that these gradings are $0,m_0,m_0,m_0-1,m_0-1$ in some order.  This is only possible if that constant is zero and $d=e=m_0-1$, except possibly if $m_0=-1$ and $\{d,e\}=\{0,1\}$.  But we can rule out this last case because it would imply that $S^3_1(K)$ has Casson invariant
\[ \lambda(S^3_1(K)) = \chi(\hfred(S^3_1(K))) - \frac{1}{2}d(S^3_1(K)) = 0 - 0 = 0 \]
by \cite[Theorem~1.3]{osz-absolutely}, and yet $\lambda(S^3_1(K)) = \frac{\Delta_K''(1)}{2} = \pm2$ by the surgery formula for the Casson invariant.  This completes the proof.
\end{proof}

\section{The dimension of $\hfhat$} \label{sec:dim-hfhat}

\subsection{The invariants $\rhat$ and $\nuhat$}

For a fixed knot $K \subset S^3$, the dimension of $\hfhat(S^3_{p/q}(K))$ varies in a predictable way with $p$ and $q$.  We will make use of this where possible, since it is easier to apply in practice than the mapping cone formula.

\begin{proposition} \label{prop:r-nu}
Let $K \subset S^3$ be a knot.  Then there are integers $\rhat(K)$ and $\nuhat(K)$ such that
\[ \dim_\Q \hfhat(S^3_{p/q}(K)) = q\cdot \rhat(K) + |p-q\nuhat(K)| \]
for all coprime integers $p\neq 0$ and $q > 0$.
\end{proposition}

Hanselman \cite[Proposition~15]{hanselman-cosmetic} proved a version of Proposition~\ref{prop:r-nu} with coefficients in $\Z/2\Z$, though he pointed out that it can be extracted from \cite[Proposition~9.6]{osz-rational}, where it is proved with the desired $\Q$ coefficients.  (It can also be proved in exactly the same way as its instanton Floer analogue \cite[Theorem~1.1]{bs-concordance}, using only the surgery exact triangle and an adjunction inequality.)  In fact, if the Heegaard Floer $\nu$ invariant of $K$ \cite[Definition~9.1]{osz-rational} satisfies $\nu(K) \geq \nu(\mirror{K})$, then \cite[Equation~(40)]{osz-rational} implies the relation
\[ \rhat(K)-\nuhat(K) = \sum_{s\in\Z} \left(\dim H_*(\hat{A}_s)-1\right). \]
Moreover, we know from \cite[Lemma~10.4]{bs-concordance} that
\begin{equation} \label{eq:nuhat-nu}
\nuhat(K) = \begin{cases} \max(2\nu(K)-1,0), & \nu(K) \geq \nu(\mirror{K}) \\ -\max(2\nu(\mirror{K})-1,0), & \nu(K) \leq \nu(\mirror{K}). \end{cases}
\end{equation}

\begin{proposition} \label{prop:r-nu-properties}
The invariants $\rhat(K)$ and $\nuhat(K)$ satisfy the following properties.
\begin{enumerate}
\item The invariants of $K$ and its mirror are related by $(\rhat(\mirror{K}),\nuhat(\mirror{K})) = (\rhat(K),-\nuhat(K))$.\label{i:r-nu-mirror}
\item The difference $\rhat(K) - |\nuhat(K)|$ is a nonnegative even integer. \label{i:r-minus-nu}
\item $\nuhat$ is a smooth concordance invariant, and \[ |\nuhat(K)| \leq \max(2g_4(K)-1,0) \]
where $g_4$ denotes the smooth 4-ball genus. \label{i:nu-concordance}
\item The invariant $\nuhat(K)$ is either odd or zero. \label{i:nu-parity}
\item If $V_0(K) = 0$ then $\nuhat(K) \leq 0$. \label{i:nu-v0}
\item If $\nuhat(K) \leq 0$ then $\tau(K) \leq 0$, and if $\nuhat(K)=0$ then $\tau(K)=0$. \label{i:nu-tau}
\end{enumerate}
\end{proposition}

\begin{proof}
Claim \eqref{i:r-nu-mirror} is immediate from Proposition~\ref{prop:r-nu} and the relation $S^3_r(\mirror{K}) \cong -S^3_{-r}(K)$, together with the fact that $\dim \hfhat(Y) = \dim \hfhat(-Y)$ for all $Y$.
For \eqref{i:r-minus-nu}, we choose a positive integer $p > \nuhat(K)$ and apply Proposition~\ref{prop:r-nu} to get
\[ \dim \hfhat(S^3_p(K)) = p + (\rhat(K)-\nuhat(K)), \]
so by \cite[Proposition~5.1]{osz-properties} we have
\[ \dim \hfhat(S^3_p(K)) - \chi(\hfhat(S^3_p(K))) = \rhat(K)-\nuhat(K). \]
The left hand side is twice the dimension of the odd-graded part of $\hfhat(S^3_p(K))$, so it is evidently nonnegative and even.  The same is true of
\[ \rhat(\mirror{K}) - \nuhat(\mirror{K}) = \rhat(K) + \nuhat(K), \]
so in either case $\rhat(K) - |\nuhat(K)|$ is nonnegative and even as well.  Since $\nu(K)$ and $\nu(\mirror{K})$ are smooth concordance invariants, claims \eqref{i:nu-concordance} and \eqref{i:nu-parity} follow immediately from \eqref{eq:nuhat-nu} and the fact that $|\nu(K)| \leq g_4(K)$.

In order to prove \eqref{i:nu-v0}, we use the invariant $\nu^+(K)$ \cite{hom-wu}, which is by definition the smallest $s$ such that $V_s(K)=0$.  If $V_0(K)=0$ then \cite[Proposition~2.3]{hom-wu} tells us that
\[ \tau(K) \leq \nu(K) \leq \nu^+(K) = 0, \]
and since $\nu(\mirror{K})$ is equal to either $\tau(\mirror{K})$ or $\tau(\mirror{K})+1$ (see \cite[Equation~(34)]{osz-rational}) we have
\[ \nu(\mirror{K}) \geq \tau(\mirror{K}) = -\tau(K) \geq 0 \geq \nu(K). \]
Now \eqref{eq:nuhat-nu} tells us that $\nuhat(K) = -\max(2\nu(\mirror{K})-1,0) \leq 0$.  We prove the contrapositive of the first part of \eqref{i:nu-tau} similarly: if $\tau(K) \geq 1$ then $\nu(K) \geq \tau(K) \geq 1$ while
\[ \nu(\mirror{K}) \leq \tau(\mirror{K})+1 = -\tau(K)+1 \leq 0, \]
so $\nu(K) > \nu(\mirror{K})$, and then \eqref{eq:nuhat-nu} gives us $\nuhat(K) \geq 2\nu(K)-1 \geq 1$.  Moreover, if $\nuhat(K)=0$ then $\nuhat(\mirror{K})=0$ as well, so we have just shown that $\tau(K) \leq 0$ and $-\tau(K) = \tau(\mirror{K}) \leq 0$, hence $\tau(K)=0$ as claimed.
\end{proof}

Proposition~\ref{prop:r-nu-properties} can also be proved by repeating arguments from \cite{bs-concordance} nearly verbatim, but applied to $\hfhat(Y)$ rather than $I^\#(Y)$.  These arguments rely only on the fact that $\dim \hfhat(S^3) = 1$, together with the surgery exact triangle and adjunction inequality for $\hfhat$.

We note the following examples for later use.

\begin{lemma} \label{lem:nf-r-nu}
Suppose that $K$ is one of the genus-1 knots appearing in Theorem~\ref{thm:main-hfk} other than $5_2$ and its mirror.  Then
\[ (\rhat(K), \nuhat(K)) = (4,0). \]
We also have $(\rhat(\mirror{5_2}), \nuhat(\mirror{5_2})) = (3,1)$ and $(\rhat(5_2),\nuhat(5_2))=(3,-1)$.
\end{lemma}

\begin{proof}
Proposition~\ref{prop:nf-1-surgery-hat} applies to both $K$ and $\mirror{K}$ to tell us that
\[ \dim \hfhat(S^3_1(K)) = 5 \quad\text{and}\quad \dim \hfhat(S^3_{-1}(K)) = \dim \hfhat(S^3_1(\mirror{K})) = 5. \]
By Proposition~\ref{prop:r-nu}, these can only be equal if $\nuhat(K)=0$, and then
\[ 5 = \dim \hfhat(S^3_1(K)) = 1\cdot \rhat(K) + |1- 0\cdot\nuhat(K)| \]
implies that $\rhat(K) = 4$.

Similarly, we note from Proposition~\ref{prop:1-surgery-5_2} and Lemma~\ref{lem:hfhat-vs-hfred} that
\[ \dim\hfhat(S^3_1(\mirror{5_2})) = 3 \quad\text{and}\quad \dim \hfhat(S^3_{-1}(\mirror{5_2})) = \dim \hfhat(S^3_1(5_2)) = 5. \]
Now Proposition~\ref{prop:r-nu} only tells us that $\nuhat(\mirror{5_2}) \geq 1$, but Proposition~\ref{prop:r-nu-properties} also bounds it above by $1$ and so $\nuhat(\mirror{5_2})=1$ after all.  It now follows immediately that $\rhat(\mirror{5_2})=3$, and similarly for $5_2$.
\end{proof}

\subsection{Almost L-space knots} \label{ssec:almost}

A nontrivial knot $K\subset S^3$ is said to be an \emph{L-space knot} if $S^3_r(K)$ is an L-space for some rational slope $r>0$, meaning that $\dim \hfhat(S^3_r(K)) = |H_1(S^3_r(K);\Z)|$.  This places strong restrictions on $K$.

\begin{theorem}[\cite{osz-lens,ghiggini,ni-hfk,hedden-positivity,osz-rational}] \label{thm:l-space-knots}
If $K$ is an L-space knot, then $K$ is fibered and strongly quasipositive, and $r$-surgery on $K$ is an L-space if and only if $r \geq 2g(K)-1$.
\end{theorem}

\begin{remark} \label{rem:lspace-nu-r}
It follows quickly that a knot $K$ of genus $g \geq 1$ is an L-space knot if and only if $\rhat(K) = \nuhat(K) = 2g-1$.
\end{remark}

In this section, we develop similar restrictions on knots which fall just short of being L-space knots.  We recall the following from Definition~\ref{def:almost-lspace}.

\begin{definition} \label{def:almost}
A knot $K\subset S^3$ is an \emph{almost L-space knot} if
\[ \dim_\Q \hfhat(S^3_n(K)) = n+2 \]
for some $n \geq 2g(K)-1$.
\end{definition}

\begin{lemma} \label{lem:almost-r-nu}
A knot $K\subset S^3$ is an almost L-space knot if and only if $\rhat(K)-\nuhat(K) = 2$.
\end{lemma}

\begin{proof}
We note that $K$ must be nontrivial since all surgeries on the unknot are L-spaces.  Using the inequality
\[ \nuhat(K) \leq \max(2g_4(K)-1,0) \leq 2g(K)-1 \]
of Proposition~\ref{prop:r-nu-properties}, it follows that if $n \geq 2g(K)-1$ then
\[ \dim \hfhat(S^3_n(K)) = \rhat(K) + |n - \nuhat(K)| = n + (\rhat(K)-\nuhat(K)). \]
By assumption the left side is $n+2$ for some such $n$, which proves the lemma.
\end{proof}

\begin{lemma} \label{lem:hfp-almost}
If $K \subset S^3$ is an almost L-space knot of genus $g \geq 1$, then 
\[ \hfhat(S^3_{2g-1}(K),s) \cong \begin{cases} \Q^3 & s=0 \\ \Q & 1 \leq |s| \leq g-1, \end{cases} \]
and similarly there is some $n \geq 1$ such that
\[ \hfp(S^3_{2g-1}(K),s) \cong \begin{cases} \cT^+ \oplus \Q[U]/U^n & s = 0 \\ \cT^+ & 1\leq |s| \leq g-1 \end{cases} \]
as $\Q[U]$-modules.
\end{lemma}

\begin{proof}
Let $Y=S^3_{2g-1}(K)$.  By Lemma~\ref{lem:almost-r-nu} and $\nuhat(K) \leq 2g-1$ we have
\[ \sum_{s\in\Z/(2g-1)\Z} \dim \hfhat(Y,s) = \dim \hfhat(Y) = 2g+1. \]
Each $\hfhat(Y,s)$ has Euler characteristic $1$ \cite[Proposition~5.1]{osz-properties} and hence odd dimension.  Since the total dimension is $2g+1$ there must be a unique $s_0$ with
\[ \dim \hfhat(Y,s_0) = 3 \]
and $\dim \hfhat(Y,s) = 1$ for all other $s \neq s_0$.  But we have
\[ \hfhat(Y,s_0) \cong \hfhat(Y,-s_0) \]
by conjugation symmetry \cite[Theorem~2.4]{osz-properties}, recalling from Remark~\ref{rem:s-conjugate} that $s$ and $-s$ determine conjugate $\spc$ structures, so $-s_0 \equiv s_0 \pmod{2g-1}$ and therefore $s_0 = 0$.

In order to pass from $\hfhat$ to $\hfp$, we use the exact triangle \eqref{eq:hfhat-from-u} to see that if
\[ \hfp(Y,s) \cong \cT^+ \oplus \left( \bigoplus_{i=1}^k \Q[U]/U^{n_i} \right) \]
as $\Q[U]$-modules for some $k \geq 0$ and $n_1,\dots,n_k \geq 1$, then
\[ \dim \hfhat(Y,s) = \dim \coker(U) + \dim \ker(U) = k + (k+1) = 2k+1. \]
From this we conclude that $k=1$ if $s\equiv 0\pmod{2g-1}$ and $k=0$ otherwise, proving the lemma.
\end{proof}

\begin{proposition} \label{prop:almost-fibered}
Let $K$ be an almost L-space knot of genus $g \geq 1$.  Then exactly one of the following must hold.
\begin{itemize}
\item $g=1$, and $K$ is the left-handed trefoil, figure eight, or $\mirror{5_2}$.
\item $g\geq 2$, and $K$ is fibered with $V_{g-1}(K) = 1$.
\end{itemize}
\end{proposition}

\begin{proof}
According to Lemma~\ref{lem:ker-v-g-1} we have
\begin{equation} \label{eq:hfred-hfkhat-vg-1}
\dim \hfred(S^3_{2g-1}(K),g-1) = \dim \hfkhat(K,g) - V_{g-1}(K).
\end{equation}
We also recall from Proposition~\ref{prop:vs-hs} that $V_g(K) = 0$ and $V_g(K) \leq V_{g-1}(K) \leq V_g(K)+1$, so $V_{g-1}(K)$ is either $0$ or $1$.

Now suppose that $g=1$.  In this case, we know by Lemma~\ref{lem:hfp-almost} that
\[ \hfp(S^3_1(K)) \cong \cT^+ \oplus \Q[U]/U^n \]
for some $n \geq 1$, and Lemma~\ref{lem:ker-v-g-1} says that the $U$-action on $\hfred(S^3_1(K)) \cong \Q[U]/U^n$ is trivial, so $n=1$.  Then $\dim \hfred(S^3_1(K),0) = 1$, and \eqref{eq:hfred-hfkhat-vg-1} becomes
\[ \dim \hfkhat(K,1) = \begin{cases} 1, & V_0(K) = 0 \\ 2, & V_0(K) = 1. \end{cases}  \]
Thus if $V_0(K)=0$ then $K$ is fibered \cite{ghiggini}, and the right-handed trefoil is an L-space knot, so $K$ must be the left-handed trefoil or the figure eight instead; and in the remaining cases we have $V_0(K)=1$ and $\dim \hfkhat(K,1) = 2$.  In these cases, Propositions~\ref{prop:1-surgery-5_2} and \ref{prop:nf-1-surgery-plus} tell us that
\begin{equation} \label{eq:v0-nf}
V_0(K) = -\frac{1}{2}d(S^3_1(K)) = \begin{cases} 1, & K \cong \mirror{5_2} \\ 0, & K \not\cong \mirror{5_2}, \end{cases}
\end{equation}
so $K$ must be $\mirror{5_2}$.

From now on we suppose that $g \geq 2$.  Here the $\spc$ structures $0$ and $g-1$ on $S^3_{2g-1}(K)$ are different, so by Lemma~\ref{lem:hfp-almost} we have
\[ \hfred(S^3_{2g-1}(K),g-1) = 0 \]
and so \eqref{eq:hfred-hfkhat-vg-1} becomes $0 = \dim \hfkhat(K,g) - V_{g-1}(K)$.  Thus
\[ \dim \hfkhat(K,g) = V_{g-1}(K) \leq 1. \]
But this dimension must be positive \cite[Theorem~1.2]{osz-genus}, so it is equal to $1$, and then this implies that $K$ is fibered \cite{ni-hfk}.
\end{proof}

\begin{proposition} \label{prop:almost-implies-sqp}
If $K$ is an almost L-space knot of genus $g \geq 2$, then $\tau(K) = g$ and so $K$ is strongly quasipositive.
\end{proposition}

\begin{proof}
Proposition~\ref{prop:almost-fibered} says that $K$ is fibered, and that $V_{g-1}(K) = 1$.  Since $K$ is fibered, it is strongly quasipositive if and only if $\tau(K) = g$ \cite[Theorem~1.2]{hedden-positivity}.  Thus we will suppose that $\tau(K) \leq g-1$ and show that this leads to a contradiction.

The assumption that $\tau(K) \leq g-1$ is equivalent to the assertion that the map
\[ H_*(C\{i=0,j\leq g-1\}) \to H_*(C\{i=0\}) \cong \hfhat(S^3) \cong \Q \]
is surjective.  In this case the short exact sequence of complexes
\[ 0 \to C\{i=0,j\leq g-1\} \to C\{i=0\} \to C\{0,g\} \to 0 \]
gives rise to a long exact sequence in homology which splits as
\[ 0 \to \underbrace{H_{*+1}(C\{0,g\})}_{\cong\hfkhat(K,g)\cong\Q} \to H_*(C\{i=0,j\leq g-1\}) \to \underbrace{H_*(C\{i=0\})}_{\cong\hfhat(S^3)\cong\Q} \to 0, \]
so $H_*(C\{i=0,j\leq g-1\}) \cong \Q^2$.

We now consider the short exact sequence of complexes
\[ 0 \to C\{i<0, j=g-1\} \xrightarrow{\iota} \hat{A}_{g-1} \to C\{i=0, j\leq g-1\} \to 0, \]
whose first term is equal to
\[ C\{-1, g-1\} \cong C\{0,g\} \cong \hfkhat(K,g) \cong \Q. \]
The hat version of the large surgeries formula (Theorem~\ref{thm:large-surgeries}) tells us that
\[ H_*(\hat{A}_{g-1}) \cong \hfhat(S^3_{2g-1}(K),g-1) \cong \Q \]
by Lemma~\ref{lem:hfp-almost}, so we get a long exact sequence 
\[ \dots \to \underbrace{H_*(C\{-1,g-1\})}_{\cong \Q} \xrightarrow{\iota_*} \underbrace{H_*(\hat{A}_{g-1})}_{\cong \Q} \to \underbrace{H_*(C\{i=0,j\leq g-1\})}_{\Q^2} \to \dots, \]
from which the map $\iota_*: H_*(C\{-1,g-1\}) \to H_*(\hat{A}_{g-1})$ is zero.

Finally, the inclusion map $C\{-1,g-1\} \hookrightarrow A^+_{g-1}$ factors through $\iota$ as
\[ C\{-1,g-1\} \xrightarrow{\iota} \hat{A}_{g-1} \hookrightarrow A^+_{g-1}, \]
so the induced map
\[ H_*(C\{-1,g-1\}) \to H_*(A^+_{g-1}) \]
on homology must be zero, since it factors through $\iota_* = 0$.  But this map belongs to the short exact sequence
\[ 0 \to H_*(C\{-1,g-1\}) \to H_*(A^+_{g-1}) \xrightarrow{(v^+_{g-1})_*} H_*(B^+) \]
of Lemma~\ref{lem:ker-v-g-1}, so it must also be injective, and since $H_*(C\{-1,g-1\}) \cong \Q$ is nonzero, we have a contradiction.  Therefore $\tau(K)=g$, as desired.
\end{proof}

\begin{corollary} \label{cor:almost-r-nu}
If $K$ is an almost L-space knot of genus $g \geq 2$, then $(\rhat(K),\nuhat(K)) = (2g+1,2g-1)$.
\end{corollary}

\begin{proof}
Proposition~\ref{prop:almost-implies-sqp} says that $\tau(K)=g$.  The invariant $\nu(K)$ of \cite[Definition~9.1]{osz-rational} is equal to either $\tau(K)$ or $\tau(K)+1$ by \cite[Equation~(34)]{osz-rational}, but it is also at most $g$ by definition, so we have
\[ \nu(K)=g \quad\text{and}\quad \nu(\mirror{K}) \leq \tau(\mirror{K})+1 = -g+1. \]
Since $\nu(K) > \nu(\mirror{K})$, we apply \eqref{eq:nuhat-nu} to get $\nuhat(K)=\max(2\nu(K)-1,0) = 2g-1$.  Then $\rhat(K) = 2g+1$ as well by Lemma~\ref{lem:almost-r-nu}.
\end{proof}

\begin{remark}
Let $K$ be an almost L-space knot of genus $g\geq 2$.  Then Lemma~\ref{lem:hfp-almost} and the large surgeries formula imply that $H_*(\hat{A}_s) \cong \Q$ for all $s\geq 1$, and so one can repeat the proof of \cite[Theorem~1.2]{osz-lens} to show, among other things, that
\[ \dim \hfkhat(K,a) = 0 \text{ or } 1 \quad \text{for all } a \geq 2, \]
hence by symmetry whenever $|a| \geq 2$; the corresponding $t^a$-coefficients of $\Delta_K(t)$ must then be either $0$ or $\pm1$.  We will not pursue this further here.
\end{remark}

We conclude by noting the following consequences, which we will not use in this paper.

\begin{theorem} \label{thm:rhat-3}
We have $\rhat(K) \leq 3$ if and only if $K$ has crossing number at most $5$.
\end{theorem}

\begin{proof}
We replace $K$ with its mirror as needed to ensure that $\nuhat(K) \geq 0$, since this does not change $\rhat(K)$.  Now by Proposition~\ref{prop:r-nu-properties} the difference $\rhat(K)-\nuhat(K)$ is nonnegative and even, and we have
\[ 0\leq \rhat(K)-\nuhat(K) \leq \rhat(K) \leq 3, \]
so it must be either $0$ or $2$.

Supposing that the difference is $2$, then $K$ is an almost L-space knot by Lemma~\ref{lem:almost-r-nu}.  If $g(K) \geq 2$ then Corollary~\ref{cor:almost-r-nu} says that $\rhat(K) = 2g(K)+1 \geq 5$, which cannot happen.  So $g(K) = 1$, and then Proposition~\ref{prop:almost-fibered} says that $K$ is either $T_{-2,3}$, a figure eight, or $\mirror{5_2}$.  

Otherwise we have $\rhat(K) = \nuhat(K)$, so by Remark~\ref{rem:lspace-nu-r}, if $K$ is nontrivial then it must be a nontrivial L-space knot satisfying $\rhat(K) = 2g(K)-1$.  But then $\rhat(K) \leq 3$ implies that $g(K)$ is either $1$ or $2$, so $K$ must be a right-handed trefoil \cite{ghiggini} or a $(2,5)$ torus knot \cite{frw-cinquefoil}.  Up to mirroring we have now accounted for all knots of at most five crossings and ruled out everything else, so this completes the proof.
\end{proof}

\begin{theorem} \label{thm:1-surgery-3d}
If $\dim_\Q \hfhat(S^3_1(K)) = 3$, then $K$ is either the left-handed trefoil, figure eight, or $\mirror{5_2}$.
\end{theorem}

\begin{proof}
Proposition~\ref{prop:r-nu} says that
\begin{equation} \label{eq:dim-1-surgery-3}
3 = \dim_\Q \hfhat(S^3_1(K)) = \rhat(K) + |1-\nuhat(K)|,
\end{equation}
so $\rhat(K) \leq 3$ with equality if and only if $\nuhat(K)=1$.  If $\nuhat(K) > 1$ then we have $3 > \rhat(K) \geq \nuhat(K) > 1$, so $\nuhat(K)=2$ and this contradicts Proposition~\ref{prop:r-nu-properties}.  Thus $\nuhat(K) \leq 1$ and now \eqref{eq:dim-1-surgery-3} becomes $\rhat(K)-\nuhat(K) = 2$.  So $K$ is an almost L-space knot, with genus $1$ by Corollary~\ref{cor:almost-r-nu}, and now Proposition~\ref{prop:almost-fibered} says that it must be one of the knots claimed above.
\end{proof}

\section{The mirror of $5_2$} \label{sec:mirror}

Our goal in this section is to prove that not only are nonnegative slopes characterizing for $\mirror{5_2}$, but in fact the Heegaard Floer homology of such surgeries characterizes $\mirror{5_2}$.

\begin{theorem} \label{thm:m5_2-nonnegative}
Suppose for some rational number $r\geq 0$ and knot $K \subset S^3$ that there is an isomorphism
\[ \hfp(S^3_r(K)) \cong \hfp(S^3_r(\mirror{5_2})) \]
of graded $\Q[U]$-modules.  Then $K$ is isotopic to $\mirror{5_2}$.
\end{theorem}

We recall from Lemma~\ref{lem:nf-r-nu} that $\rhat(\mirror{5_2}) = 3$ and $\nuhat(\mirror{5_2}) = 1$.  Thus if $p$ and $q$ are relatively prime, with $p\neq 0$ and $q>0$, then
\begin{equation} \label{eq:m52-hfhat}
\dim \hfhat(S^3_{p/q}(\mirror{5_2})) = 3q + |p-q| = \begin{cases} p+2q, & p \geq q \\ 4q-p, & p \leq q. \end{cases}
\end{equation}
Throughout this section we will make implicit use of the fact that $\hfp(Y)$ completely determines $\hfhat(Y)$.

\begin{lemma} \label{lem:m52-leq-1}
Suppose that $0 < \frac{p}{q} \leq 1$ and that there is an isomorphism
\[ \hfp(S^3_{p/q}(K)) \cong \hfp(S^3_{p/q}(\mirror{5_2})) \]
of graded $\Q[U]$-modules.  Then $K$ is an almost L-space knot of genus 1.
\end{lemma}

\begin{proof}
By equation \eqref{eq:m52-hfhat} we have
\begin{align*}
4q-p &= q \cdot \rhat(K) + |p - q \nuhat(K)| \\
&= \begin{cases} p + q(\rhat(K)-\nuhat(K)), & \frac{p}{q} \geq \nuhat(K) \\[0.25em] q(\rhat(K)+\nuhat(K))-p, & \frac{p}{q} < \nuhat(K). \end{cases}
\end{align*}
In the case $\frac{p}{q} \leq \nuhat(K)$ this simplifies to $\rhat(K)+\nuhat(K) = 4$, and given that
\[ \rhat(K) \geq \nuhat(K) \geq \frac{p}{q} > 0, \]
Proposition~\ref{prop:r-nu-properties} says that this is only possible if $\rhat(K) = 3$ and $\nuhat(K) = 1$.

Now we suppose instead that $\frac{p}{q} > \nuhat(K)$, and then we have
\[ 4q-p = p + q(\rhat(K)-\nuhat(K)), \]
or
\[ \frac{p}{q} = 2 - \frac{\rhat(K)-\nuhat(K)}{2}. \]
Since $0 < \frac{p}{q} \leq 1$, and $\frac{1}{2}(\rhat(K)-\nuhat(K))$ is a nonnegative integer, it follows that $\frac{p}{q}=1$ and that $\rhat(K)-\nuhat(K)=2$.  But then $\nuhat(K) < \frac{p}{q} = 1$, and $r_0(K) \geq |\nuhat(K)|$ by Proposition~\ref{prop:r-nu-properties}, so $(\rhat(K),\nuhat(K))$ must be either $(2,0)$ or $(1,-1)$.

In all cases we have shown that $K$ is an almost L-space knot and $|\nuhat(K)| \leq 1$.  According to Corollary~\ref{cor:almost-r-nu}, if $g(K) \geq 2$ then $\nuhat(K) = 2g(K)-1 \geq 3$, which is impossible, so in fact $g(K)=1$ and the proof is complete.
\end{proof}

\begin{lemma} \label{lem:m52-geq-1}
Suppose that $\frac{p}{q} > 1$ and that there is an isomorphism
\[ \hfp(S^3_{p/q}(K)) \cong \hfp(S^3_{p/q}(\mirror{5_2})) \]
of graded $\Q[U]$-modules.  Then $K$ is an almost L-space knot of genus 1.
\end{lemma}

\begin{proof}
By equation \eqref{eq:m52-hfhat} we have
\begin{align*}
p+2q &= q \cdot \rhat(K) + |p - q \nuhat(K)| \\
&= \begin{cases} p + q(\rhat(K)-\nuhat(K)), & \frac{p}{q} \geq \nuhat(K) \\[0.25em] q(\rhat(K)+\nuhat(K))-p, & \frac{p}{q} \leq \nuhat(K). \end{cases}
\end{align*}
Now if $\frac{p}{q} \geq \nuhat(K)$ then this immediately reduces to
\[ \rhat(K) - \nuhat(K) = 2, \]
so $K$ is an almost L-space knot by Lemma~\ref{lem:almost-r-nu}.

In the remaining case we have $\nuhat(K) > \frac{p}{q} > 1$, and so the above equation becomes
\[ p+2q = q(\rhat(K)+\nuhat(K)) - p, \]
or equivalently
\begin{equation} \label{eq:m52-p-q-geq-1}
\frac{p}{q} = \frac{\rhat(K) + \nuhat(K)}{2} - 1.
\end{equation}
Now we combine this with $\frac{p}{q} < \nuhat(K)$ and rearrange to get
\[ \rhat(K) - 2 < \nuhat(K), \]
and then by Proposition~\ref{prop:r-nu-properties} it follows that $\rhat(K) = \nuhat(K)$ and so $K$ is an L-space knot.  Remark~\ref{rem:lspace-nu-r} says that $\rhat(K)=\nuhat(K)=2g(K)-1$, so in fact \eqref{eq:m52-p-q-geq-1} becomes
\[ \frac{p}{q} = 2g(K)-2. \]
By the assumption $\frac{p}{q} > 1$ it follows that $g(K) \geq 2$.

Now in either case, if we suppose that $g(K) = g \geq 2$, then we have $V_{g-1}(K) = 1$.  Indeed, if $K$ is an almost L-space knot then this is part of Proposition~\ref{prop:almost-fibered}.  If instead $K$ is an L-space knot then it is strongly quasipositive by Theorem~\ref{thm:l-space-knots}, so the invariant $\nu^+(K)$ of \cite{hom-wu} is equal to $g(K)$ by \cite[Proposition~3]{hom-wu}; this is by definition the least $s$ such that $V_s(K)=0$, so in particular $V_{g-1}(K) = 1$ as claimed.  Either way, we have $V_1(K) \geq 1$ by Proposition~\ref{prop:vs-hs}.  But then Proposition~\ref{prop:compare-gradings} says that if $\frac{p}{q} > 1$ and \[\hfp(S^3_{p/q}(\mirror{5_2})) \cong \hfp(S^3_{p/q}(K))\] then $V_s(K)=0$ for all $s \geq 1$, so this is a contradiction.  Thus $g=1$.

We conclude that $K$ cannot be an L-space knot, since that would have implied that $g(K) \geq 2$, and so $K$ must be an almost L-space knot of genus 1 after all.
\end{proof}

Combining the above lemmas yields the following.

\begin{proposition} \label{prop:m52-positive}
Suppose that $\frac{p}{q} > 0$ and that there is an isomorphism
\[ \hfp(S^3_{p/q}(K)) \cong \hfp(S^3_{p/q}(\mirror{5_2})) \]
of graded $\Q[U]$-modules.  Then $K$ is isotopic to $\mirror{5_2}$.
\end{proposition}

\begin{proof}
We know that $K$ is an almost L-space knot of genus 1, by Lemma~\ref{lem:m52-leq-1} if $0 < \frac{p}{q} \leq 1$ and by Lemma~\ref{lem:m52-geq-1} if $\frac{p}{q} > 1$.  Then its Alexander polynomial must have the form
\[ \Delta_K(t) = at + (1-2a) + at^{-1} \]
for some $a\in\Z$.  We have $\Delta''_K(1) = 2a$, whereas $\Delta''_{\mirror{5_2}}(t) = 4$, so $a=2$ by Proposition~\ref{prop:compare-gradings}.  This proves that
\[ \Delta_K(t) = \Delta_{\mirror{5_2}}(t) = 2t-3+2t^{-1}. \]
But none of the genus-1 knots in Proposition~\ref{prop:almost-fibered} have this Alexander polynomial except for $\mirror{5_2}$ itself, so $K \cong \mirror{5_2}$.
\end{proof}

We can also handle zero-surgery by a somewhat different argument.

\begin{proposition} \label{prop:zero-surgery}
Suppose for some knot $K \subset S^3$ that there is an isomorphism
\[ \hfp(S^3_0(K)) \cong \hfp(S^3_0(\mirror{5_2})) \]
of graded $\Q[U]$-modules.  Then $K \cong \mirror{5_2}$.  Similarly, if we have an isomorphism
\[ \hfp(S^3_0(K)) \cong \hfp(S^3_0(5_2)) \]
then $K \cong 5_2$.
\end{proposition}

\begin{proof}
We show first that $g(K) \leq 1$.  Supposing instead that $K$ has genus $g \geq 2$, there is a non-torsion $\spc$ structure $\spinc_{g-1}$ for which $\hfp(S^3_0(K),\spinc_{g-1}) \neq 0$, namely the one specified by $\langle c_1(\spinc_{g-1}), [\hat\Sigma]\rangle = 2g-2$ for a capped-off Seifert surface $\hat\Sigma$, by the isomorphism $\hfp(S^3_0(K),\spinc_{g-1}) \cong \hfkhat(K,g)$ of \cite[Corollary~4.5]{osz-knot} together with the fact that $\hfkhat$ detects genus \cite[Theorem~1.2]{osz-genus}.
On the other hand, since $5_2$ and its mirror both have genus $1$, we have 
\[ \hfp(S^3_0(5_2),\spinc) \cong \hfp(S^3_0(\mirror{5_2}),\spinc) \cong 0 \]
in all non-torsion $\spc$ structures, by the adjunction inequality \cite[Theorem~7.1]{osz-properties}.  Thus $g\leq 1$ as claimed.

Next, we recall from Lemma~\ref{lem:nf-r-nu} that $(\rhat(\mirror{5_2}),\nuhat(\mirror{5_2})) = (3,1)$, so we have
\begin{align*}
\dim \hfhat(S^3_1(\mirror{5_2})) &= 3, &
\dim \hfhat(S^3_{-1}(\mirror{5_2})) &= 5
\end{align*}
and so $\dim \hfhat(S^3_0(\mirror{5_2})) = 4$ by the surgery exact triangle for $\hfhat$, since it differs by 1 from each of these other dimensions.  We also have $\dim \hfhat(S^3_0(5_2))=4$ by the same argument, so in either case $\dim \hfhat(S^3_0(K)) = 4$, and then
\[ \dim \hfhat(S^3_1(K)) = 3 \text{ or } 5 \]
again by the surgery exact triangle.  This means that $K$ cannot be unknotted, so $g(K)=1$.  We apply Corollary~\ref{cor:hfhat-vs-hfkhat} to get
\[ \dim \hfkhat(K,1) - V_0(K) = \frac{\dim\hfhat(S^3_1(K)) - 1}{2} = 1 \text{ or } 2, \]
and $g(K) = 1$ implies that $0 \leq V_0(K) \leq 1$ by Proposition~\ref{prop:vs-hs}, hence $\dim\hfkhat(K,1) \leq 3$.

Now we use the fact that $\hfp(S^3_0(K))$ determines the Alexander polynomial $\Delta_K(t)$, by \cite[Proposition~10.14]{osz-properties} and \cite[Theorem~10.17]{osz-properties}, to see that
\[ \Delta_K(t) = \Delta_{5_2}(t) = 2t-3+2t^{-1}. \]
But the linear coefficient $2$ is equal to the Euler characteristic $\chi(\hfkhat(K,1))$, so in particular $\dim \hfkhat(K,1)$ must be even.  It follows from the above bound that
\[ \dim \hfkhat(K,1) = 2 \]
and so $K$ must be one of the knots listed in Theorem~\ref{thm:main-hfk}.

Finally, we can read the correction terms $d_{\pm1/2}(S^3_0(K))$ off of $\hfp(S^3_0(K))$, since they are defined as the grading of the bottom-most element of a tower $\cT^+$ in grading $\pm\frac{1}{2} \pmod{2}$.  According to \cite[Proposition~4.12]{osz-absolutely}, these are determined by the formulas
\begin{align*}
d_{1/2}(S^3_0(K)) &= d(S^3_1(K)) + \tfrac{1}{2}, \\
d_{-1/2}(S^3_0(K)) &= d(S^3_{-1}(K)) - \tfrac{1}{2} = d(-S^3_1(\mirror{K})) - \tfrac{1}{2} \\
&= -d(S^3_1(\mirror{K})) - \tfrac{1}{2}.
\end{align*}
Now Theorem~\ref{thm:ni-wu-d} tells us that
\begin{align*}
d_{1/2}(S^3_0(K)) &= -2V_0(K) + \tfrac{1}{2}, &
d_{-1/2}(S^3_0(K)) &= 2V_0(\mirror{K}) - \tfrac{1}{2}
\end{align*}
and so $\hfp(S^3_0(K))$ determines both $V_0(K)$ and $V_0(\mirror{K})$.  But we saw in \eqref{eq:v0-nf} that if $K$ is one of the knots in Theorem~\ref{thm:main-hfk}, then
\[ (V_0(K),V_0(\mirror{K})) = \begin{cases} (1,0), & K \cong \mirror{5_2} \\ (0,1), & K \cong 5_2 \\ (0,0), & \text{otherwise}, \end{cases} \]
so $\hfp(S^3_0(\mirror{5_2}))$ and $\hfp(S^3_0(5_2))$ are different from each other and from each of the invariants $\hfp(S^3_0(K))$ where $K$ is another of the knots in Theorem~\ref{thm:main-hfk}.  This completes the proof.
\end{proof}

Combining Proposition~\ref{prop:m52-positive} in the case $r>0$ and Proposition~\ref{prop:zero-surgery} for $r=0$, this completes the proof of Theorem~\ref{thm:m5_2-nonnegative}. \hfill\qed

\section{The knot $5_2$} \label{sec:5_2}

In this section we start to consider whether positive slopes are characterizing slopes for $5_2$.  We will achieve partial results in this direction without using the mapping cone formula (Theorem~\ref{thm:mapping-cone}), which we then apply in Section~\ref{sec:5_2-mapping-cone}.

\begin{lemma} \label{lem:5_2-V_0}
Suppose that there is some knot $K$ and some rational $r>0$ such that \[ \hfp(S^3_r(5_2)) \cong \hfp(S^3_r(K)) \] as graded $\Q[U]$-modules.  Then $V_s(K) = 0$ for all $s \geq 0$.  In addition, if $g(K) =1$ then $\Delta_K(t)=\Delta_{5_2}(t) = 2t-3+2t^{-1}$.
\end{lemma}

\begin{proof}
We recall from Proposition~\ref{prop:1-surgery-5_2} that
\[ V_0(5_2) = -\frac{1}{2}d(S^3_1(5_2)) = 0, \]
and then Propositions~\ref{prop:compare-gradings} and \ref{prop:vs-hs} say that $V_0(K)=0$ and that the sequence of $V_s(K)$ is nonincreasing, proving the first claim.  The second claim also follows from Proposition~\ref{prop:compare-gradings}, once we use $g(K)=1$ to write $\Delta_K(t) = at + (1-2a) + at^{-1}$ for some $a$ and then observe that
\[ a = \frac{\Delta''_K(1)}{2} = \frac{\Delta''_{5_2}(1)}{2} = 2. \qedhere \]
\end{proof}

\begin{lemma} \label{lem:5_2-r-minus-nu}
Suppose for some knot $K\not\cong 5_2$ and some rational $r>0$ that \[ \hfp(S^3_r(5_2)) \cong \hfp(S^3_r(K)) \] as graded $\Q[U]$-modules.  Then $\rhat(K)=4$ and $\nuhat(K)=0$.
\end{lemma}

\begin{proof}
Write $r = \frac{p}{q}$ for some coprime $p,q > 0$.  We note that since $\frac{p}{q} > 0 > \nuhat(5_2)$, we have
\begin{align*}
\dim \hfhat(S^3_{p/q}(5_2)) &= q \cdot \rhat(5_2) + |p - q\nuhat(5_2)| \\
&= 3q + |p+q| = p+4q,
\end{align*}
and by hypothesis this is equal to $\dim \hfhat(S^3_{p/q}(K))$.

We next observe that $\nuhat(K) \leq 0$: according to Proposition~\ref{prop:r-nu-properties}, it is enough to show that $V_0(K) = 0$, and this was already proved in Lemma~\ref{lem:5_2-V_0}.  Thus $\frac{p}{q} > \nuhat(K)$, and we have
\begin{align*}
\dim \hfhat(S^3_{p/q}(K)) &= q\cdot \rhat(K) + (p-q\nuhat(K)) \\
&= p + q(\rhat(K)-\nuhat(K)).
\end{align*}
This is equal to $\dim \hfhat(S^3_{p/q}(5_2)) = p+4q$, so we must have $\rhat(K)-\nuhat(K)=4$.

Now since $0 \leq \rhat(K) = \nuhat(K)+4 \leq 4$ and $\rhat(K) \geq |\nuhat(K)|$, the only possibilities for these invariants are
\[ (\rhat(K),\nuhat(K)) = (4,0) \text{ or } (3,-1) \text{ or } (2,-2), \]
and the last is impossible because Proposition~\ref{prop:r-nu-properties} says that $\nuhat(K)$ must be $0$ or odd.  If $(\rhat(K),\nuhat(K))=(3,-1)$ then $(\rhat(\mirror{K}), \nuhat(\mirror{K})) = (3,1)$, so $\mirror{K}$ is an almost L-space knot by Lemma~\ref{lem:almost-r-nu}, and then it must have genus 1 by Corollary~\ref{cor:almost-r-nu}.  Now Proposition~\ref{prop:almost-fibered} says that either $\mirror{K} \cong \mirror{5_2}$, or $\Delta_K(t) = \Delta_{\mirror{K}}(t)$ is different from $\Delta_{5_2}(t)$.  But the first option is ruled out by the assumption $K \not\cong 5_2$, and the second by Lemma~\ref{lem:5_2-V_0}.  We conclude that $(\rhat(K),\nuhat(K))$ cannot be $(3,-1)$, and so the only remaining possibility is $(4,0)$.
\end{proof}

\begin{proposition} \label{prop:5_2-fibered}
Suppose for some rational $r > 0$ and some knot $K \not\cong 5_2$ that
\[ \hfp(S^3_r(K)) \cong \hfp(S^3_r(5_2)) \]
as graded $\Q[U]$-modules.  Then $\tau(K)=0$, and the following must hold.
\begin{itemize}
\item If $g(K)=1$ then $K$ is either $15n_{43522}$ or $\Wh^-(T_{2,3},2)$, up to mirroring.
\item If $g(K) \geq 2$ then $K$ is fibered, and
\[ H_*(A^+_s(K)) \cong \begin{cases} \cT^+ \oplus \Q, & |s|=g(K)-1 \\ \cT^+, & \text{otherwise} \end{cases} \]
for all $|s| \leq g(K)-1$.  In this case the maps
\[ v^+_s: A^+_s(K) \to B^+(K) \quad\text{and}\quad h^+_{-s}: A^+_{-s}(K) \to B^+(K) \]
are quasi-isomorphisms for $0 \leq s \leq g(K)-2$.
\end{itemize}
\end{proposition}

\begin{proof}
Let $g=g(K)$.  Lemma~\ref{lem:5_2-r-minus-nu} tells us that $\rhat(K)=4$ and $\nuhat(K)=0$, so $\tau(K)=0$ by Proposition~\ref{prop:r-nu-properties}, and we also have
\[ \dim \hfhat(S^3_{2g-1}(K)) = 4 + |(2g-1)-0| = 2g+3. \]
Lemma~\ref{lem:5_2-V_0} says that $V_{g-1}(K) = 0$, so Corollary~\ref{cor:hfhat-vs-hfkhat} becomes
\begin{equation} \label{eq:5_2-hfkhat-top}
\dim \hfkhat(K,g) = \frac{\dim \hfhat(S^3_{2g-1}(K),g-1) - 1}{2}.
\end{equation}
We will use this to bound $\dim \hfkhat(K,g)$ from above.

We suppose first that $g=1$.  In this case we have
\[ \dim \hfhat(S^3_1(K),0) = \dim \hfhat(S^3_1(K)) = 2g+3 = 5, \]
so \eqref{eq:5_2-hfkhat-top} becomes $\dim \hfkhat(K,1) = 2$.  From Lemma~\ref{lem:5_2-V_0} we have $\Delta_K(t) = 2t-3+2t^{-1}$, so Theorem~\ref{thm:main-hfk} now tells us that $K$ must be one of $5_2$, $15n_{43522}$, or $\Wh^-(T_{2,3},2)$ up to mirroring.  But we have assumed that $K$ is not $5_2$, and it cannot be $\mirror{5_2}$ since $V_0(\mirror{5_2})=1$, so this leaves only the knots named in the proposition.

Now we suppose instead that $g \geq 2$.  In this case, the unique self-conjugate element of
\[ \spc(S^3_{2g-1}(K)) \cong \Z/(2g-1)\Z \]
is identified with $0$, and in particular it is different from $g-1$, which is conjugate to $1-g$.  Since $\dim \hfkhat(S^3_{2g-1}(K),s)$ is odd for all $s$, we use the conjugation symmetry of $\hfhat$ (see Remark~\ref{rem:s-conjugate}) to show that
\begin{align*}
2g+3 &= \dim \hfhat(S^3_{2g-1}(K)) \\
&= \sum_{s\in\Z/(2g-1)\Z} \dim \hfhat(S^3_{2g-1}(K),s) \\
&= 2\dim \hfhat(S^3_{2g-1}(K),g-1) + \sum_{|s| \leq g-2} \dim \hfhat(S^3_{2g-1}(K),s) \\
&\geq 2 \dim \hfhat(S^3_{2g-1}(K),g-1) + (2g-3),
\end{align*}
since there are $2g-3$ different summands on the right.  This shows that
\[ \dim \hfhat(S^3_{2g-1}(K), g-1) \leq 3, \]
and then \eqref{eq:5_2-hfkhat-top} becomes $\dim \hfkhat(K,g) \leq 1$.  But $\dim \hfkhat(K,g)$ must be positive, so equality holds, which implies that
\begin{itemize}
\item $\dim \hfkhat(K,g) = 1$, and then $K$ must be fibered \cite{ni-hfk}; and
\item $\dim \hfhat(S^3_{2g-1}(K),s)$ is $3$ if $s \equiv \pm(g-1) \pmod{2g-1}$, and $1$ otherwise.
\end{itemize}
Applying Lemmas~\ref{lem:ker-v-g-1} and \ref{lem:hfhat-vs-hfred}, we conclude that
\[ \hfred(S^3_{2g-1}(K),s) \cong \begin{cases} \Q, & s=\pm(g-1) \\ 0, & 2-g \leq s \leq g-2. \end{cases} \]
The large surgery formula (Theorem~\ref{thm:large-surgeries}) says that
\[ H_*(A^+_s) \cong \hfp(S^3_{2g-1}(K),s) \]
whenever $|s| \leq g-1$, so this completes the description of $H_*(A^+_s)$.

Now if $0 \leq s \leq g-2$ then $v^+_s: A^+_s \to B^+$ induces a map on homology of the form
\[ (v^+_s)_*: \cT^+ \cong H_*(A^+_s) \to H_*(B^+) \cong \cT^+, \]
and this map is multiplication by $U^{V_s(K)}$, but Lemma~\ref{lem:5_2-V_0} says that $V_s(K) = 0$ and so $(v^+_s)_*$ is an isomorphism.  The map $(h^+_{-s})_*$ has the same form and is identified with multiplication by $U^{H_{-s}(K)}$, but Proposition~\ref{prop:vs-hs} says that $H_{-s}(K) = V_s(K) = 0$, so $(h^+_{-s})_*$ is an isomorphism as well.
\end{proof}

\subsection{$\hfkhat$ in the higher genus case} \label{ssec:5_2-fibered-hfk}

Suppose that we have a homeomorphism
\[ S^3_r(K) \cong S^3_r(5_2) \]
for some slope $r>0$ and some knot $K$ of genus $g \geq 2$.  Then Proposition~\ref{prop:5_2-fibered} says that $K$ is fibered, that $\tau(K)=0$, and that
\[ H_*(A^+_s) \cong \begin{cases} \cT^+ \oplus \Q, & |s|=g-1 \\ \cT^+, & \text{otherwise}. \end{cases} \]
In addition, Lemma~\ref{lem:5_2-V_0} together with Proposition~\ref{prop:vs-hs} tells us that
\[ V_s(K) = \begin{cases} 0, & s \geq 0 \\ |s|, & s < 0 \end{cases} \]
for all $s\in\Z$.  We will use all of this information to determine $\hfkhat(K)$ as a bigraded vector space.

\begin{lemma} \label{lem:As-tower-grading}
There is some integer $d\in\Z$ such that
\[ H_*(A^+_s) \cong \begin{cases}
\cT^+_{(0)} \oplus \Q^{\vphantom{+}}_{(d)}, & s=g-1 \\
\cT^+_{(2-2g)} \oplus \Q^{\vphantom{+}}_{(d+2-2g)}, & s=1-g \\
\cT^+_{(\min(0,2s))}, & \text{otherwise}.
\end{cases} \]
\end{lemma}

\begin{proof}
We consider each of the maps
\[ (v^+_s)_*: H_*(A^+_s) \to H_*(B^+) \cong \cT^+_{(0)}, \]
which are induced by projections at the chain level.  For $s\geq 0$ we have $V_s(K)=0$, so these maps restrict to graded isomorphisms on the towers $\cT^+ \subset H_*(A^+_s)$; thus these towers have their bottom-most elements in grading $0$.  By contrast, for $s < 0$ the maps $(v^+_s)_*$ are modeled on multiplication by $U^{V_s(K)} = U^{|s|}$, so the element of $\cT^+ \subset H_*(A^+_s)$ in grading $0$ is at height $|s|$ in the tower, meaning that the bottom element has grading $-2|s| = 2s$.

Having determined the grading on each tower, we set $d$ equal to the grading of the $\Q$ summand of $H_*(A^+_{g-1})$.  Then it only remains to identify the grading on the $\Q$ summand of $H_*(A^+_{1-g})$.  We apply the large surgery formula, Theorem~\ref{thm:large-surgeries}, to get relatively graded isomorphisms
\begin{align*}
\hfp(S^3_{2g-1}(K),g-1) &\cong H_*(A^+_{g-1}), \\
\hfp(S^3_{2g-1}(K),1-g) & \cong H_*(A^+_{1-g}).
\end{align*}
By conjugation symmetry these $\hfp$ invariants are isomorphic to each other, so we also have a relatively graded isomorphism
\[ H_*(A^+_{g-1}) \cong H_*(A^+_{1-g}). \]
But this means that the grading of the $\Q$ summand of $H_*(A^+_{1-g})$ must be $d$ greater than that of the bottom of the tower $\cT^+_{(2-2g)}$, so its grading is $d+2-2g$ as claimed.
\end{proof}

We now start with the top-most Alexander grading of $\hfkhat(K)$, which we already know to be 1-dimensional because $K$ is fibered.

\begin{lemma} \label{lem:hfk-top-grading}
We have $\hfkhat(K,g) \cong \Q_{(d+2)}$ and $\hfkhat(K,-g) \cong \Q_{(d+2-2g)}$, where $d$ is the integer from Lemma~\ref{lem:As-tower-grading}.
\end{lemma}

\begin{proof}
Lemma~\ref{lem:ker-v-g-1} gives us a short exact sequence
\[ 0 \to \hfkhat_{*+2}(K,g) \to \cT^+_{(0)}\oplus \Q^{\vphantom{+}}_{(d)} \xrightarrow{(v^+_{g-1})_*} \cT^+_{(0)} \to 0, \]
where $(v^+_{g-1})_*$ has kernel $\Q_{(d)}$.  The grading on $\hfkhat(K,-g)$ now comes from the symmetry
\[ \hfkhat_m(K,a) \cong \hfkhat_{m-2a}(K,-a) \]
of \cite[Equation~(3)]{osz-knot}.
\end{proof}

Throughout the remainder of this section we write
\[ \cF_s = C\{i=0,j\leq s\} \]
to denote the filtration
\[ 0 \subset \cF_{-g} \subset \cF_{1-g} \subset \dots \subset \cF_{g} \]
of $\cfhat(S^3)$ whose associated graded groups are the various $\hfkhat(K,a)$.  In particular the short exact sequence
\[ 0 \to \cF_{s-1} \hookrightarrow \cF_s \to C\{0,s\} \to 0 \]
of chain complexes gives rise to a long exact sequence
\begin{equation} \label{eq:ses-F_s}
\cdots \to H_*(\cF_{s-1}) \to H_*(\cF_s) \to \hfkhat_*(K,s) \to H_{*-1}(\cF_{s-1}) \to \cdots .
\end{equation}

\begin{lemma} \label{lem:pi_s-exact-sequence}
For all $s\in\Z$, there is a long exact sequence
\[ \cdots \to H_{*-(2s-2)}(\cF_{-s}) \to H_*(A^+_{s-1}) \xrightarrow{(\pi_s)_*} H_*(A^+_s) \to H_{*-(2s-1)}(\cF_{-s}) \to \cdots , \]
and $(v^+_{s-1})_*$ is equal to the composition
\[ H_*(A^+_{s-1}) \xrightarrow{(\pi_s)_*} H_*(A^+_s) \xrightarrow{(v^+_s)_*} H_*(B^+). \]
\end{lemma}

\begin{proof}
There is a short exact sequence of chain complexes
\begin{equation} \label{eq:A_s-projection}
0 \to C\{i\leq-1,j=s-1\} \to A^+_{s-1} \xrightarrow{\pi_s} A^+_{s} \to 0
\end{equation}
in which $\pi_s$ is projection.  Then $v^+_{s-1} = v^+_s \circ \pi_s$ at the chain level, hence $(v^+_{s-1})_*$ factors as claimed.  We also have a chain homotopy equivalence
\begin{align*}
C_*\{i\leq-1,j=s-1\} &\xrightarrow{U^{s-1}} C_{*-(2s-2)}\{i\leq-s,j=0\} \\
&\xrightarrow{\hphantom{U^{s-1}}} C_{*-(2s-2)}\{i=0,j\leq-s\}
\end{align*}
so the long exact sequence of homology groups associated to \eqref{eq:A_s-projection} takes the form promised by the lemma.
\end{proof}

\begin{lemma} \label{lem:homology-F_0}
We have
\[ H_*(\cF_0) \cong \begin{cases} \Q_{(0)}, & g \geq 3 \\ \Q_{(0)} \oplus \Q_{(d)}, & g = 2. \end{cases} \]
\end{lemma}

\begin{proof}
We apply Lemma~\ref{lem:pi_s-exact-sequence} with $s=0$: supposing for now that $g\geq 3$, the composition
\[ \begin{tikzcd}
H_*(A^+_{-1}) \ar[r,"(\pi_0)_*"] \ar[d,"\cong"] &
H_*(A^+_0) \ar[r,"(v^+_0)_*"] \ar[d,"\cong"] &
H_*(B^+) \ar[d,"\cong"] \\
\cT^+_{(-2)} \ar[r,"(\pi_0)_*"] &
\cT^+_{(0)} \ar[r,"U^{V_0(K)}=1"] &
\cT^+_{(0)}
\end{tikzcd} \]
is equal to $(v^+_{-1})_*$ and hence identified with multiplication by $U^{V_{-1}(K)} = U$.  In particular, the map $(\pi_0)_*$ is surjective and also identified with multiplication by $U$, so the long exact sequence of Lemma~\ref{lem:pi_s-exact-sequence} splits as
\[ 0 \to H_{i+2}(\cF_0) \to H_i(A^+_{-1}) \xrightarrow{(\pi_0)_*} H_i(A^+_0) \to 0 \]
for each $i$, and we have
\[ H_{i+2}(\cF_0) \cong \ker((\pi_0)_*) \cong \begin{cases} \Q, & i = -2 \\ 0, & \text{otherwise} \end{cases} \]
since $-2$ is the grading of the bottom-most element of $H_*(A^+_{-1}) \cong \cT^+_{(-2)}$.

Now suppose that $g=2$.  Then we factor $(v^+_{-1})_*$ as
\[ \begin{tikzcd}
H_*(A^+_{-1}) \ar[r,"(\pi_0)_*"] \ar[d,"\cong"] &
H_*(A^+_0) \ar[r,"(v^+_0)_*"] \ar[d,"\cong"] &
H_*(B^+) \ar[d,"\cong"] \\
\cT^+_{(-2)}\oplus\Q^{\vphantom{+}}_{(d-2)} \ar[r,"(\pi_0)_*"] &
\cT^+_{(0)} \ar[r,"U^{V_0(K)}=1"] &
\cT^+_{(0)}
\end{tikzcd} \]
where the gradings on $H_*(A^+_{-1})$ come from Lemma~\ref{lem:As-tower-grading}.  In this case $(\pi_0)_*$ is still surjective, so once again we identify its kernel $\Q_{(-2)} \oplus \Q_{(d-2)}$ with $H_{*+2}(\cF_0)$.
\end{proof}

\begin{proposition} \label{prop:hfk-nonpositive}
We have $\hfkhat(K,-g) \cong \Q_{(d+2-2g)}$, and $\hfkhat(K,1-g) \cong \Q^2_{(d+3-2g)}$.  If $g \geq 3$ then 
\[ \hfkhat(K,s) \cong \begin{cases} \Q_{(d+4-2g)}, & s=2-g \\ 0, & 3-g \leq s \leq -1 \\ \Q_{(0)}, & s=0. \end{cases} \]
If $g=2$ instead, then $\hfkhat(K,0) \cong \Q^{\vphantom{2}}_{(0)} \oplus \Q^2_{(d)}$.
\end{proposition}

\begin{proof}
The computation of $\hfkhat(K,-g)$ is Lemma~\ref{lem:hfk-top-grading}.  When $s=g-1$, we can factor $(v^+_{g-2})_*$ as
\[ \begin{tikzcd}
H_*(A^+_{g-2}) \ar[r,"(\pi_{g-1})_*"] \ar[d,"\cong"] &
H_*(A^+_{g-1}) \ar[r,"(v^+_{g-1})_*"] \ar[d,"\cong"] &
H_*(B^+) \ar[d,"\cong"] \\
\cT^+_{(0)} \ar[r,"(\pi_{g-1})_*"] &
\cT^+_{(0)} \oplus \Q_{(d)}^{\vphantom{+}} \ar[r,"U^{V_s(K)}=1"] &
\cT^+_{(0)},
\end{tikzcd} \]
and the composition is an isomorphism $\cT^+_{(0)} \to \cT^+_{(0)}$ since $V_{g-2}(K)=0$.  Thus $(\pi_{g-1})_*$ is injective, with cokernel $\Q_{(d)}$.  Now the sequence of Lemma~\ref{lem:pi_s-exact-sequence} splits as
\[ 0 \to H_*(A^+_{g-2}) \xrightarrow{(\pi_{g-1})_*} H_*(A^+_{g-1}) \to H_{*-(2g-3)}(\cF_{1-g}) \to 0, \]
so we have $H_*(\cF_{1-g}) \cong \Q_{(d-(2g-3))}$.  But we also know that
\[ H_*(\cF_{-g}) \cong \hfkhat(K,-g) \cong \Q_{(d+2-2g)} \]
by Lemma~\ref{lem:hfk-top-grading}, so the induced map $H_*(\cF_{-g}) \to H_*(\cF_{1-g})$ must be zero for grading reasons.  Thus when $s=1-g$ the exact sequence \eqref{eq:ses-F_s} splits and we have
\[ \hfkhat_*(K,1-g) \cong H_*(\cF_{1-g}) \oplus H_{*-1}(\cF_{-g}) \cong \Q^2_{(d+3-2g)}. \]

Now if $g \geq 3$ then we consider the map $(v^+_{s-1})_*$ for each of $s=1,2,\dots,g-2$ in turn.  In each case $(v^+_{s-1})_*$ factors as
\[ \begin{tikzcd}
H_*(A^+_{s-1}) \ar[r,"(\pi_s)_*"] \ar[d,"\cong"] &
H_*(A^+_s) \ar[r,"(v^+_s)_*"] \ar[d,"\cong"] &
H_*(B^+) \ar[d,"\cong"] \\
\cT^+_{(0)} \ar[r,"(\pi_s)_*"] &
\cT^+_{(0)} \ar[r,"U^{V_s(K)}=1"] &
\cT^+_{(0)},
\end{tikzcd} \]
and is an isomorphism, since it is identified with multiplication by $U^{V_{s-1}(K)} = 1$ as a map $\cT^+_{(0)} \to \cT^+_{(0)}$.  It follows that each $(\pi_s)_*$ is an isomorphism, so the exact sequence of Lemma~\ref{lem:pi_s-exact-sequence} tells us that
\[ H_*(\cF_{-s}) = 0, \quad s=1,2,\dots,g-2. \]
Applying the long exact sequence \eqref{eq:ses-F_s} for $s=3-g,4-g,\dots,0$, we know that $H_*(\cF_{s-1}) = 0$ for each $s$, and so
\[ \hfkhat_*(K,s) \cong H_*(\cF_{s}) \cong \begin{cases} \Q_{(0)}, & s=0 \\ 0, & 3-g \leq s \leq -1, \end{cases} \]
the case $s=0$ having been computed in Lemma~\ref{lem:homology-F_0}.

Similarly, if we take $s=2-g$ in \eqref{eq:ses-F_s} then we get a long exact sequence
\begin{equation} \label{eq:ses-F_s-2-g}
\cdots \to H_*(\cF_{1-g}) \to H_*(\cF_{2-g}) \to \hfkhat_*(K,2-g) \to H_{*-1}(\cF_{1-g}) \to \cdots.
\end{equation}
For $g \geq 3$ we have seen that $H_*(\cF_{2-g}) = 0$, and so
\[ \hfkhat_*(K,2-g) \cong H_{*-1}(\cF_{1-g}) \cong \Q_{(d+4-2g)}. \]
If $g=2$ instead, then we have computed above that
\[ H_*(\cF_{-1}) = H_*(\cF_{1-g}) \cong \Q_{(d-1)} \]
while $H_*(\cF_0) \cong \Q_{(0)} \oplus \Q_{(d)}$ by Lemma~\ref{lem:homology-F_0}, so it remains to be seen whether the map $\iota: H_*(\cF_{-1}) \to H_*(\cF_{0})$ is zero or not.

Assuming that $g=2$, we now consider the inclusion-induced maps
\[ \begin{tikzcd}
H_*(\cF_{-1}) \ar[r,"\iota"] \ar[d,"\cong"] &
H_*(\cF_0) \ar[r] \ar[d,"\cong"] &
\hfhat(S^3) \ar[d,"\cong"] \\
\Q_{(d-1)} \ar[r] &
\Q_{(0)} \oplus \Q_{(d)} \ar[r] &
\Q_{(0)},
\end{tikzcd} \]
where $H_*(\cF_{-1}) = H_*(\cF_{1-g})$ was computed above, and we used Lemma~\ref{lem:homology-F_0} to identify $H_*(\cF_0)$.  If $\iota$ is nonzero then for degree reasons we must have $d=1$, and then its image is the $\Q_{(0)}$ summand of $H_*(\cF_0)$.  But the map $H_*(\cF_0) \to \hfhat(S^3)$ is surjective since $\tau(K) \leq 0$, so it must be nonzero on this $\Q_{(0)}$ summand, in which case the composition across the top row is also surjective.  This would in turn imply that $\tau(K) \leq -1$, contradicting Proposition~\ref{prop:5_2-fibered}.  We conclude that $\iota = 0$, so \eqref{eq:ses-F_s-2-g} splits as
\[ 0 \to H_*(\cF_0) \to \hfkhat_*(K,0) \to H_{*-1}(\cF_{-1}) \to 0. \]
Thus $\hfkhat_*(K,0) \cong \Q_{(0)} \oplus \Q^2_{(d)}$, completing the proof.
\end{proof}

\section{The mapping cone formula and $5_2$} \label{sec:5_2-mapping-cone}

Suppose for some knot $K\not\cong 5_2$ and some rational slope $r>0$ that $S^3_r(K) \cong S^3_r(5_2)$.  In this section we will apply the mapping cone formula, Theorem~\ref{thm:mapping-cone}, to compare $\hfp(S^3_r(K))$ to $\hfp(S^3_r(5_2))$.

Throughout this section we will assume that $K$ has genus $g \geq 2$.  Then Proposition~\ref{prop:5_2-fibered} says that $K$ is fibered, and that we can write
\[ H_*(A^+_{g-1}(K)) \cong \cT^+_{(0)} \oplus \Q^{\vphantom{+}}_{(d)} \]
for some integer $d\in\Z$.  Proposition~\ref{prop:hfk-nonpositive} then describes $\hfkhat(K)$ completely in terms of $g$ and $d$.

We also record from Lemma~\ref{lem:5_2-V_0}, together with Proposition~\ref{prop:vs-hs}, the values
\begin{align*}
V_s(K) &= \begin{cases} 0, & s \geq 0 \\ -s, & s < 0, \end{cases} &
H_s(K) &= \begin{cases} s, & s \geq 0 \\ 0, & s < 0. \end{cases}
\end{align*}
The values of $V_s(5_2)$ and $H_s(5_2)$ are identical, so we will refer to these throughout without reference to the particular knot.

\subsection{Preliminaries}

We begin by recording some facts about the mapping cone formula which will simplify our computations.

\begin{proposition} \label{prop:truncate}
Let $K \subset S^3$ be a nontrivial knot of genus $g \geq 1$, and let $p,q > 0$ be relatively prime integers.  Fix an integer $i$, and suppose there are some integers $s \leq s'$ such that
\begin{itemize}
\item $h^+_{\lfloor\frac{i+pt}{q}\rfloor}$ is a quasi-isomorphism for all $t<s$, and
\item $v^+_{\lfloor\frac{i+pt}{q}\rfloor}$ is a quasi-isomorphism for all $t>s'$.
\end{itemize}
Define truncated complexes
\begin{align*}
\bA^{[s,s']}_{i,p/q} &= \bigoplus_{s\leq t \leq s'} \left(t,A^+_{\lfloor\frac{i+pt}{q}\rfloor}\right), &
\bB^{[s,s']}_{i,p/q} &= \bigoplus_{s< t \leq s'} \left(t,B^+\right),
\end{align*}
and a map
\begin{align*}
D^{[s,s']}_{i,p/q}: \bA^{[s,s']}_{i,p/q} &\to \bB^{[s,s']}_{i,p/q} \\
(t,a_t) &\mapsto (t, v^+_{\lfloor\frac{i+pt}{q}\rfloor}(a_t)) + (t+1,h^+_{\lfloor\frac{i+pt}{q}\rfloor}(a_t))
\end{align*}
where we interpret $(s,v^+_{\lfloor\frac{i+ps}{q}\rfloor}(a_s))$ and $(s'+1,h^+_{\lfloor\frac{i+ps'}{q}\rfloor}(a_{s'}))$ as zero.  Then there is an isomorphism
\[ \hfp(S^3_{p/q}(K),i) \cong \ker\left( (D^{[s,s']}_{i,p/q})_*: H_*(\bA^{[s,s']}_{i,p/q}) \to H_*(\bB^{[s,s']}_{i,p/q}) \right) \]
of relatively graded $\Q[U]$-modules.
\end{proposition}

\begin{proof}
Theorem~\ref{thm:mapping-cone} gives a relatively graded isomorphism between $\hfp(S^3_{p/q}(K),i)$ and the homology of the mapping cone $\bX^+_{i,p/q}$, which we can write as
\[ \begin{tikzcd}
\cdots \ar[dr,"h^+"] &
A^+_{\lfloor\frac{i+(t-1)p}q\rfloor} \ar[d,"v^+"] \ar[dr,"h^+"] &
A^+_{\lfloor\frac{i+t p}q\rfloor} \ar[d,"v^+"] \ar[dr,"h^+"] &
A^+_{\lfloor\frac{i+(t+1)p}q\rfloor} \ar[d,"v^+"] \ar[dr,"h^+"] &
\cdots \\
\cdots & B^+ & B^+ & B^+ & \cdots
\end{tikzcd} \]
where we understand each $h^+$ or $v^+$ with domain $A^+_j$ to mean $h^+_j$ or $v^+_j$ respectively.  We observe that the subcomplex
\[ \begin{tikzcd}
\cdots \ar[dr,"h^+"] &
A^+_{\lfloor\frac{i+(s-1)p}q\rfloor} \ar[d,"v^+"] \ar[dr,"h^+"] &
\hphantom{A^+_{\lfloor\frac{i+s p}q\rfloor}} &
\hphantom{A^+_{\lfloor\frac{i+(s+1)p}q\rfloor}} &
\hphantom{\cdots} \\
\cdots & B^+ & B^+, & \hphantom{B^+} & \hphantom{\cdots}
\end{tikzcd} \]
consisting of all summands $(t,A^+_{\lfloor\frac{i+tp}{q}\rfloor})$ with $t<s$ and all $(t,B^+)$ with $t\leq s$, is acyclic because each of its $h^+$ maps is a quasi-isomorphism.  Similarly, the subcomplex
\[ \begin{tikzcd}
\hphantom{\cdots} &
\hphantom{A^+_{\lfloor\frac{i+(s'-1)p}q\rfloor}} &
\hphantom{A^+_{\lfloor\frac{i+s'p}q\rfloor}} &
A^+_{\lfloor\frac{i+(s'+1)p}q\rfloor} \ar[d,"v^+"] \ar[dr,"h^+"] &
\cdots \\
\phantom{\cdots} & \hphantom{B^+} & \hphantom{B^+} & B^+ & \cdots,
\end{tikzcd} \]
consisting of all summands $(t,A^+_{\lfloor\frac{i+tp}{q}\rfloor})$ and $(t,B^+)$ with $t > s'$, is acyclic because each of its $v^+$ maps is a quasi-isomorphism.  Thus we may take the quotient of $\bX^+_{i,p/q}$ by each of these subcomplexes in turn, and the projection maps are both quasi-isomorphisms.  But this leaves the truncated complex
\[ \begin{tikzcd}
A^+_{\lfloor\frac{i+ps}{q}\rfloor} \ar[dr,"h^+"] &
A^+_{\lfloor\frac{i+p(s+1)}q\rfloor} \ar[d,"v^+"] \ar[dr,"h^+"] &
\cdots \ar[dr,"h^+"] &
A^+_{\lfloor\frac{i+p(s'-1)}q\rfloor} \ar[d,"v^+"] \ar[dr,"h^+"] &
A^+_{\lfloor\frac{i+ps'}q\rfloor} \ar[d,"v^+"] \\
 & B^+ & \cdots & B^+ & B^+,
\end{tikzcd} \]
which is precisely the mapping cone $\bX^{[s,s']}_{i,p/q}$ of $D^{[s,s']}_{i,p/q}$, and so we have
\[ \hfp(S^3_{p/q}(K),i) \cong H_*(\bX^{[s,s']}_{i,p/q})). \]

The truncated mapping cone fits into a long exact sequence
\[ \cdots \to H_{*+1}(\bX^{[s,s']}_{i,p/q}) \to H_*(\bA^{[s,s']}_{i,p/q}) \xrightarrow{(D^{[s,s']}_{i,p/q})_*} H_*((\bB^{[s,s']}_{i,p/q}) \to \dots, \]
and so it now suffices to prove that $(D^{[s,s']}_{i,p/q})_*$ is surjective, cf.\ \cite[Lemma~2.8]{ni-wu}.  But the restriction of $(D^{[s,s']}_{i,p/q})_*$ to all of the tower summands
\[ \cT^+ \subset H_*(A^+_{\lfloor\frac{i+pt}{q}\rfloor}) \subset \bigoplus_{s\leq t \leq s'} H_*(A^+_{\lfloor\frac{i+pt}{q}\rfloor}) \cong H_*(\bA^{[s,s']}_{i,p/q}) \]
has the form 
\[ \begin{tikzcd}
\cT^+ \ar[dr,"h^+_*"] &
\cT^+ \ar[d,"v^+_*"] \ar[dr,"h^+_*"] &
\cdots \ar[dr,"h^+_*"] &
\cT^+ \ar[d,"v^+_*"] \ar[dr,"h^+_*"] &
\cT^+ \ar[d,"v^+_*"] 
& \subset & H_*(\bA^{[s,s']}_{i,p/q}) \ar[d,"(D^{[s,s']}_{i,p/q})_*"] \\
 & \cT^+ & \cdots & \cT^+ & \cT^+ & \cong & H_*(\bB^{[s,s']}_{i,p/q}),
\end{tikzcd} \]
and each of the $v^+_*$ and $h^+_*$ components are surjective, so it follows that the total map is surjective as well.  This identifies $H_*(\bX^{[s,s']}_{i,p/q})$, and hence $\hfp(S^3_{p/q}(K),i)$, with the kernel of $(D^{[s,s']}_{i,p/q})_*$ up to an overall grading shift, as promised.
\end{proof}

\begin{corollary} \label{cor:large-rational}
Let $K \subset S^3$ be a nontrivial knot of genus $g \geq 1$, and let $p,q > 0$ be relatively prime integers.  Fix an integer $i$, and suppose there is some $s \in \Z$ such that
\begin{itemize}
\item $h^+_{\lfloor\frac{i+pt}{q}\rfloor}$ is a quasi-isomorphism for all $t<s$, and
\item $v^+_{\lfloor\frac{i+pt}{q}\rfloor}$ is a quasi-isomorphism for all $t>s$.
\end{itemize}
Then $\hfp(S^3_{p/q}(K),i) \cong H_*(A^+_{\lfloor\frac{i+ps}{q}\rfloor})$ as relatively graded $\Q[U]$-modules.
\end{corollary}

\begin{proof}
We apply Proposition~\ref{prop:truncate} to identify $\hfp(S^3_{p/q}(K),i)$ with the kernel of the map
\[ (D^{[s,s]}_{i,p/q})_*: H_*(A^+_{\lfloor\frac{i+ps}{q}\rfloor}) \to 0. \qedhere \]
\end{proof}

\begin{proposition} \label{prop:very-large-rational}
Let $K\subset S^3$ be a knot of genus $g \geq 1$, and fix $i\in\Z$ and $\frac{p}{q} \geq 2g-1$.  Then there is at most one $s\in\Z$ such that 
\[ 1-g \leq \left\lfloor\frac{i+ps}{q}\right\rfloor \leq g-1, \]
and we have
\[ \hfp(S^3_{p/q}(K),i) = \begin{cases} H_*(A^+_{\lfloor\frac{i+ps}{q}\rfloor}) & \text{if }s\text{ exists} \\ \cT^+ & \text{otherwise} \end{cases} \]
as relatively graded $\Q[U]$-modules.
\end{proposition}

\begin{proof}
Suppose first that $s$ exists.  The desired inequality is equivalent to
\[ q(1-g) \leq i+ps < qg. \]
Thus if there is a solution $s$, then for all integers $t > s$ we have
\[ i+pt \geq i+p(s+1) \geq q(1-g)+p \geq q(1-g) + q(2g-1) = qg, \]
while for all integers $t < s$ we have
\[ i+pt \leq i+p(s-1) < qg -p \leq qg - q(2g-1) = q(1-g). \]
In either case $t$ cannot be a solution, so if $s$ exists then it is unique.  But then we know that
\begin{itemize}
\item $h^+_{\lfloor\frac{i+pt}{q}\rfloor}$ is a quasi-isomorphism for all $t<s$, since $\lfloor\frac{i+pt}{q}\rfloor \leq -g$; and
\item $v^+_{\lfloor\frac{i+pt}{q}\rfloor}$ is a quasi-isomorphism for all $t>s$, since $\lfloor\frac{i+pt}{q}\rfloor \geq g$.
\end{itemize}
So Corollary~\ref{cor:large-rational} tells us that $\hfp(S^3_{p/q}(K),i) \cong H_*(A^+_{\lfloor\frac{i+ps}{q}\rfloor})$, as claimed.

Now if no such $s$ exists, then we let $\sigma$ be the least integer such that $\lfloor\frac{i+p\sigma}{q}\rfloor \geq 0$.  It follows that $\lfloor\frac{i+pt}{q}\rfloor \leq -g$ for all $t < \sigma$, and that $\lfloor\frac{i+pt}{q}\rfloor \geq g$ for all $t > \sigma$, so now Corollary~\ref{cor:large-rational} says that
\[ \hfp(S^3_{p/q}(K),i) \cong H_*(A^+_{\lfloor\frac{i+p\sigma}{q}\rfloor}). \]
But in fact $\lfloor\frac{i+p\sigma}{q}\rfloor \geq g$, so $H_*(A^+_{\lfloor\frac{i+p\sigma}{q}\rfloor}) \cong H_*(B^+) \cong \cT^+$ and this completes the proof.
\end{proof}

\subsection{Computations for $5_2$} \label{ssec:5_2-computations}

We begin by computing $\hfp(S^3_{p/q}(5_2),i)$ for all slopes $\frac{p}{q} \geq 1$.  We recall from \eqref{eq:5_2-A0} that
\[ H_*(A^+_0(5_2)) \cong \cT^+_{(0)} \oplus \Q^2_{(0)}. \]

\begin{lemma} \label{lem:5_2-large-rational}
If $\frac{p}{q} \geq 1$ and $0 \leq i \leq p-1$, then we have
\[ \hfp(S^3_{p/q}(5_2),i) \cong \begin{cases} \cT^+_{(0)} \oplus \Q^2_{(0)}, & i = 0,1,\dots,q-1 \\ \cT^+_{(0)}, & \text{otherwise} \end{cases} \]
as relatively graded $\Q[U]$-modules.
\end{lemma}

\begin{proof}
The condition $\lfloor\frac{i+ps}{q}\rfloor = 0$ is equivalent to
\[ 0 \leq i+ps < q, \]
so we can find such an $s$ if and only if $i\equiv 0,1,\dots,q-1 \pmod{p}$, or (since we assumed $0 \leq i \leq p-1$) if and only if $0 \leq i \leq q-1$.  If $s$ does not exist then $\hfp(S^3_{p/q}(K),i) \cong \cT^+$ by Proposition~\ref{prop:very-large-rational} (applied with $g=g(5_2)=1$).  If instead $s$ exists, then we must have $0\leq i \leq q-1$, and now Proposition~\ref{prop:very-large-rational}, together with \eqref{eq:5_2-A0}, says that
\[ \hfp(S^3_{p/q}(5_2),i) \equiv H_*(A^+_0(5_2)) \cong \cT^+_{(0)} \oplus \Q^2_{(0)} \]
as relatively graded $\Q[U]$-modules.
\end{proof}

\begin{proposition} \label{prop:5_2-small-rational}
Suppose that $0 < \frac{p}{q} < 1$.  Then
\[ \hfp(S^3_{p/q}(5_2), i) \cong \cT^+_{(0)} \oplus \Q^{2n_i}_{(0)} \]
as relatively graded $\Q[U]$-modules, where $n_i$ is the number of $t\in\Z$ such that $0 \leq i+pt < q$.
\end{proposition}

\begin{proof}
We define a pair of integers $s, s'$ by
\begin{align*}
s &= \min\{ t \in \Z \mid i+pt \geq 0 \}, \\
s' &= \max\{ t \in \Z \mid i+pt \leq q-1 \}.
\end{align*}
Then $p<q$ implies that $s \leq s'$, and for all $t \in \Z$ we have 
\[ \left\lfloor\frac{i+pt}{q}\right\rfloor = 0 \quad\Longleftrightarrow\quad s \leq t \leq s', \]
so $n_i = s' - s + 1$.

Now Proposition~\ref{prop:truncate} says that $\hfp(S^3_{p/q}(5_2),i)$ is isomorphic to the kernel of $(D^{[s,s']}_{i,p/q})_*$.  Recalling again from \eqref{eq:5_2-A0} that $H_*(A^+_0) \cong \cT^+_{(0)} \oplus \Q^2_{(0)}$, this map has the form
\[ \begin{tikzcd}
\cT^+_{(0)}\oplus\Q^2_{(0)} \ar[dr,"h^+_*"] &
\cT^+_{(0)}\oplus\Q^2_{(0)} \ar[d,"v^+_*"] \ar[dr,"h^+_*"] &
\cdots \ar[dr,"h^+_*"] &
\cT^+_{(0)}\oplus\Q^2_{(0)} \ar[d,"v^+_*"] \ar[dr,"h^+_*"] &
\cT^+_{(0)}\oplus\Q^2_{(0)} \ar[d,"v^+_*"] 
\\
 & \cT^+_{(-1)} & \cdots & \cT^+_{(-1)} & \cT^+_{(-1)}.
\end{tikzcd} \]
Here we are able to assign these gradings to each summand because $V_0(5_2) = H_0(5_2) = 0$, and so each of the maps $v^+_* = (v^+_0)_*$ and $h^+_* = (h^+_0)_*$ gives a degree-$(-1)$ isomorphism between the respective towers.

We see by inspection that $\ker (D^{[s,s']}_{i,p/q})_*$ contains a tower $\cT^+$ whose bottom-most element is in grading $0$, as an alternating sum of the bottom-most elements of the towers $\cT^+_{(0)} \subset H_*(A^+_{\lfloor\frac{i+pt}{q}\rfloor})$, $s \leq t \leq s'$.  The map $(D^{[s,s']}_{i,p/q})_*$ also sends
\[ H_0(\bA^{[s,s']}_{i,p/q}) \cong \Q^{3(s'-s+1)} \]
onto 
\[ H_{-1}(\bB^{[s,s']}_{i,p/q}) \cong \Q^{s'-s}, \]
so its kernel has total dimension $2(s'-s)+3$ in degree zero.  We conclude that
\[ \hfp(S^3_{p/q}(K), i) \cong \cT^+_{(0)} \oplus \Q^{2(s'-s+1)}_{(0)} \]
as relatively graded $\Q[U]$-modules.
\end{proof}

\subsection{General facts about the kernel of $U$} \label{ssec:K-ker-U}

We will show that under most circumstances, a positive $r$-surgery on a knot of genus at least $2$ cannot have the same Heegaard Floer homology as the corresponding surgery on $5_2$.  We will handle the cases $r < 1$ and $r \geq 1$ in the next few subsections; before that, we prepare for this work here by proving some general facts about the kernel of the $U$-action on $\hfp$ of these surgeries.


\begin{lemma} \label{lem:gK-tower-gradings}
Let $K$ be a knot of genus $g \geq 2$, and suppose for some relatively prime integers $p,q > 0$ that
\[ \hfp(S^3_{p/q}(K)) \cong \hfp(S^3_{p/q}(5_2)) \]
as absolutely graded $\Q[U]$-modules.  Fix an integer $i$, and lift the relative gradings on the complexes $\bA^+_{i,p/q}$ and $\bB^+_{i,p/q}$ to absolute $\Z$-gradings so that $D^+_{i,p/q}$ has degree $-1$.  Let $d_s$ denote the grading of the bottom-most element of the tower
\[ \cT^+ \subset (s, H_*(A^+_{\lfloor\frac{i+ps}{q}\rfloor})) \subset H_*(\bA^+_{i,p/q}) \]
for each $s$.
\begin{enumerate}
\item If $\lfloor\frac{i+ps}{q}\rfloor \geq 0$, then $d_{s+1} = d_s + 2\lfloor\frac{i+ps}{q}\rfloor$.
\item If $\lfloor\frac{i+ps}{q}\rfloor \leq 0$, then $d_s = d_{s-1} + 2\lfloor\frac{i+ps}{q}\rfloor$.
\item If $\lfloor\frac{i+ps}{q}\rfloor \leq 0$ and $\lfloor\frac{i+p(s+1)}{q}\rfloor \geq 0$, then $d_s = d_{s+1}$.
\end{enumerate}
\end{lemma}

\begin{proof}
If $\lfloor\frac{i+ps}{q}\rfloor \geq 0$, then the map $(D^+_{i,p/q})_*$ on homology restricts to the sum of all of the towers $(s,\cT^+_{(d_s)}) \subset H_*(\bA^+_{i,p/q})$ as
\[ \begin{tikzcd}
\cdots & t=s-1 & s & s+1 & \dots \\[-0.75em]
\cdots \ar[dr,"h^+_*"] &
\cT^+_{(d_{s-1})} \ar[d,"v^+_*"] \ar[dr,"h^+_*"] &
\cT^+_{(d_s)} \ar[d,"v^+_*"] \ar[dr,"h^+_*"] &
\cT^+_{(d_{s+1})} \ar[d,"v^+_*"] \ar[dr,"h^+_*"] &
\dots \\
\cdots & \cT^+_{(e_{s-1})} & \cT^+_{(e_s)} & \cT^+_{(e_{s+1})} & \dots
\end{tikzcd} \]
for some integers $e_{s-1},\ e_s,\ e_{s+1}$.

Let $n = \lfloor\frac{i+ps}{q}\rfloor$.  If $n \geq 0$ then $H_n(K) = n$, so the $h^+_*$ map with domain in column $s$ above has the form
\[ (h^+_n)_*: \cT^+_{(d_s)} \xrightarrow{U^n} \cT^+_{(e_{s+1})}, \]
sending a generator in degree $d_s+2n$ to one in degree $e_{s+1}$, so we have
\[ (d_s+2n)-1 = e_{s+1}. \]
But then $\lfloor\frac{i+p(s+1)}{q}\rfloor \geq n \geq 0$, so the $v^+_*$ map in column $s+1$ is identified with the identity map $\cT^+_{(d_{s+1})} \to \cT^+_{(e_{s+1})}$ and thus
\[ d_{s+1} = e_{s+1}+1 = d_s + 2n. \]
Similarly, if $n \leq 0$ then we have $H_{\lfloor\frac{i+p(s-1)}{q}\rfloor} = 0$ and $V_n = -n$, hence
\[ (d_s + 2(-n)) - 1 = e_s = d_{s-1} - 1, \]
or $d_s = d_{s-1}+2n$.

In the case where $\lfloor\frac{i+ps}{q}\rfloor \leq 0$ and $\lfloor\frac{i+p(s+1)}{q}\rfloor \geq 0$, we note that the $h^+_*$ and $v^+_*$ maps into the $\cT^+_{(e_{s+1})}$ tower in column $s+1$ are both modeled on multiplication by $1$, since $H_{\lfloor\frac{i+ps}{q}\rfloor}(K) = 0$ and $V_{\lfloor\frac{i+p(s+1)}{q}\rfloor}(K) = 0$.  Thus
\[ d_s = e_{s+1}+1 = d_{s+1}, \]
completing the proof.
\end{proof}

\begin{lemma} \label{lem:grading-homology-tower}
Assume the hypotheses and notation of Lemma~\ref{lem:gK-tower-gradings}, and let
\[ s_0 = \min\left\{ t \in \Z \,\middle\vert\, \left\lfloor\frac{i+pt}{q}\right\rfloor \geq 0 \right\}. \]
Fix integers $s$ and $s'$ satisfying the hypotheses of Proposition~\ref{prop:truncate}, and consider the map
\[ (D^{[s,s']}_{i,p/q})_*: H_*(\bA^{[s,s']}_{i,p/q}) \to H_*(\bB^{[s,s']}_{i,p/q}) \]
between the homologies of the corresponding truncated complexes.  If $s \leq s_0 \leq s'$, then
\[ \ker(D^{[s,s']}_{i,p/q})_* \cap \ker(U) \]
contains a $\Q$ submodule in grading $d_{s_0}$.
\end{lemma}

\begin{proof}
Consider the restriction of $(D^{[s,s']}_{i,p/q})_*$ to the sum of all the towers $(t,\cT^+_{(d_t)}) \subset H_*(\bA^{[s,s']}_{i,p/q})$.  By hypothesis we have
\[ \left\lfloor\frac{i+p(s_0-1)}{q}\right\rfloor < 0 \quad\text{and}\quad \left\lfloor\frac{i+ps_0}{q}\right\rfloor \geq 0, \]
so Lemma~\ref{lem:gK-tower-gradings} says that the sequence of gradings $\big(d_t\big)$ satisfies
\[ \cdots > d_s > d_{s+1} > \dots > d_{s_0-1} = d_{s_0} \leq d_{s_0+1} \leq \cdots \leq d_{s'} \leq \cdots. \]
Let
\[ s_1 = \max \{ t \in \Z \mid d_t = d_{s_0} \}, \]
so that for all $t\in\{s,\dots,s'\}$, we have $d_t = d_{s_0}$ if and only if $s_0-1 \leq t \leq s_1$.  Then near the indices $[s_0-1,s_1+1]$, the restriction of $(D^{[s,s']}_{i,p/q})_*$ has the form
\[ \begin{tikzcd}
t=s_0-1 & s_0 & \cdots & s_1 & s_1+1 \\[-0.75em]
\cT^+_{(d_{s_0})} \ar[d,"v^+_*"] \ar[dr,"1"] &
\cT^+_{(d_{s_0})} \ar[d,"1"] \ar[dr,"1"] &
\cdots \ar[dr,"1"] &
\cT^+_{(d_{s_0})} \ar[d,"1"] \ar[dr,"h^+_*"] &
\cT^+_{(d_{s_1+1})} \ar[d,"v^+_*"] \\
\cT^+ & \cT^+_{(d_{s_0-1})} & \cdots & \cT^+_{(d_{s_0-1})} & \cT^+
\end{tikzcd} \]
in which we omit any columns at either end whose indices are not in $[s,s']$.

To see that the maps labeled ``$1$'' are indeed modeled on multiplication by $U^0 = 1$, we note that they are one of
\begin{itemize}
\item an $h^+_*$ map with domain in column $s_0-1$, and then since $\lfloor\frac{i+p(s_0-1)}{q}\rfloor < 0$ we have $H_{\lfloor\frac{i+p(s_0-1)}{q}\rfloor}(K) = 0$;
\item a $v^+_*$ map with domain in column $t \geq s_0$, and then since $\lfloor\frac{i+pt}{q}\rfloor \geq 0$ we have $V_{\lfloor\frac{i+pt}{q}\rfloor}(K) = 0$; or
\item an $h^+_*$ map from column $t \geq s_0$ to column $t+1$ where $d_t = d_{t+1} = d_{s_0}$, and then Lemma~\ref{lem:gK-tower-gradings} says that
\[ 0 = d_{t+1} - d_t = 2\left\lfloor\frac{i+pt}{q}\right\rfloor, \]
so that $H_{\lfloor\frac{i+pt}{q}\rfloor}(K) = H_0(K) = 0$.
\end{itemize}
Moreover, the $v^+_*$ map in column $s_0-1$ is modeled on multiplication by $U^a$, where
\[ a = V_{\lfloor\frac{i+p(s_0-1)}{q}\rfloor}(K) = -\left\lfloor\frac{i+p(s_0-1)}{q}\right\rfloor \geq 1. \]
Similarly the $h^+_*$ map in column $s_1$ is modeled on multiplication by $U^b$, where
\begin{align*}
b = H_{\lfloor\frac{i+ps_1}{q}\rfloor}(K) &= \left\lfloor\frac{i+ps_1}{q}\right\rfloor \\
&= \tfrac{1}{2}\left( d_{s_1+1} - d_{s_1} \right) > 0,
\end{align*}
by Lemma~\ref{lem:gK-tower-gradings} and the definition of $s_1$.

We now label generators at the bottom of each tower by
\begin{align*}
x_t &\in (t, \cT^+_{(d_t)}) \subset H_*(\bA^{[s,s']}_{i,p/q}), &
y_t &\in (t, \cT^+) \subset H_*(\bB^{[s,s']}_{i,p/q}),
\end{align*}
so that $Ux_t=0$ and $Uy_t=0$ for all $t$, and the various $v^+_*$ and $h^+_*$ maps carry elements of the form $U^i x_t$ to elements of the form $U^j y_t$ and $U^k y_{t+1}$ respectively.  We then define
\[ z = \sum_{t=s_0-1}^{s_1} (-1)^t x_t \subset H_{d_{s_0}}(\bA^{[s,s']}_{i,p/q}), \]
treating any terms whose indices are not in $[s,s']$ as zero, and it follows from the above discussion that $Uz=0$ and that
\[ (D^{[s,s']}_{i,p/q})_*(z) = (-1)^{s_0-1} y_{s_0} + \left( \sum_{t=s_0}^{s_1-1} (-1)^t(y_t + y_{t+1})\right) + (-1)^{s_1} y_{s_1} = 0. \]
Thus $z$ generates the desired $\Q$ summand.
\end{proof}

\begin{lemma} \label{lem:gK-large-ker-U}
Assume the hypotheses and notation of Lemma~\ref{lem:gK-tower-gradings}, and let $s \leq s'$ be integers satsifying the hypotheses of Proposition~\ref{prop:truncate}.  Suppose that
\[ \left\lfloor\frac{i+ps'}{q}\right\rfloor = g-1. \]
If $d\in\Z$ denotes the integer such that $H_*(A^+_{g-1}) \cong \cT^+_{(0)} \oplus \Q^{\vphantom{+}}_{(d)}$, as in Lemma~\ref{lem:As-tower-grading}, then we can write
\[ (s',H_*(A^+_{g-1})) \cong \cT^+_{(d_{s'})} \oplus \Q^{\vphantom{+}}_{(d_{s'}+d)} \]
as $\Q[U]$-modules such that the $\Q_{(d_{s'}+d)}$ summand lies in
\[ \ker(D^{[s,s']}_{i,p/q})_* \cap \ker(U). \]
\end{lemma}

\begin{proof}
The rightmost portion of the truncated mapping cone complex has the form
\[ \begin{tikzcd}
 \cdots & t=s'-1 & s' \\[-0.75em]
\cdots \ar[dr,"h^+_*"] &
H_*(A^+_{\lfloor\frac{i+p(s'-1)}{q}\rfloor}) \ar[d,"v^+_*"] \ar[dr,"h^+_*"] &
\cT^+_{(d_{s'})} \oplus \Q^{\vphantom{+}}_{(d_{s'}+d)} \ar[d,"(v^+_{g-1})_*"] \\
\cdots & \cT^+ & \cT^+_{(d_{s'}-1)},
\end{tikzcd} \]
where the grading on the bottom $\cT^+$ in column $s'$ follows from $V_{g-1}(K)=0$.  Let $x_{s'}$ and $y_{s'}$ be bottom-most elements of the towers at the top and bottom of column $s'$, chosen so that $(v^+_{g-1})_*(x_{s'}) = y_{s'}$.

Let $z$ generate the $\Q_{(d_{s'}+d)}$ summand in column $s'$, so that $Uz=0$.  If $(v^+_{g-1})_*(z)=0$ then we are done, since $z$ generates the desired submodule.  Otherwise, we observe that
\[ U \cdot (v^+_{g-1})_*(z) = (v^+_{g-1})_*(Uz) = (v^+_{g-1})_*(0) = 0, \]
so $(v^+_{g-1})_*(z)$ must be a nonzero element at the bottom of the $\cT^+_{(d_{s'}-1)}$ tower.  In this case, we can write
\[ (v^+_{g-1})_*(z) = \lambda y_{s'} = \lambda \cdot (v^+_{g-1})_*(x_{s'}) \]
for some nonzero $\lambda\in\Q$.  For grading reasons it now follows that $d=0$, since $z$ must lie in grading $d_{s'}$, and so
\[ z - \lambda x_{s'} \in \ker(v^+_{g-1})_*. \]
Now we can write the $H_*(A^+_{g-1})$ in column $s'$ as the $\Q[U]$-module
\[ \cT^+\langle x_{s'}\rangle \oplus \Q\langle z-\lambda x_{s'}\rangle, \]
and the $\Q$ summand is in $\ker(v^+_{g-1})_* = \ker(D^{[s,s']}_{i,p/q})_*$ as well as $\ker(U)$, as desired.
\end{proof}

\subsection{Small positive surgeries} \label{ssec:K-small}

In Proposition~\ref{prop:5_2-small-rational} we showed that if $0 < r < 1$, then there is a relatively graded isomorphism of the form
\[ \hfp(S^3_r(5_2),i) \cong \cT^+_{(0)} \oplus \Q^{2n_i}_{(0)} \]
for all $i$.  We will show that this cannot be the case for $\hfp(S^3_r(K))$ if $K$ is a knot of genus at least 2 that satisfies the hypotheses of Proposition~\ref{prop:5_2-fibered}.

\begin{proposition} \label{prop:gK-large-r-small}
Let $K$ be a knot of genus $g \geq 2$, and fix relatively prime positive integers $q > p > 0$.  Then
\[ \hfp(S^3_{p/q}(K)) \not\cong \hfp(S^3_{p/q}(5_2)) \]
as absolutely graded $\Q[U]$-modules.
\end{proposition}

\begin{proof}
If $\hfp(S^3_{p/q}(K)) \cong \hfp(S^3_{p/q}(5_2))$, then $K$ satisfies the conclusions of Proposition~\ref{prop:5_2-fibered}.  In this case, Proposition~\ref{prop:5_2-small-rational} says that for all $i$, the submodule
\[ \ker(U) \subset \hfp(S^3_{p/q}(5_2),i) \]
lies in a single homological grading.  Thus the same must be true for
\[ \ker(U) \subset \hfp(S^3_{p/q}(K),i), \]
so we will find an integer $i$ for which this is not the case, giving a contradiction.

We fix an integer $i$ between $0$ and $p-1$, inclusive, such that
\[ i \equiv gq-1 \pmod{p}. \]
We then define
\begin{align*}
s &= \min\{ t \in \Z \mid i+pt \geq (1-g)q \}, &
s' &= \frac{gq-1-i}{p}.
\end{align*}
By construction we have
\[ \left\lfloor\frac{i+ps'}{q}\right\rfloor = g-1 \quad\text{and}\quad \left\lfloor\frac{i+p(s'+1)}{q}\right\rfloor \geq g, \]
and since $1 \leq p+1 \leq q$ we have 
\[ \left\lfloor\frac{i+p(s'-1)}{q}\right\rfloor = \left\lfloor\frac{gq-(p+1)}{q}\right\rfloor = g-1 \]
as well.  We also observe that $\lfloor\frac{i+pt}{q}\rfloor \geq 0$ if and only if $t \geq 0$, and so $s \leq t \leq s'$.

According to Proposition~\ref{prop:truncate}, we can identify $\hfp(S^3_{p/q}(K),i)$ with the kernel of
\[ (D^{[s,s']}_{i,p/q})_*: H_*(\bA^{[s,s']}_{i,p/q}) \to H_*(\bB^{[s,s']}_{i,p/q}) \]
up to an overall grading shift, so it will suffice to show that
\[ \ker (D^{[s,s']}_{i,p/q})_* \cap \ker(U) \]
does not lie in a single homological grading.  Supposing otherwise, we choose an arbitrary lift of the relative gradings on $\bA^{[s,s']}_{i,p/q}$ and $\bB^{[s,s']}_{i,p/q}$ to an absolute $\Z$-grading, and let $d_t \in \Z$ denote the bottom-most grading in each tower
\[ \cT^+_{(d_t)} \subset (t, H_*(A^+_{\lfloor\frac{i+pt}{q}\rfloor})) \subset H_*(\bA^{[s,s']}_{i,p/q}). \]
Lemma~\ref{lem:grading-homology-tower} now says that there is a $\Q$-submodule of $\ker(D^{[s,s']}_{i,p/q})_*$ in grading $d_{s_0}$, and Lemma~\ref{lem:gK-large-ker-U} says that there is also a $\Q$-submodule in grading $d_{s'}+d$, hence
\[ d_{s'}+d = d_{s_0} \]
by hypothesis.  But according to Lemma~\ref{lem:gK-tower-gradings}, we also have
\begin{align*}
d_{s'} &= d_{s'-1} + 2\left\lfloor\frac{i+p(s'-1)}{q}\right\rfloor \\
&= d_{s'-1} + 2(g-1) \\
&\geq d_{s_0} + 2(g-1), 
\end{align*}
so then $d = d_{s_0} - d_{s'} \leq 2-2g$.

We now examine the rightmost portion of the truncated complex $\bX^{[s,s']}_{i,p/q}$.  Since $\lfloor\frac{i+p(s'-1)}{q}\rfloor = g-1$, the last two columns have the form
\[ \begin{tikzcd}
 \cdots & t=s'-1 & s' \\[-0.75em]
\cdots \ar[dr,"h^+_*",outer sep=-2pt] &
\cT^+_{(d_{s'-1})} \oplus \Q^{\vphantom{+}}_{(d_{s'-1}+d)} \ar[d,"(v^+_{g-1})_*"] \ar[dr,"(h^+_{g-1})_*"] &
\cT^+_{(d_{s'})} \oplus \Q^{\vphantom{+}}_{(d_{s'}+d)} \ar[d,"(v^+_{g-1})_*"] \\
\cdots & \cT^+_{(d_{s'-1}-1)} & \cT^+_{(d_{s'}-1)},
\end{tikzcd} \]
with $d_{s'} = d_{s'-1} + 2(g-1)$ as above.  Since $d \leq 2-2g \leq -2$, the map $(v^+_{g-1})_*$ in column $s'-1$ must send the $\Q_{(d_{s'-1}+d)}$ submodule to zero for grading reasons.  That same submodule must be sent by $(h^+_{g-1})_*$ into column $s'$, in grading
\begin{align*}
d_{s'-1}+d - 1 &= (d_{s'} - 2(g-1)) + d-1 \\
&= (d_{s'}+d) + (1-2g) \\
&\leq d_{s'} - 1 - 2g.
\end{align*}
This is strictly less than the bottom-most grading of the corresponding tower, so this image also must be zero, and it follows that in column $s'-1$ we have
\[ \Q_{(d_{s'-1}+d)} \subset \ker (D^{[s,s']}_{i,p/q})_* \cap \ker(U) \]
as well.  Since
\[ d_{s'-1}+d = (d_{s'}+d) - (2g-2) < d_{s'}+d, \]
it follows that $\ker (D^{[s,s']}_{i,p/q})_* \cap \ker(U)$ is not supported in a single grading, and so we have a contradiction.
\end{proof}

\subsection{Large positive surgeries} \label{ssec:K-large}

In this subsection we attempt to understand when there can be a homeomorphism
\[ S^3_r(K) \cong S^3_r(5_2) \]
for some slope $r\geq 1$ and some knot $K$ of genus at least $2$.  We implicitly identify
\[ \spc(S^3_{p/q}(K)) \cong \Z/p\Z \]
throughout, as in the statement of Theorem~\ref{thm:mapping-cone}.

The following lemma will help us find $\spc$ structures $\spinc$ where $\hfp(S^3_{p/q}(K),\spinc)$ differs from $\hfp$ of any $\spc$ structure on $S^3_{p/q}(5_2)$.

\begin{lemma} \label{lem:5_2-find-i}
Let $g \geq 2$ be an integer, and let $p>q>0$ be relatively prime positive integers such that $p$ does not divide $2g-2$.  Then there exists an integer $i\in\Z$ for which the equation
\[ \left\lfloor \frac{i+ps}{q} \right\rfloor = g-1 \]
has an integer solution $s \in \Z$, but the equation
\[ \left\lfloor \frac{i+ps}{q} \right\rfloor = 1-g \]
does not.
\end{lemma}

\begin{proof}
We note that $\lfloor\frac{i+ps}{q}\rfloor = g-1$ admits a solution $s\in\Z$ if and only if
\[ q(g-1) \leq i+ps \leq qg - 1, \]
or equivalently if and only if
\begin{equation} \label{eq:5_2-K-fibered-large-1}
i \equiv qg - j \pmod{p} \text{ for some } j\in\{1,2,\dots,q\}.
\end{equation}
Similarly, the equation $\lfloor\frac{i+ps}{q}\rfloor = 1-g$ has a solution $s\in\Z$ if and only if
\[ q(1-g) \leq i+ps \leq q(2-g)-1, \]
or equivalently if and only if
\begin{equation} \label{eq:5_2-K-fibered-large-2}
i \equiv q(2-g)-k \pmod{p} \text{ for some } k\in\{1,2,\dots,q\}.
\end{equation}
Each of \eqref{eq:5_2-K-fibered-large-1} and \eqref{eq:5_2-K-fibered-large-2} is solved by exactly $q$ residue classes modulo $p$, and these solutions coincide if and only if
\[ qg \equiv q(2-g) \pmod{p}, \]
which is equivalent to $2g-2 \equiv 0\pmod{p}$ since $p$ and $q$ are coprime.  But we have assumed that this is not the case, so the set of $i$ in \eqref{eq:5_2-K-fibered-large-1} is not a subset of the set in \eqref{eq:5_2-K-fibered-large-2}, and hence there is some $i$ which satisfies \eqref{eq:5_2-K-fibered-large-1} but not \eqref{eq:5_2-K-fibered-large-2}.  This is the desired $i$.
\end{proof}

\begin{proposition} \label{prop:5_2-medium-fibered}
Let $K \subset S^3$ be a nontrivial knot of genus $g \geq 2$, and let $p \geq q > 0$ be relatively prime positive integers.  If there is an isomorphism
\[ \hfp(S^3_{p/q}(K)) \cong \hfp(S^3_{p/q}(5_2)) \]
of graded $\Q[U]$-modules, then $p$ divides $2g-2$.
\end{proposition}

\begin{proof}
We suppose that $p \nmid 2g-2$.  Then Lemma~\ref{lem:5_2-large-rational} says that
\[ \dim_\Q \hfred(S^3_{p/q}(5_2),i) \cong 0 \text{ or } 2 \text{ for all } i, \]
so for the sake of a contradiction it will suffice to find $i$ such that $\hfred(S^3_{p/q}(K),i)$ is $1$-dimensional.  We start by applying Lemma~\ref{lem:5_2-find-i} to find $i \in \Z$ and $s' \in \Z$ such that
\[ \left\lfloor\frac{i+ps'}{q}\right\rfloor = g-1 \]
and such that $\lfloor\frac{i+pt}{q}\rfloor = 1-g$ has no solutions $t\in\Z$; this will be the desired $i$.

Let $s$ be the least integer satisfying
\[ \left\lfloor\frac{i+ps}{q}\right\rfloor \geq 0. \]
Then $h^+_{\lfloor\frac{i+pt}{q}\rfloor}$ is a quasi-isomorphism for all $t < s$, since then $\lfloor\frac{i+pt}{q}\rfloor$ is negative but not equal to $1-g$; if $1-g < \lfloor\frac{i+pt}{q}\rfloor < 0$ then this is part of Proposition~\ref{prop:5_2-fibered}, and if $\lfloor\frac{i+pt}{q}\rfloor < 1-g$ then this is true for arbitrary genus-$g$ knots.  Likewise $v^+_{\lfloor\frac{i+pt}{q}\rfloor}(K)$ is a quasi-isomorphism for all $t > s'$, since then $\lfloor\frac{i+pt}{q}\rfloor \geq g$.  Thus Proposition~\ref{prop:truncate} says that $\hfp(S^3_{p/q}(K),i)$ is isomorphic to the kernel of the truncated map
\[ (D^{[s,s']}_{i,p/q})_*: H_*(\bA^{[s,s']}_{i,p/q}) \to H_*(\bB^{[s,s']}_{i,p/q}). \]
The domain is a sum of relatively graded $\Q[U]$-modules of the form
\[ H_*(A^+_{\lfloor\frac{i+pt}{q}\rfloor}) \cong \begin{cases} \cT^+, & s \leq t < s' \\ \cT^+ \oplus \Q, & t=s', \end{cases} \]
and we know that $H_*(B^+) \cong \cT^+$, so $(D^{[s,s']}_{i,p/q})_*$ has the form
\[ \begin{tikzcd}[column sep=2.5em]
t=s & s+1 & \dots & s'-1 & s' \\[-0.75em]
\cT^+ \ar[dr,"h^+_*"] &
\cT^+ \ar[d,"v^+_*"] \ar[dr,"h^+_*"] &
\cdots\vphantom{\cT^+} \ar[dr,"h^+_*"] &
\cT^+ \ar[d,"v^+_*"] \ar[dr,"h^+_*"] &
\cT^+ \oplus \Q \ar[d,"(v^+_{g-1})_*"] \\
& \cT^+ & \cdots & \cT^+ & \cT^+.
\end{tikzcd} \]

Lemma~\ref{lem:gK-large-ker-U} says that we can arrange for the $\Q$ summand in column $s'$ above to belong to $\ker(D^{[s,s']}_{i,p/q})_*$.  Having done so, we see that $\ker(D^{[s,s']}_{i,p/q})_*$ is isomorphic as a $\Q[U]$-module to the direct sum of that $\Q$ with the kernel of
\[ \begin{tikzcd}[column sep=2.5em]
t=s & s+1 & \dots & s'-1 & s' \\
\cT^+ \ar[dr,"h^+_*"] &
\cT^+ \ar[d,"1"] \ar[dr,"h^+_*"] &
\cdots\vphantom{\cT^+} \ar[dr,"h^+_*"] &
\cT^+ \ar[d,"1"] \ar[dr,"h^+_*"] &
\cT^+ \ar[d,"1"] \\
  & \cT^+ & \cdots & \cT^+ & \cT^+.
\end{tikzcd} \]
(Here each $v^+_*$ map is identified with multiplication by $U^0=1$, since $t \geq s$ implies that $\lfloor\frac{i+pt}{q}\rfloor \geq 0$ and hence $V_{\lfloor\frac{i+pt}{q}\rfloor}(K) = 0$.)
But this kernel can be identified with the $\cT^+$ in column $s$, so now we apply Proposition~\ref{prop:truncate} to conclude that
\[ \hfp(S^3_{p/q}(K),i) \cong \ker(D^{[s,s']}_{i,p/q})_* \cong \cT^+ \oplus \Q \]
up to an overall grading shift.  This says that $\hfred(S^3_{p/q}(K),i) \cong \Q$, which gives the desired contradiction.
\end{proof}

Proposition~\ref{prop:5_2-medium-fibered} takes care of most slopes $r \geq 1$ (for knots of a fixed genus $g$) without making use of gradings on the mapping cone complex.  By being careful about gradings, we can handle the remaining non-integral cases as well.  

\begin{proposition} \label{prop:5_2-medium-fibered-2}
Let $K$ be a nontrivial knot of genus $g \geq 2$, and let $p \geq q > 0$ be relatively prime positive integers.  If there is an isomorphism
\[ \hfp(S^3_{p/q}(K)) \cong \hfp(S^3_{p/q}(5_2)) \]
of graded $\Q[U]$-modules, then $q=1$ and $p$ divides $2g-2$.
\end{proposition}

\begin{proof}
Proposition~\ref{prop:5_2-medium-fibered} tells us that $p$ divides $2g-2$, so it remains to be seen that $q=1$.  We will assume to the contrary that $q \geq 2$.  If we write $e=\frac{2g-2}{p}$ then $\frac{p}{q} = \frac{2g-2}{qe}$, and the assumption $q \geq 2$ means that $\frac{p}{q}$ is neither $2g-2$ nor $g-1$, so it follows that $qe \geq 3$, or $\frac{p}{q} \leq \frac{2g-2}{3}$.

As usual, we will take $d\in\Z$ such that
\[ H_*(A^+_{g-1}) \cong \cT^+_{(0)} \oplus \Q^{\vphantom{+}}_{(d)}, \]
as guaranteed by Lemma~\ref{lem:As-tower-grading}.  This integer $d$ depends only on $K$, which is the key fact we will use below to rule out any case where $q \geq 2$.

Fixing some choice of 
\[ i = q(g-1)+j, \quad j=0,1,\dots,q-1, \]
we take $s=-qe$ and $s'=0$, and then we have
\[ \left\lfloor\frac{i+ps}{q}\right\rfloor = \left\lfloor\frac{(q(g-1)+j)-pqe}{q}\right\rfloor = \left\lfloor g-1-pe + \frac{j}{q}\right\rfloor = 1-g \]
and $\left\lfloor\frac{i + ps'}{q}\right \rfloor = g-1$, while (since $\frac{p}{q} \geq 1$) $\lfloor\frac{i+p(s-1)}{q}\rfloor \leq -g$ and $\lfloor\frac{i+p(s'+1)}{q}\rfloor \geq g$.  Thus
\[ \hfp(S^3_{p/q}(K),i) \cong \ker\left( (D^{[s,s']}_{i,p/q})_*: H_*(\bA^{[s,s']}_{i,p/q}) \to H_*(\bB^{[s,s']}_{i,p/q}) \right) \]
by Proposition~\ref{prop:truncate}.  We put an absolute $\Z$-grading on the truncated mapping cone complex $\bX^{[s,s']}_{i,p/q}$, with $d_t$ denoting the bottom-most grading for the tower in each summand $(t,H_*(A^+_{\lfloor\frac{i+pt}{q}\rfloor}))$ as usual, and we let
\[ s_0 = \min \left\{ t \in \Z \,\middle\vert\, \left\lfloor\frac{i+pt}{q}\right\rfloor \geq 0 \right\}. \]
Then Lemmas~\ref{lem:grading-homology-tower} and \ref{lem:gK-large-ker-U} tell us that 
\[ \ker(D^{[s,s']}_{i,p/q})_* \cap \ker(U) \]
contains $\Q$ submodules in gradings $d_{s_0}$ and $d_{s'}+d$ respectively.  But by Proposition~\ref{lem:5_2-large-rational} these gradings must be the same, so we have
\[ -d = d_{s'} - d_{s_0}. \]
We remark that since $\frac{p}{q} \leq \frac{2g-2}{3}$, it follows that $s_0 \leq s'-1$.

We now attempt to work out this value in more detail.  According to Lemma~\ref{lem:gK-tower-gradings}, we have
\[ -d = d_{s'} - d_{s_0} = \sum_{t=s_0}^{s'-1} (d_{t+1}-d_t) = 2\sum_{t=s_0}^{s'-1} \left\lfloor \frac{i+pt}{q} \right\rfloor, \]
which, since $i=q(g-1)+j$, can be written as
\begin{equation} \label{eq:d-from-i-j}
-d = 2\sum_{t=s_0}^{s'-1} \left((g-1) + \left\lfloor \frac{j+pt}{q} \right\rfloor\right).
\end{equation}
We note that
\begin{align*}
\left\lfloor \frac{i+pt}{q} \right\rfloor \geq 0 &\Longleftrightarrow (q(g-1)+j)+pt \geq 0 \\
&\Longleftrightarrow t \geq -q\left(\frac{g-1}{p}\right) - \frac{j}{p} = -q\left(\frac{e}{2}\right) - \frac{j}{p}
\end{align*}
and so we have
\begin{equation} \label{eq:s0-2g-2}
s_0 = \left\lceil -\frac{qe}{2} - \frac{j}{p}\right\rceil = -\left\lfloor \frac{qe}{2} + \frac{j}{p} \right\rfloor,
\end{equation}
while $s'-1=-1$ since $s'=0$ by definition.  This makes it clear that while the various $d_t$ may have depended on our choice of $i$ and on the absolute grading on $\bX^{[s,s']}_{i,p/q}$, the expression \eqref{eq:d-from-i-j} for $d$ depends only on $p$, $q$, $g$, and our choice of $j\in\{0,1,\dots,q-1\}$.  But we have already remarked that $d$ depends only on $K$, so we will show that different choices of $j$ lead to different values of $d$ and thus get a contradiction.

Supposing first that $q \cdot e$ is even, we have $\frac{qe}{2} \in \Z$ while $0 \leq \frac{j}{p} \leq \frac{q-1}{p} < 1$, and so
\[ s_0 = -q\left(\frac{e}{2}\right) \quad\text{for } j=0,1,2,\dots,q-1. \]
In particular, the indices in the sum \eqref{eq:d-from-i-j} are the same for each such choice of $j$, and the individual summands are monotonically increasing in $j$.  But the value of $d$ must be independent of $j$, so the sum in \eqref{eq:d-from-i-j} must be the same term-by-term for $j=0$ as it is for $j=q-1$.  Thus we have
\[ \left\lfloor\frac{0+pt}{q}\right\rfloor = \left\lfloor\frac{(q-1)+pt}{q}\right\rfloor \quad\text{for }s_0 \leq t \leq s'-1. \]
And this in turn requires that $0+pt$ be a multiple of $q$: otherwise, there will be some $j\in\{1,\dots,q-1\}$ such that $j+pt$ is a multiple of $q$, and then we have
\[ \left\lfloor\frac{0+pt}{q}\right\rfloor \leq  \left\lfloor\frac{j-1+pt}{q}\right\rfloor < \left\lfloor\frac{j+pt}{q}\right\rfloor \leq \left\lfloor\frac{(q-1)+pt}{q}\right\rfloor. \]
In the case $t=-1$ it follows that $-p$ is a multiple of $q$, but since $p$ and $q$ are coprime and $q \geq 2$ this is impossible.

In the remaining case both $q$ and $e = \frac{2g-2}{p}$ are odd, so in particular $p$ must be even.  In this case $\frac{qe}{2}$ is a half-integer, with floor $\frac{qe-1}{2} \geq 1$ since $q > 1$, so we compute from \eqref{eq:s0-2g-2} that
\[ s_0 = \begin{cases}
-\lfloor\frac{qe}{2}\rfloor, & 0 \leq j \leq \frac{p}{2}-1 \\
-\lfloor\frac{qe}{2}\rfloor-1, & \frac{p}{2} \leq j \leq q-1.
\end{cases} \]
(We note that $q-1<p$, so that $0 \leq \frac{j}{p} < 1$ for all such $j$.)  If the second possibility occurs then the $t=s_0$ term in the sum \eqref{eq:d-from-i-j} is
\begin{align*}
g-1 + \left\lfloor\frac{j+ps_0}{q}\right\rfloor
&= g-1 + \left\lfloor\frac{j+p(-\frac{qe+1}{2})}{q}\right\rfloor \\
&= g-1 + \left\lfloor\frac{j - q(g-1) - \frac{p}{2}}{q}\right\rfloor \\
&= \left\lfloor\frac{j - \frac{p}{2}}{q}\right\rfloor = 0,
\end{align*}
so we may as well omit it and begin with $t=s_0+1 = -\lfloor\frac{qe}{2}\rfloor$.  Thus either way \eqref{eq:d-from-i-j} becomes
\[ -d = 2\sum_{t=-\frac{qe-1}{2}}^{-1} \left((g-1) + \left\lfloor \frac{j+pt}{q} \right\rfloor\right) \]
for any of $j=0,1,\dots,q-1$.  Now by exactly the same argument as in the case $\frac{qe}{2}\in\Z$, we set $t=-1$ and let $j$ be either of $0$ and $q-1$, and we conclude that
\[ \left\lfloor\frac{-p}{q}\right\rfloor = \left\lfloor\frac{(q-1)-p}{q}\right\rfloor \]
and then that $-p$ is a multiple of $q$, giving a contradiction.

We have now found a contradiction in all cases where $p \mid 2g-2$ and $q \geq 2$, so we conclude that $q=1$ after all.
\end{proof}

\subsection{Conclusion}

Combining earlier results throughout this section and Section~\ref{sec:5_2}, we have nearly proved the following.

\begin{theorem} \label{thm:5_2-positive}
Let $K \not\cong 5_2$ be a knot of genus $g \geq 2$ in $S^3$, and suppose for some rational $r > 0$ that
\[ S^3_r(K) \cong S^3_r(5_2). \]
Then $r$ is an integer dividing $2g-2$.  Moreover, in these cases $\hfkhat(K)$ is completely determined by the integers $g$ and
\[ d = \begin{cases}
-(g-1)\left(\frac{g-1}{r}-1\right), & r \mid g-1 \\[0.5em]
-\frac{(2g-2-r)^2}{4r}, & r \nmid g-1
\end{cases} \]
as in Proposition~\ref{prop:hfk-nonpositive}.  In particular, $K$ has Alexander polynomial
\[ \Delta_K(t) = t^g - 2t^{g-1} + t^{g-2} + 1 + t^{2-g} - 2t^{1-g} + t^{-g}. \]
\end{theorem}

\begin{proof}
We have shown that
\[ \hfp(S^3_r(K)) \not\cong \hfp(S^3_r(5_2)) \]
in each of the following cases:
\begin{itemize}
\item when $0<r<1$, by Proposition~\ref{prop:gK-large-r-small};
\item when $r=\frac{p}{q} \geq 1$ with $p \nmid 2g-2$, by Proposition~\ref{prop:5_2-medium-fibered};
\item when $r=\frac{p}{q} \geq 1$ is non-integral and $p \mid 2g-2$, by Proposition~\ref{prop:5_2-medium-fibered-2}.
\end{itemize}
This leaves only the cases where $r$ is an integer dividing $2g-2$.

In the remaining cases, we once again write $H_*(A^+_{g-1}) \cong \cT^+_{(0)} \oplus \Q^{\vphantom{+}}_{(d)}$, and then $\hfkhat(K)$ is determined by $g$ and $d$ according to Proposition~\ref{prop:hfk-nonpositive}.  Following the argument and notation from the proof of Proposition~\ref{prop:5_2-medium-fibered-2}, with $(p,q,i,j) = (r,1,g-1,0)$, we set $s'=0$ and
\[ s_0 = -\left\lfloor\frac{e}{2}\right\rfloor = -\left\lfloor\frac{g-1}{p}\right\rfloor \]
as in \eqref{eq:s0-2g-2}.  Then by \eqref{eq:d-from-i-j} we see that $d$ is even, hence Proposition~\ref{prop:hfk-nonpositive} determines the Alexander polynomial of $K$ as promised; and we have
\begin{align*}
-d &= 2\sum_{t=s_0}^{-1} \left((g-1) + \left\lfloor \frac{j+pt}{q} \right\rfloor\right) \\
&= (2g-2)|s_0| + 2\sum_{t'=1}^{|s_0|} r\cdot(-t') \\
&= (2g-2)|s_0| - r|s_0|(|s_0|+1).
\end{align*}
When $p$ divides $g-1$ we have $|s_0| = \frac{g-1}{r}$, and thus
\begin{align*}
-d &= \frac{2(g-1)^2}{r} - (g-1)\left(\frac{g-1}{r}+1\right) \\
&= (g-1) \left( \frac{g-1}{r} - 1 \right).
\end{align*}
Otherwise, since $p$ divides $2g-2$ it follows that $\frac{g-1}{p}$ is a half-integer; then
\[ s_0 = -\left(\frac{g-1}{p}-\frac{1}{2}\right) = - \frac{2g-2-p}{2p} \]
and $p$ is an even integer.  Since $r=p$ we have
\begin{align*}
-d &= \frac{(2g-2)(2g-2-r)}{2r} - \left(\frac{2g-2-r}{2}\right)\left(\frac{2g-2+r}{2r}\right) \\
&= \tfrac{1}{4r}\left( \left(2(2g-2)^2 - 2r(2g-2)\right) - \left((2g-2)^2 - r^2\right) \right) \\
&= \frac{(2g-2 - r)^2}{4r}.
\end{align*}
Thus $d$ is exactly as claimed.
\end{proof}

\begin{remark}
We can collapse the Alexander--Maslov bigrading $(a,m)$ on $\hfkhat(K)$ into a single grading $\delta=m-a$.  If $S^3_r(K) \cong S^3_r(5_2)$ for some $r>0$, then according to Proposition~\ref{prop:hfk-nonpositive}, all of $\hfkhat(K)$ except for a $\Q_{(0)}$ summand in Alexander grading $0$ must be supported in $\delta$-grading $d+2-g$.  Using Theorem~\ref{thm:5_2-positive} (for which we recall the assumption $g\geq 2$), we see that if $r \mid g-1$ then
\[ d \leq 2-2g < g-2, \]
whereas if $r \nmid g-1$ then
\[ d \leq 0 \leq g-2 \]
with equality on the left and on the right if and only if $r=2g-2$ and $g=2$, respectively.  In any case $\hfkhat(K)$ is supported in non-positive $\delta$-gradings, and it is thin if and only if $g(K)=r=2$.
\end{remark}

\section{Quantum obstructions to surgery} \label{sec:quantum}

Ito \cite{ito-lmo} used the LMO invariant of closed 3-manifolds to produce obstructions to cosmetic and other surgeries in terms of finite type invariants.  These include the coefficients $a_{2n}(K)$ of the Conway polynomial
\[ \nabla_K(z) = a_0(K) + a_2(K)z^2 + a_4(K)z^4 + \dots, \]
as well as an invariant $v_3(K) \in \frac{1}{4}\Z$ which is determined by the Jones polynomial of $K$.  In particular, he proved the following, which we will apply to improve Theorem~\ref{thm:5_2-positive}.

\begin{theorem}[{\cite[Corollary~1.3(iv)]{ito-lmo}}] \label{thm:ito-obstruction}
Suppose for some knots $K,K' \subset S^3$ and rational $r\neq 0$ that $S^3_r(K) \cong S^3_r(K')$.  Then either
\begin{enumerate}
\item $a_4(K) = a_4(K')$ and $v_3(K) = v_3(K')$, or
\item $a_4(K) \neq a_4(K')$ and $v_3(K) \neq v_3(K')$, in which case
\begin{equation} \label{eq:ito-lmo-r}
r = -\frac{5(a_4(K)-a_4(K'))}{4(v_3(K)-v_3(K'))}.
\end{equation}
\end{enumerate}
\end{theorem}

\begin{remark}
The sign in front of the right side of \eqref{eq:ito-lmo-r} was omitted in \cite{ito-lmo}.  In fact, \cite[Theorem~1.2]{ito-lmo} gives a surgery formula for the degree-2 part $\lambda_2(S^3_r(K))$ of the LMO invariant, in which one of the terms is $-\frac{5a_4(K)}{4} \cdot \frac{1}{r^2}$.  In the proof of \cite[Corollary~1.3(iv)]{ito-lmo} this term appears without the minus sign, which accounts for the discrepancy.
\end{remark}

In order to apply Theorem~\ref{thm:ito-obstruction} to a potential surgery $S^3_r(K) \cong S^3_r(5_2)$, we first recall that the Conway polynomial can be recovered from the Alexander polynomial by the relation
\[ \Delta_K(t^2) = \nabla_K(t-t^{-1}). \]
In particular, we have
\[ \nabla_{5_2}(t-t^{-1}) = 2t^2 - 3 + 2t^{-2} = 1 + 2(t-t^{-1})^2, \]
so $\nabla_{5_2}(z) = 1+2z^2$ and thus $a_4(5_2) = 0$.  The computation of $a_4(K)$ is more involved.

\begin{lemma} \label{lem:compute-a4}
Suppose for some knot $K \not\cong 5_2$ and $r\in\Q$ that $S^3_r(K) \cong S^3_r(5_2)$.  If $g(K) \geq 2$, then $a_2(K) = 2$ and $a_4(K) = (g(K)-1)^2$.
\end{lemma}

\begin{proof}
Theorem~\ref{thm:5_2-positive} tells us that $r$ is a positive integer dividing $2g(K)-2$, and that if we write
\[ f_g(t) = t^g - 2t^{g-1} + t^{g-2} + 1 + t^{2-g} - 2t^{1-g} + t^{-g} \]
for all integers $g \geq 2$, then $\Delta_K(t) = f_{g(K)}(t)$.  These polynomials satisfy the relation
\[ (f_g(t) - 1)(t+t^{-1}) = (f_{g+1}(t)-1) + (f_{g-1}(t)-1) \]
for all $g \geq 3$, and if we write $t=s^2$ then this becomes
\begin{equation} \label{eq:conway-relation}
(f_g(s^2)-1)\big( (s-s^{-1})^2 + 2 \big) = f_{g+1}(s^2) + f_{g-1}(s^2) - 2.
\end{equation}

Define polynomials $p_g(z)$ for all $g \geq 2$ such that
\[ p_g(s-s^{-1}) = f_g(s^2). \]
We can check that
\begin{align*}
p_2(z) &= 1 + 2z^2 + z^4 \\
p_3(z) &= 1 + 2z^2 + 4z^4 + z^6,
\end{align*}
and then \eqref{eq:conway-relation} becomes
\[ \big(p_g(s-s^{-1})-1\big) \big( (s-s^{-1})^2 + 2 \big) = p_{g+1}(s+s^{-1}) + p_{g-1}(s+s^{-1}) - 2. \]
Substituting $z = s-s^{-1}$, we have
\begin{equation} \label{eq:conway-recurrence}
p_{g+1}(z) = (z^2+2)(p_g(z)-1) - p_{g-1}(z) + 2
\end{equation}
for all $g \geq 3$, and moreover $p_{g(K)}(z)$ is the Conway polynomial $\nabla_K(z)$.

We now claim by induction that
\[ p_g(K) = 1 + 2z^2 + (g-1)^2z^4 + O(z^6) \]
for all $g \geq 2$.  It is certainly true for $g=2$ and $g=3$, and then for $g\geq 3$ we examine \eqref{eq:conway-recurrence} modulo $z^6$ to get
\begin{align*}
p_{g+1}(z) &\equiv (z^2+2)\left(2z^2+(g-1)^2z^4\right) - \left(1 + 2z^2 + (g-2)^2z^4\right) + 2 \\
&\equiv \left((2g^2-4g+4) z^4 + 4z^2\right) - \left((g^2-4g+4)z^4 + 2z^2\right) + 1 \\
&\equiv g^2 z^4 + 2z^2 + 1 \pmod{z^6}
\end{align*}
exactly as claimed.  But this means that the coefficients $a_2(K)$ and $a_4(K)$ of $z^2$ and $z^4$ in $\nabla_K(z) = p_{g(K)}(z)$ are $2$ and $(g(K)-1)^2$ respectively, proving the lemma.
\end{proof}

We can evaluate $v_3(K)$ in terms of the Jones polynomial $V_K(q)$ as follows.

\begin{lemma} \label{lem:4v3-jones}
We have $4v_3(K) = -\frac{1}{36}( V'''_K(1) + 3 V''_K(1) )$.
\end{lemma}

\begin{proof}
We note from \cite[Lemma~2.1]{ito-lmo} that if we evaluate the Jones polynomial
\[ V_K(q) = \sum_{i\in\Z} c_i q^i \]
at $q=e^h$ and write the corresponding power series as
\[ \sum_{n=0}^\infty j_n(K) h^n = V_K(e^h) = \sum_{i\in\Z} c_i \left(\sum_{n=0}^\infty \frac{(ih)^n}{n!} \right), \]
then $v_3(K) = -\frac{1}{24} j_3(K)$.  Comparing $h^3$-coefficients gives us
\[ 4v_3(K) = -\frac{1}{6} j_3(K) = -\frac{1}{36} \sum_{i\in\Z} c_i \cdot i^3. \]
At the same time,  we have
\[ V_K'''(1) + 3V_K''(1) + V_K'(1) = \sum_{i\in\Z} c_i \cdot \left( (i^3-3i^2+2i) + 3(i^2-i) + i \right) = \sum_{i\in\Z} c_i \cdot i^3, \]
and we know that $V_K'(1) = 0$ \cite[\S12]{jones-polynomial}, so the lemma follows.
\end{proof}

\begin{example} \label{ex:4v3-52}
We know that
\[ V_{5_2}(q) = q^{-1} - q^{-2} + 2q^{-3} - q^{-4} + q^{-5} - q^{-6} \]
and since $V'''_{5_2}(1) = 144$ and $V''_{5_2}(1) = -12$, we get $4v_3(5_2) = -3$.
\end{example}

We can use this obstruction to prove that non-characterizing slopes for $5_2$ cannot arise from other knots of genus 1.

\begin{proposition} \label{prop:5_2-positive-genus-1}
Suppose for some knot $K$ of genus 1 and some $r\in\Q$ that $S^3_r(K) \cong S^3_r(5_2)$.  Then $K$ is isotopic to $5_2$.
\end{proposition}

\begin{proof}
If $K\not\cong 5_2$ then Proposition~\ref{prop:5_2-fibered} says that $K$ is either $15n_{43522}$ or $\Wh^-(T_{2,3},2)$, up to mirroring.  But in these cases we have
\[ a_4(K) = a_4(5_2) = 0, \]
since $\Delta_K(t) = \Delta_{5_2}(t) = 2t-3+2t^{-1}$, and yet we can compute from Lemma~\ref{lem:4v3-jones} that
\[ 4v_3(K) = \pm 7 \text{ or } \pm1 \]
respectively, while $4v_3(5_2) = -3$.  Thus Theorem~\ref{thm:ito-obstruction} says that $S^3_r(K) \not\cong S^3_r(5_2)$ after all.
\end{proof}

We can now use Lemmas~\ref{lem:compute-a4} and \ref{lem:4v3-jones} to identify potentially non-characterizing slopes.

\begin{proposition} \label{prop:quantum-obstruction-g-even}
Suppose that $S^3_r(K) \cong S^3_r(5_2)$ for some integer $r \geq 1$, and that $K$ is not isotopic to $5_2$.  Then the Jones polynomial $V_K(q)$ satisfies $\frac{1}{36} V_K'''(1) \in \Z$, and we have
\[ r = \frac{5(g(K)-1)^2}{\frac{1}{36}V_K'''(1)-4}. \]
Moreover, if $g(K)$ is even then $r$ divides $g(K)-1$.
\end{proposition}

\begin{proof}
Write $g=g(K)$.  We know that $g \geq 2$ by Proposition~\ref{prop:5_2-positive-genus-1}, hence Lemma~\ref{lem:compute-a4} says that $a_4(K) = (g-1)^2$, which is different from $a_4(5_2) = 0$.  We thus apply Theorem~\ref{thm:ito-obstruction} to see that
\[ r = -\frac{5(a_4(K)-a_4(5_2))}{4(v_3(K)-v_3(5_2))} = -\frac{5(g-1)^2}{4v_3(K) + 3}. \]
Proposition~\ref{prop:compare-gradings} tells us that $\Delta''_K(1) = \Delta''_{5_2}(1) = 4$, so $V''_K(1) = -3\Delta''_K(1) = -12$, again by \cite[\S12]{jones-polynomial}.  Thus
\[ 4v_3(K) - 4v_3(5_2) = -\frac{1}{36}(V_K'''(1) - 36) + 3 = 4 - \frac{V_K'''(1)}{36}, \]
which must be an integer since $4v_3(K)$ is, and this completes the determination of $r$.

Now supposing that $g$ is even, we have expressed $r$ as a divisor of the odd integer $5(g-1)^2$.  Thus $r$ is odd, and it divides $2g-2$ by Theorem~\ref{thm:5_2-positive}, so it must in fact divide $g-1$ as claimed.
\end{proof}

This last result allows us to complete the proof of Theorem~\ref{thm:main-5_2-positive}.

\begin{proof}[Proof of Theorem~\ref{thm:main-5_2-positive}]
If $S^3_r(K) \cong S^3_r(5_2)$ but $K \not\cong 5_2$, then Proposition~\ref{prop:5_2-positive-genus-1} says that $K$ has genus $g \geq 2$.  In this case Theorem~\ref{thm:5_2-positive} says that $r$ is a positive integer dividing $2g-2$, and that $\hfkhat(K)$ has the claimed form.  The only remaining claim is that if $g$ is even then $r$ divides $g-1$, and this is part of Proposition~\ref{prop:quantum-obstruction-g-even}.
\end{proof}

\begin{remark} \label{rem:pretzel-1-surgery}
As a final example of the effectiveness of Proposition~\ref{prop:quantum-obstruction-g-even}, let us suppose that $S^3_r(5_2) \cong S^3_r(P(-3,3,2n))$ for some integers $r \geq 1$ and $n$.  Since $P(-3,3,2n)$ has genus $2$, Proposition~\ref{prop:quantum-obstruction-g-even} says that $r=1$.  Moreover, an exercise with the skein relation for the Jones polynomial shows that
\begin{align*}
V_{P(-3,3,2n)}(q) &= q^{-2n}V_{P(-3,3,0)}(q) + (1-q^{-2n}) \\
&= -q^{-2n-3} + q^{-2n-2} - q^{-2n-1} + 2q^{-2n} - q^{-2n+1} + q^{-2n+2} - q^{-2n+3} + 1.
\end{align*}
(We note that $P(-3,3,0) \cong T_{2,3}\#T_{-2,3}$.)  From this one can show that
\[ \tfrac{1}{36}V_{P(-3,3,2n)}'''(1) - 4 = 2n-3, \]
so $r=1=\frac{5}{2n-3}$ implies that $2n=8$.
\end{remark}

In Section~\ref{sec:non-characterizing} we will see that $S^3_1(5_2)$ is in fact homeomorphic to $S^3_1(P(-3,3,8))$.

\section{Non-characterizing slopes for $5_2$} \label{sec:non-characterizing}

In this section we prove that $1$ is not a characterizing slope for $5_2$.

\begin{proposition} \label{prop:bdc-surgery-5_2}
For any integer $n\in\Z$, the 3-manifold $S^3_n(5_2)$ is the branched double cover of the link $L_n$ shown in Figure~\ref{fig:bdc-5_2}.
\end{proposition}

\begin{proof}
The knot $5_2$ is strongly invertible, meaning that there is an involution $\tau:S^3\to S^3$ such that $\tau(5_2)=5_2$, and the fixed set of $\tau$ is an unknot $U$ meeting $5_2$ in two points.  In the quotient $S^3/\tau \cong S^3$, we remove a neighborhood of $5_2/\tau$; this turns $U/\tau$ into a tangle with four endpoints, whose branched double cover is $S^3\setminus N(5_2)$, and we can fill in this tangle by gluing in a rational tangle to get a link $L_r$ whose branched double cover is any Dehn surgery $S^3_r(5_2)$.
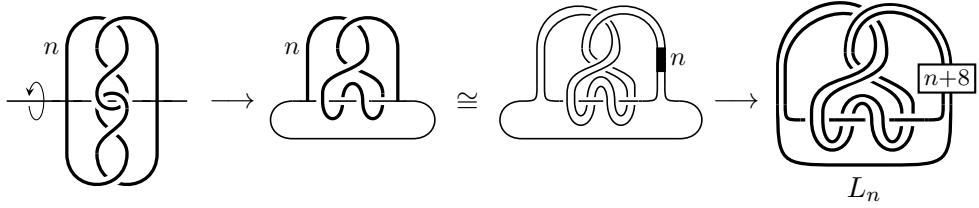
\begin{figure}
\tikzset{twistregion/.style={draw, fill=white, thick, minimum width=0.6cm}}
\tikzset{thinlink/.style = { white, double = black, line width = 1.2pt, double distance = 0.6pt, looseness=1.75 }}
\newcommand{\AxisRotator}{\tikz [x=0.10cm,y=0.25cm,line width=.2ex,-stealth] \draw[thin] (0,0) arc (-165:165:1 and 1);}
\begin{tikzpicture}

\begin{scope}
\draw[semithick] (-0.8,0) -- (1.6,0);
\draw[link] (0.8,0.1) to[out=90,in=270,looseness=1] ++(-0.4,0.6);
\draw[link] (0,0) -- ++(0,0.7) node[left,black,inner sep=2pt] {\small$n$} to[out=90,in=90] ++(0.8,0) to[out=270,in=90,looseness=1] ++(-0.4,-0.6) to[out=270,in=270] ++(0.4,0);
\draw[link] (0.4,0.7) to[out=90,in=90] ++(0.8,0) -- ++(0,-0.7);
\draw[link] (0.4,-0.7) to[out=270,in=270] ++(0.8,0) -- ++(0,0.7);
\draw[link] (0,0) -- ++(0,-0.7) to[out=270,in=270] ++(0.8,0) to[out=90,in=270,looseness=1] ++(-0.4,0.6) to[out=90,in=90] ++(0.4,0);
\draw[link] (0.8,-0.1) to[out=270,in=90,looseness=1] ++(-0.4,-0.6);
\begin{scope}
\clip (0.2,-0.2) rectangle (0.6,0.2);
\draw[link] (0.8,0.7) to[out=270,in=90,looseness=1] ++(-0.4,-0.6) to[out=270,in=270] ++(0.4,0);
\end{scope}
\draw[semithick] (-0.2,0) -- (0.2,0) (1.0,0) -- (1.4,0);
\node at (-0.4,0) {\AxisRotator};
\end{scope}

\node at (2.2,-0.02) {$\longrightarrow$};

\begin{scope}[xshift=3.2cm]
\draw[semithick] (-0.2,0) -- (1.4,0);
\draw[link] (1,0.1) to[out=90,in=270,looseness=1] ++(-0.6,0.6);
\draw[link] (0,0) -- ++(0,0.7) node[left,black,inner sep=2pt] {\small$n$} to[out=90,in=90] ++(0.8,0) to[out=270,in=90,looseness=1] ++(-0.6,-0.6) -- ++(0,-0.1) to[out=270,in=270,looseness=3] ++(0.2667,0) to[out=90,in=90,looseness=3] ++(0.2666,0) to[out=270,in=270,looseness=3] ++(0.2667,0) -- ++(0,0.1);
\draw[link] (0.4,0.7) to[out=90,in=90] ++(0.8,0) -- ++(0,-0.7);
\draw[semithick] (-0.1,0) -- (0.1,0) (1.1,0) -- (1.3,0);
\draw[thinlink] (0.35,0) -- ++(0.2,0) ++(0.3,0) -- ++(0.3,0);
\draw[semithick,looseness=2] (-0.2,0) to[out=180,in=180] ++(0,-0.5) -- ++(1.6,0) to[out=0,in=0] ++(0,0.5);
\end{scope}

\node at (5.325,-0.02) {$\cong$};

\begin{scope}[xshift=6.5cm]
\draw[semithick] (0.05,0) -- (1.15,0);

\draw[thinlink] (0.95,0) -- ++(0,0.1) to[out=90,in=270,looseness=1] ++(-0.6,0.6);
\draw[thinlink] (1.05,0) -- ++(0,0.2) to[out=90,in=270,looseness=0.75] ++(-0.6,0.5);

\draw[thinlink] (-0.25,0.2) -- ++(0,0.5) to[out=90,in=90] ++(1.1,0) to[out=270,in=90,looseness=1] ++(-0.6,-0.6) -- ++(0,-0.1) to[out=270,in=270,looseness=4] ++(0.1667,0) to[out=90,in=90,looseness=3] ++(0.3667,0) to[out=270,in=270,looseness=4] ++(0.1667,0);
\draw[thinlink] (-0.15,0.2) -- ++(0,0.5) to[out=90,in=90,looseness=1.5] ++(0.9,0.1) to[out=270,in=90,looseness=1.25] ++(-0.6,-0.7)  -- ++(0,-0.1) to[out=270,in=270,looseness=3] ++(0.3667,0) to[out=90,in=90,looseness=4] ++(0.1667,0) to[out=270,in=270,looseness=3] ++(0.3667,0);

\draw[thinlink] (0.35,0.7) to[out=90,in=90] ++(1.1,0) -- ++(0,-0.5);
\draw[thinlink] (0.45,0.7) to[out=90,in=90] ++(0.9,0) -- ++(0,-0.5);
\draw[fill=black] (1.35,0.7) rectangle (1.45,0.4) node[right,midway,inner sep=3pt] {\small$n$};

\draw[thinlink] (0.35,0) -- ++(0.25,0) ++(0.25,0) -- ++(0.25,0);

\draw[thinlink] (-0.25,0.2) to[out=270,in=0] ++(-0.2,-0.2);
\draw[thinlink] (-0.15,0.2) to[out=270,in=180] ++(0.2,-0.2);
\draw[thinlink] (1.35,0.2) to[out=270,in=0] ++(-0.2,-0.2);
\draw[thinlink] (1.45,0.2) to[out=270,in=180] ++(0.2,-0.2);
\draw[thinlink,looseness=2] (-0.45,0) to[out=180,in=180] ++(0,-0.5) -- ++(2.1,0) to[out=0,in=0] ++(0,0.5);
\end{scope}

\node at (8.9,-0.02) {$\longrightarrow$};

\begin{scope}[xshift=10cm,yshift=-0.25cm]
\draw[link] (-0.25,0) -- (1.45,0);

\draw[link] (1.10,0) -- ++(0,0.1) to[out=90,in=270,looseness=1] ++(-0.75,0.8);
\draw[link] (1.25,0) -- ++(0,0.2) to[out=90,in=270,looseness=0.75] ++(-0.8,0.7);

\draw[link] (-0.55,0.2) -- ++(0,0.5) to[out=90,in=90] ++(1.45,0.2) to[out=270,in=90,looseness=1] ++(-0.85,-0.8) -- ++(0,-0.1) to[out=270,in=270,looseness=4] ++(0.25,0) to[out=90,in=90,looseness=2] ++(0.55,0) to[out=270,in=270,looseness=4] ++(0.25,0);
\draw[link] (-0.45,0.2) -- ++(0,0.5) to[out=90,in=90,looseness=1.5] ++(1.2,0.3) to[out=270,in=90,looseness=1.25] ++(-0.85,-0.9)  -- ++(0,-0.1) to[out=270,in=270,looseness=2.5] ++(0.55,0) to[out=90,in=90,looseness=3] ++(0.25,0) to[out=270,in=270,looseness=2.5] ++(0.55,0);

\draw[link] (0.35,0.9) to[out=90,in=90] ++(1.4,-0.2) -- ++(0,-0.5);
\draw[link] (0.45,0.9) to[out=90,in=90] ++(1.2,-0.2) -- ++(0,-0.5);
\node[twistregion,inner sep=2pt] at (1.7,0.55) {\footnotesize$n{+}8$};
\draw[link] (0.2,0) -- ++(0.4,0) ++(0.4,0) -- ++(0.4,0);

\draw[link] (1.75,0.2) to[out=270,in=0] ++(-0.8,-0.8) -- node[midway,below,black] {$L_n$} ++(-0.7,0) to[out=180,in=270] ++(-0.8,0.8);
\draw[link] (-0.45,0.2) to[out=270,in=180] ++(0.2,-0.2);
\draw[link] (1.65,0.2) to[out=270,in=0] ++(-0.2,-0.2);
\end{scope}

\end{tikzpicture}
\caption{A link $L_n$ whose branched double cover is $S^3_n(5_2)$.  We quotient $5_2$ by a rotation $\tau$ around the indicated axis of symmetry, simplify the resulting diagram by an isotopy, and then replace the arc $5_2/\tau$ by a rational tangle.  The box labeled ``$n+8$'' corresponds to $n+8$ signed half-twists.}
\label{fig:bdc-5_2}
\end{figure}

This process is illustrated in Figure~\ref{fig:bdc-5_2}.  In order to determine that the box with $n+8$ twists actually corresponds to $S^3_n(5_2)$, we observe that replacing it with the rational tangle
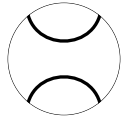
\[ \begin{tikzpicture}
\begin{scope}
\clip (0,0) circle (0.75);
\draw[link] (-0.5,0.75) to[out=270,in=270] ++(1,0);
\draw[link] (-0.5,-0.75) to[out=90,in=90] ++(1,0);
\draw[thin] (0,0) circle (0.75);
\end{scope}
\end{tikzpicture} \]
turns $L_n$ into an unknot, whose branched double cover $S^3$ is the result of $\frac{1}{0}$-surgery on $5_2$.  Then each possible number of half-twists corresponds to a surgery with slope at distance $1$ from $\frac{1}{0}$, so these are exactly the integral slopes.  We can finally compute that $\det(L_n) = |n|$, so that $\dcover(L_n)$ is identified with $S^3_n(5_2)$ as claimed.
\end{proof}

\begin{remark}
Another construction of links with branched double cover $S^3_n(5_2)$ was given in \cite[Lemma~8.3]{bs-concordance}, where the argument was specialized to $n=-3$ but works for arbitrary integers.  That construction uses a different involution, and hence produces different links (illustrated in \cite[Figure~12]{bs-concordance}) in general.
\end{remark}

\begin{proposition} \label{prop:slope-1}
There is an orientation-preserving homeomorphism
\[ S^3_1(5_2) \cong S^3_1(P(-3,3,8)). \]
\end{proposition}

\begin{proof}
Let $P=P(-3,3,8)$ for convenience.  Then $P$ is strongly invertible, and we can adapt the proof of \cite[Proposition~7.6]{bs-concordance}, which was originally due to Ken Baker, to realize $S^3_1(P)$ as the branched double cover of a knot $K \subset S^3$, as shown in Figure~\ref{fig:P-8-dcover}.
\begin{figure}
\tikzset{twistregion/.style={draw, fill=white, thick, minimum width=0.6cm}}
\tikzset{thinlink/.style = { white, double = black, line width = 1.2pt, double distance = 0.6pt, looseness=1.75 }}
\newcommand{\AxisRotator}{\tikz [x=0.10cm,y=0.25cm,line width=.2ex,-stealth] \draw[thin] (0,0) arc (-165:165:1 and 1);}
\begin{tikzpicture}
\begin{scope} 
\draw[semithick] (-1.6,0.75) -- (1.6,0.75); 
\draw[link,looseness=1] (0,1.5) to[out=270,in=90] ++(0.4,-0.5) ++(-0.4,0) to[out=270,in=90] ++(0.4,-0.5) ++(-0.4,0) to[out=270,in=90] ++(0.4,-0.5) to[out=270,in=270,looseness=1.75] ++(0.4,0) to[out=90,in=270] ++(0.4,0.5) ++(-0.4,0) to[out=90,in=270] ++(0.4,0.5) ++(-0.4,0) to[out=90,in=270] ++(0.4,0.5);
\draw[link,looseness=1] (1.2,0) to[out=90,in=270] ++(-0.4,0.5) ++(0.4,0) to[out=90,in=270] ++(-0.4,0.5) ++(0.4,0) to[out=90,in=270] ++(-0.4,0.5) to[out=90,in=90,looseness=1.75] ++(-0.4,0) to[out=270,in=90] ++(-0.4,-0.5) ++(0.4,0) to[out=270,in=90] ++(-0.4,-0.5) ++(0.4,0) to[out=270,in=90] ++(-0.4,-0.5);
\draw[link] (0,1.5) to[out=90,in=90] ++(-0.4,0) to[out=270,in=90] ++(0,-1.5) to[out=270,in=270] ++(0.4,0);
\draw[link,looseness=1] (1.2,1.5) to[out=90,in=0] ++(-0.4,0.5) to[out=180,in=0] node[midway,above,black] {$1$} ++(-1.2,0) to[out=180,in=90] ++(-0.4,-0.5) to[out=270,in=90] ++(0,-1.5) to[out=270,in=180] ++(0.4,-0.5) to[out=0,in=180] ++(1.2,0) to[out=0,in=270] ++(0.4,0.5);
\node[twistregion] at (-0.6,1.25) {$4$};
\node[twistregion] at (-0.6,0.25) {$4$};
\draw[semithick] (-1.2,0.75) -- ++(1,0); 
\node at (-1.2,0.75) {\AxisRotator};
\end{scope}

\node at (2.2,0.725) {$\longrightarrow$};

\begin{scope}[xshift=3cm]
\draw[thinlink] (-0.1,0) -- ++(0,0.5) to[out=90,in=180,looseness=1.25] ++(0,0.75) -- ++(0.2,0) to[out=0,in=90,looseness=1.25] ++(0,-0.75) -- ++(0,-0.5);
\draw[thinlink] (0.1,0) to[out=270,in=270] (0.8,0) -- ++(0,1.75) to[out=90,in=90] ++(0.8,0);
\draw[thinlink] (-0.1,0) to[out=270,in=270] (1.0,0) -- ++(0,1.75) to[out=90,in=90] ++(0.4,0);
\draw[thinlink,looseness=1.25] (1.6,1.75) -- ++(0,-0.25) to[out=270,in=0] ++(-0.8,-1) to[out=180,in=180,looseness=4] ++(0,0.2) to[out=0,in=270] ++(0.6,0.8) -- ++(0,0.25);
\draw[thinlink] (0.8,0.6) -- ++(0,0.2) ++(0.2,-0.2) -- ++(0,0.2); 
\draw[link,looseness=1.25] (-0.1,1.25) to[out=90,in=180] ++(1,0.5) to[out=0,in=0,looseness=5] ++(0,-0.2) to[out=180,in=90] ++(-0.8,-0.3);
\draw[thinlink] (0.8,1.65) -- ++(0,0.1) to[out=90,in=90] ++(0.8,0); 
\draw[thinlink] (1.0,1.65) -- ++(0,0.1) to[out=90,in=90] ++(0.4,0);
\draw[semithick] (-0.1,0.5) to[out=90,in=180,looseness=1.25] ++(0,0.75) -- ++(0.2,0) to[out=0,in=90,looseness=1.25] ++(0,-0.75);
\node[twistregion,semithick] at (0,0.25) {$4$};
\node at (0,1.95) {\small$1$};
\end{scope}

\node at (5.1,0.725) {$\longrightarrow$};

\begin{scope}[xshift=6cm]
\draw[link,looseness=1.5] (0.3,1.35) to[out=180,in=180] ++(0,0.2) -- ++(0.9,0) to[out=0,in=0] ++(0,-0.2) -- ++(-0.5,0) to[out=180,in=0,looseness=1] ++(-0.4,-0.2);
\draw[link] (-0.1,0) -- ++(0,0.8) to[out=90,in=180,looseness=1.25] ++(0.7,0.95) -- ++(0.5,0) to[out=0,in=0,looseness=2.5] ++(0,-0.6) -- ++(-0.4,0) to[out=180,in=0,looseness=1] ++(-0.4,0.2) ++(0,-0.2) to[out=180,in=90,looseness=1] ++(-0.2,-0.2) -- (0.1,0);
\draw[link] (0.1,0) to[out=270,in=270] (0.8,0) -- ++(0,1.75) to[out=90,in=90] ++(1.2,0);
\draw[link] (-0.1,0) to[out=270,in=270] (1.0,0) -- ++(0,1.75) to[out=90,in=90] ++(0.8,0);
\draw[link,looseness=1.25] (2,1.75) -- ++(0,-0.25) to[out=270,in=0] ++(-1.2,-1) to[out=180,in=180,looseness=4] ++(0,0.2) to[out=0,in=270] ++(1,0.8) -- ++(0,0.25);
\draw[link] (0.8,0.6) -- ++(0,0.2) ++(0.2,-0.2) -- ++(0,0.2); 
\draw[link] (0.3,1.15) ++(0.2,0.2) ++(0.2,0) -- ++(0.4,0) ++(0,-0.2) -- ++(-0.4,0);
\node[twistregion] at (0,0.25) {$4$};
\end{scope}

\node at (8.35,0.75) {$\cong$};

\begin{scope}[xshift=9cm]
\draw[link] (-0.1,0) ++(0,0.8) to[out=90,in=180,looseness=1.25] ++(0.7,0.95) -- ++(0.5,0) to[out=0,in=0,looseness=2.5] ++(0,-0.2) -- ++(-0.5,0) to[out=180,in=180] ++(0,-0.2) -- ++(0.5,0) to[out=0,in=0] ++(0,-0.2) -- ++(-0.5,0) to[out=180,in=90,looseness=1] ++(-0.5,-0.35);
\draw[link] (0.1,0) \foreach \i in {1,...,4} { to[out=90,in=270,looseness=1] ++(-0.2,0.2) ++ (0.2,0) };
\draw[link] (-0.1,0) \foreach \i in {1,...,4} { to[out=90,in=270,looseness=1] ++(0.2,0.2) ++ (-0.2,0) };
\draw[link] (0.1,0) to[out=270,in=270] (0.8,0) -- ++(0,1.75) to[out=90,in=90] ++(1.2,0);
\draw[link] (-0.1,0) to[out=270,in=270] (1.0,0) -- ++(0,1.75) to[out=90,in=90] ++(0.8,0);
\draw[link,looseness=1.25] (2,1.75) -- ++(0,-0.25) to[out=270,in=0] ++(-1.2,-1) to[out=180,in=180,looseness=4] ++(0,0.2) to[out=0,in=270] ++(1,0.8) -- ++(0,0.25);
\draw[link] (0.8,0.6) -- ++(0,0.2) ++(0.2,-0.2) -- ++(0,0.2); 
\draw[link] (0.3,1.15) ++(0.2,0.2) ++(0.2,0.2) -- ++(0.4,0) ++(0,-0.4) -- ++(-0.4,0);
\node at (1.5,0.05) {$K$};
\end{scope}

\end{tikzpicture}
\caption{Identifying $S^3_1(P(-3,3,8))$ as a branched double cover $\dcover(K)$.  We quotient $P(-3,3,8)$ by a rotation $\tau$ around an axis of symmetry and simplify by an isotopy, following \cite[Figure~7]{bs-concordance}.  We then replace a neighborhood of the arc $P(-3,3,8)/\tau$ with a rational tangle, and isotope further to get the desired knot $K$.}
\label{fig:P-8-dcover}
\end{figure}
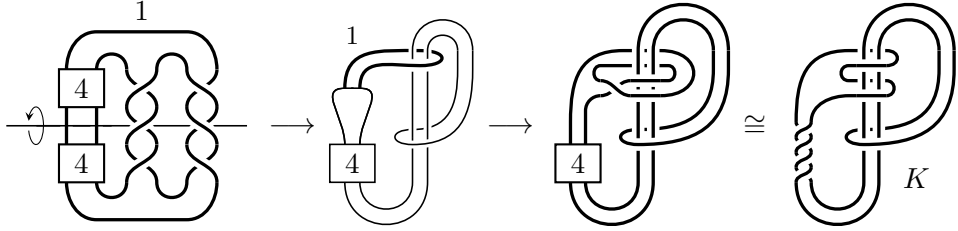

We now claim that $K$ is isotopic to the knot $L_1$ from Figure~\ref{fig:bdc-5_2}, and so
\[ S^3_1(P) \cong \dcover(K) \cong \dcover(L_1) \cong S^3_1(5_2) \]
by Proposition~\ref{prop:bdc-surgery-5_2}.  Rather than find this isotopy explicitly, we observe that SnapPy recognizes each of $K$ and $L_1$ as either $14n_{14254}$ or its mirror, so that
\[ S^3_1(P) \cong \dcover(K) \quad\text{and}\quad S^3_1(5_2) \cong \dcover(L_1) \]
are homeomorphic up to orientation.  But we cannot have $S^3_1(5_2) \cong -S^3_1(P)$, since their Casson invariants satisfy
\begin{align*}
\lambda(S^3_1(5_2)) &= \tfrac{1}{2}\Delta_{5_2}''(1) = 2, \\
\lambda(-S^3_1(P)) &= \lambda(S^3_{-1}(\mirror{P})) = -\tfrac{1}{2}\Delta_{\mirror{P}}''(1) = -2.
\end{align*}
(This computation follows from $\Delta_{\mirror{P}}(t) = t^2-2t+3-2t^{-1}+t^{-2}$.)  Thus $S^3_1(5_2) \cong S^3_1(P)$ as oriented 3-manifolds.
\end{proof}

\section{The $\Sigma(2,3,11)$ realization problem} \label{sec:2-3-11}

Let $Y = -\Sigma(2,3,11)$.  Then we have orientation-preserving homeomorphisms
\[ Y \cong S^3_{1/2}(T_{2,3}) \cong S^3_1(\mirror{5_2}). \]
(Up to an overall orientation reversal, the latter identification is the case $S^3_{-1}(K(2,4)) \cong S^3_{-1/2}(K(2,2))$ of \cite[Proposition~7.2]{bs-concordance}, for example.)  Our goal in this section is to prove that these are the only ways to express $Y$ as Dehn surgery on a knot in $S^3$.

\begin{theorem} \label{thm:2-3-11}
Suppose for some knot $K\subset S^3$ and some rational $r\in\Q$ that
\[ S^3_r(K) \cong -\Sigma(2,3,11). \]
Then $(K,r)$ is either $(T_{2,3}, \frac{1}{2})$ or $(\mirror{5_2},1)$.
\end{theorem}

This is equivalent to Theorem~\ref{thm:main-2-3-11}, as can be seen by the identity $S^3_r(K) \cong -S^3_{-r}(\mirror{K})$.

\begin{proof}[Proof of Theorem~\ref{thm:2-3-11}]
Since $Y$ is a homology sphere, we can write $r=\frac{1}{n}$ for some nonzero $n\in\Z$.  If $n=1$ and hence $r=1$, we have
\[ \hfp(S^3_1(K)) \cong \hfp(Y) \cong \hfp(S^3_1(\mirror{5_2})). \]
We then apply Theorem~\ref{thm:main-m5_2-nonnegative} to conclude that $K \cong \mirror{5_2}$.  Similarly, if $r=\frac{1}{2}$ then we must have $K \cong T_{2,3}$, since all slopes are characterizing slopes for the right-handed trefoil \cite{osz-characterizing}.

Supposing from now on that $n$ is neither $1$ nor $2$, we first claim that $n \geq 3$.  Indeed, we know that
\[ d(S^3_{1/(-n)}(\mirror{K})) = d(-S^3_{1/n}(K)) = d(-Y) = 2, \]
where we have read $d(Y) = d(S^3_1(\mirror{5_2})) = -2$ off of Proposition~\ref{prop:1-surgery-5_2}.  But if $n<0$, or equivalently $-n > 0$, then Theorem~\ref{thm:ni-wu-d} says that
\[ d(S^3_{1/(-n)}(\mirror{K})) \leq d(S^3_{1/(-n)}(U)) = d(S^3) = 0. \]
This would be a contradiction, so we must have $n>0$ and hence $n\geq 3$ as claimed.

Now that we have $n\geq 3$, we compute that $\dim \hfhat(Y) = \dim \hfhat(S^3_1(\mirror{5_2})) = 3$ from Proposition~\ref{prop:1-surgery-5_2} and Lemma~\ref{lem:hfhat-vs-hfred}.  Thus
\begin{align*}
3 = \dim \hfhat(S^3_{1/n}(K)) &= n\cdot\rhat(K) + |1-n\nuhat(K)| \\
& \geq 3\cdot \rhat(K) + 1,
\end{align*}
since $\rhat(K) \geq |\nuhat(K)| \geq 0$
and since $1-n\nuhat(K) \equiv 1 \pmod{n}$ is nonzero.  This is only possible if $\rhat(K) = 0$, in which case $\nuhat(K) = 0$ as well and then $\dim \hfhat(S^3_{1/n}(K))$ must be $1$ rather than $3$, so we have a contradiction.  This completes the proof.
\end{proof}

\bibliographystyle{alpha}
\bibliography{References}

\end{document}